\documentclass[letterpaper, 11pt]{amsart}
\usepackage{mathtools}
\usepackage{amsmath}
\usepackage{amssymb}
\usepackage{yhmath}
\usepackage{graphicx}
\usepackage{ mathrsfs }
\usepackage{bbm}
\usepackage{xcolor}
\usepackage{tikz-cd}
\usepackage{tikz}
\usetikzlibrary{patterns}
\usepackage{hyperref}

\DeclareMathAlphabet{\mathpzc}{OT1}{pzc}{m}{it}

\usepackage{thmtools}
\usepackage{thm-restate}

\usepackage{caption}

\newtheorem{theorem}{Theorem}[section]

\newtheorem*{claim*}{Claim}

\newtheorem{lemma}[theorem]{Lemma}
\newtheorem{lem}[theorem]{Lemma}
\newtheorem{corollary}[theorem]{Corollary}

\newtheorem{cor}[theorem]{Corollary}

\newtheorem{proposition}[theorem]{Proposition}

\newtheorem{prop}[theorem]{Proposition}

\theoremstyle{definition}
\newtheorem{definition}[theorem]{Definition}
\newtheorem{Def}[theorem]{Definition}
\newtheorem{example}[theorem]{Example}

\theoremstyle{remark}
\newtheorem{remark}[theorem]{Remark}

\numberwithin{equation}{section}

\newcommand{\norm}[1]{\lVert#1\rVert}
\newcommand{\op}{\operatorname}

\newcommand{\be}{\begin{equation}}
\newcommand{\ee}{\end{equation}}
\newcommand{\Ga}{\Gamma}

\newcommand{\R}{\mathbb R}
\renewcommand{\H}{\mathbb H}
\newcommand{\Z}{\mathbb Z}
\newcommand{\N}{\mathbb N}
\newcommand{\ga}{\gamma}

\newcommand{\la}{\lambda}
\newcommand{\La}{\Lambda}
\newcommand{\inte}{\op{int}}
\newcommand{\ba}{\backslash}
\newcommand{\ov}{\overline}

\newcommand{\cal}{\mathcal}
\newcommand{\br}{\mathbb R}

\newcommand{\Isom}{\op{Isom}}
\newcommand{\PSL}{\op{PSL}}
\newcommand{\F}{\cal F}

\newcommand{\bH}{\mathbb H}

\newcommand{\m}{\mathsf{m}}

\newcommand{\T}{\op{T}}
\renewcommand{\frak}{\mathfrak}

\newcommand{\e}{\varepsilon}
\newcommand{\BR}{\op{BR}}

\newcommand{\BMS}{\op{BMS}}

\newcommand{\fa}{\mathfrak a}

\renewcommand{\i}{\op{i}}

\renewcommand{\S}{\mathbb S}

\newcommand{\Gr}{\Gamma_\rho}
\newcommand{\C}{\cal C}

\newcommand{\id}{\op{id}}

\newcommand{\ess}{\mathsf{E}}
\newcommand{\B}{\mathcal{B}}
\newcommand{\fg}{\frak g}

\newcommand{\rank}{\op{rank}}
\newcommand{\Lie}{\op{Lie}}

\newcommand{\Hor}{\mathcal{H}}
\newcommand{\supp}{\op{supp}}
\newcommand{\Leb}{\op{Leb}}
\begin{document}

\title[Conformal measure rigidity and horospherical foliations]{Conformal measure rigidity and\\ergodicity of horospherical foliations}

\author{Dongryul M. Kim}
\address{Department of Mathematics, Yale University, New Haven, CT 06511}
\email{dongryul.kim@yale.edu}

\begin{abstract}
    In this paper, we prove two main theorems: conformal measure rigidity and ergodicity of horospherical foliations, especially in higher rank. Both theorems are new even for relatively Anosov groups.

    First, we establish a higher rank extension of rigidity theorems of Sullivan, Tukia, Yue, and Kim-Oh for representations of rank one discrete subgroups of divergence type, in terms of the push-forwards of conformal measures via boundary maps. We consider a certain class of higher rank discrete subgroups, which we call hypertransverse subgroups. It includes all rank one discrete subgroups, Anosov subgroups, relatively Anosov subgroups, and their subgroups. 
    Our proof of the rigidity theorem generalizes the idea of Kim-Oh to self-joinings of higher rank hypertransverse subgroups, overcoming the lack of $\mathrm{CAT}(-1)$ geometry on symmetric spaces. In contrast to the work of Sullivan, Tukia, and Yue, our argument is closely related to the study of horospherical foliations. 
    
    We also show the ergodicity of horospherical foliations with respect to Burger-Roblin measures. This generalizes the classical work of Hedlund, Burger, and Roblin in rank one and of Lee-Oh for Borel Ansov subgroups in higher rank. Moreover, we describe the ergodic decomposition of Burger-Roblin measures and Bowen-Margulis-Sullivan measures when a given parabolic subgroup is minimal.
\end{abstract}

\maketitle
\tableofcontents
\section{Introduction}

The celebrated rigidity theorem of Mostow \cite{Mostowbook} (see also Prasad {\cite[Theorem B]{Prasad1973strong}} for non-uniform lattices) states that  if $\Ga$ is a lattice of $G = \Isom^+(\H^n)$, $n \ge 3$, then any discrete faithful representation $\rho : \Ga \to G$ extends to a Lie group isomorphism $G \to G$. The crucial part of Mostow's proof is that there exists a $\rho$-equivariant homeomorphism $f : \S^{n-1} \to  \S^{n-1}$ which is quasiconformal.

Sullivan showed a rigidity theorem for discrete faithful representations of a general discrete subgroup of $G$, extending Mostow's rigidity theorem. Let $\Ga <G $ be a discrete subgroup and denote its limit set by $\La_{\Ga} \subset \S^{n-1}$. We also denote the critical exponent of $\Ga$ by $\delta_{\Ga}$, which is defined as the abscissa of convergence of the Poincar\'e series $\sum_{\ga \in \Ga} e^{-sd(o, \ga o)}$, $o \in \H^n$.
We say that $\Ga$ is of divergence type if the Poincar\'e series diverges at $s = \delta_{\Ga}$. In this case, there exists a unique $\delta_{\Ga}$-dimensional $\Ga$-conformal measure $\nu_{\Ga}$ on $\S^{n-1}$ and it charges the full mass on the conical limit set, in particular $\nu_{\Ga}(\La_{\Ga}) = 1$ \cite{Sullivan1979density}. Sullivan's conformal measure rigidity theorem is as follows:

\begin{theorem}[Sullivan, {\cite[Theorem 5]{Sullivan1982discrete}}]\label{thm.sullivan}
    Let $\Ga < G$  be a discrete subgroup of divergence type and $\rho : \Ga \to G$ a discrete faithful representation such that $\delta_{\rho(\Ga)} = \delta_{\Ga}$. If $\rho$ admits an equivariant continuous embedding $f : \La_{\Ga} \to \S^{n-1}$ and\footnote{The notation $\nu_1 \ll \nu_2$ means that $\nu_1$ is absolutely continuous with respect to $\nu_2$.} $$\nu_{\rho(\Ga)} \ll f_{*}\nu_{\Ga}$$ for some $\delta_{\rho(\Ga)}$-dimensional $\rho(\Ga)$-conformal measure $\nu_{\rho(\Ga)}$, then $\rho$ extends to a Lie group isomorphism $G \to G$.
\end{theorem}

Later, Tukia \cite[Theorem 3C]{Tukia1989rigidity} showed that the condition $\delta_{\rho(\Ga)} = \delta_{\Ga}$ is not necessary in Theorem \ref{thm.sullivan}. Since a quasiconformal homeomorphism preserves the Lebesgue measure class on $\S^{n-1}$, Mostow's rigidity theorem also follows from Theorem \ref{thm.sullivan}. Push-forwards of conformal measures via boundary maps were also used in the rigidity theorem of Besson-Courtois-Gallot (\cite{BCG}, \cite{BCG_etds}).

Sullivan and Tukia's proofs use the ergodicity of the geodesic flow with respect to the Bowen-Margulis-Sullivan measure on the unit tangent bundle of $\Ga\ba \H^n$ to deduce the conformality of the boundary map $f$ relying on the negatively curved geometry of the real hyperbolic space. Generalizing this idea, Yue extended Theorem \ref{thm.sullivan} to general rank one spaces \cite[Theorem A]{Yue1996mostow}.

In our recent  series of work with Oh (\cite{kim2022rigidity}, \cite{kim2023conformal},  \cite{kim2022rigidity2}), we introduced a new approach in the study of rigidity problems on a representation $\rho$ of a discrete subgroup (possibly with infinite-covolume), 
 that is, considering the self-joining of $\Ga$ via $\rho$ and relating to the higher rank dynamics of the self-joining subgroup.
Especially in \cite{kim2023conformal}, the conformal measure rigidity was studied 
 for representations of a rank one discrete subgroup into a simple real algebraic group of general rank, whose orbit maps are extended to Furstenberg boundaries.  It recovers Theorem \ref{thm.sullivan} as well as the work of Tukia and Yue.

\subsection*{Conformal measure rigidity in general rank}
In this paper, we establish the conformal measure rigidity theorem for representations of discrete subgroups of a general-rank simple real algebraic group. We consider a certain class of discrete subgroups, which we call hypertransverse subgroups. This includes rank one discrete subgroups, Anosov subgroups, relatively Anosov subgroups, and their subgroups.

We first introduce some terminologies and notations that we use throughout the paper. Let $G$ be a connected semisimple real algebraic group. Let $P<G$ be a minimal parabolic subgroup with a fixed Langlands decomposition $P=MAN$ where $A$ is a maximal real split torus of $G$, $M$ is the maximal compact subgroup of $P$ commuting with $A$ and $N$ is the unipotent radical of $P$.
Let $\fg = \Lie G$ and $\fa = \Lie A$. Fix a positive Weyl chamber $\fa^+<\fa$ and
set $A^+=\exp \fa^+$. We fix a maximal compact subgroup $K < G$ such that the Cartan decomposition $G=K A^+ K$ holds. We use the notation $\mu : G \to \fa^+$ for the Cartan projection, defined by the condition that $g\in K\exp \mu(g) K$ for $g \in G$. We have the associated Riemannian symmetric space $X = G/K$ and write $o = [K] \in X$.

Let $\Pi$ denote the set of all simple roots for $(\frak g, \frak a^+)$. As usual, the Weyl group is the quotient of the normalizer of $A$ in $K$ by the centralizer of $A$ in $K$. We also denote the opposition involution by
$\i:\fa \to \fa$. It induces an involution on $\Pi$ which we denote by the same notation $\i$.
Throughout the paper, we fix a non-empty subset $$\theta \subset \Pi.$$
Let $\mathfrak{a}_{\theta} =\bigcap_{\alpha \in \Pi - \theta} \ker \alpha$ and let $p_{\theta} : \fa \to \fa_{\theta}$ be the unique projection, invariant under all Weyl elements fixing $\fa_{\theta}$ pointwise. We write $\mu_{\theta} := p_{\theta} \circ \mu$.
Let $P_{\theta}$ be the standard parabolic subgroup corresponding to $\theta$ (our convention is that $P=P_{\Pi}$) and consider the $\theta$-boundary: $$\F_{\theta}=G/P_{\theta}.$$
 We say that $\xi\in \F_{\theta}$ and $\eta\in \F_{\i(\theta)}$ are in general position if the pair
$(\xi,\eta)$ belongs to the unique open $G$-orbit in $\F_{\theta}\times \F_{\i(\theta)}$ under the diagonal action of $G$.

\subsection*{Conformal measures}

Denote by $\fa_{\theta}^*=\op{Hom}(\fa_{\theta}, \R)$ the space of all linear forms on $\fa_{\theta}$.
For $\psi \in \fa_{\theta}^*$ and a closed subgroup $\Delta < G$, a Borel probability measure $\nu$ on $\mathcal{F}_{\theta}$ is called a $(\Delta, \psi)$-conformal measure (with respect to $o \in X$)
if $${d g_* \nu \over d \nu}(\xi) = e^{\psi(\beta_{\xi}^{\theta}(o, go))} \quad \mbox{for all } g \in \Delta \mbox{ and } \xi \in \F_{\theta}$$
 where ${g}_* \nu(D) = \nu(g^{-1}D)$ for any Borel subset $D\subset \F_{\theta}$ and $\beta_\xi^{\theta}$ denotes the $\fa_{\theta}$-valued Busemann map (Definition \ref{buse}). By a $\Delta$-conformal measure, we mean a $(\Delta, \psi)$-conformal measure for some $\psi \in \fa_{\theta}^*$. The linear form $\psi$ plays a role of \emph{dimension} of $\nu$.

We say that a $(\Delta, \psi)$-conformal measure $\nu$ on $\F_{\theta}$ is of \emph{divergence type} if $\psi$ is $(\Delta, \theta)$-proper\footnote{a linear form $\psi \in \fa_{\theta}^*$ is called $(\Delta, \theta)$-proper if $\psi \circ \mu_{\theta} : \Delta \to [-\varepsilon, \infty)$ is a proper map for some $\varepsilon > 0$.} and $\sum_{g \in \Delta} e^{-\psi(\mu_{\theta}(g))} = \infty$. The $(\Delta, \theta)$-properness hypothesis guarantees that the abscissa of convergence of the Poincar\'e series $\sum_{g \in \Delta} e^{-s \psi(\mu_{\theta}(g))}$, which we denote by $\delta_{\psi}$, is well-defined.

\subsection*{Hypertransverse subgroups}

 Let $\Ga<G$ be a Zariski dense discrete subgroup. We denote by $\La^{\theta} := \La_{\Ga}^{\theta} \subset \F_{\theta}$ the limit set of $\Ga$ in $\F_{\theta}$, which is the unique $\Ga$-minimal subset of $\F_{\theta}$ (Definition \ref{def.limitset}). A discrete subgroup $\Ga$ is called $\theta$-transverse if 
\begin{itemize}
    \item   $\Ga$ is $\theta$-regular, i.e., $ \liminf_{\ga\in \Ga} \alpha(\mu_{\theta}({\ga}))=\infty $ for all $\alpha\in \theta$; and
    \item $\Ga$ is $\theta$-antipodal, i.e., any two distinct $\xi, \eta\in \La^{\theta\cup \i(\theta)}$ are in general position.   
\end{itemize}
Most of the known examples of transverse subgroups come with nice actions on Gromov hyperbolic spaces. In this regard, we consider the following subclass: 
\begin{definition}
    A $\theta$-transverse subgroup $\Ga < G$ is called \emph{$\theta$-hypertransverse} if there exists a proper geodesic Gromov hyperbolic space $Z$ such that
    \begin{itemize}
        \item $\Ga$ acts on $Z$ properly discontinuously by isometries;
        \item there exists a $\Ga$-equivariant homeomorphsim $$\La^Z \to \La^{\theta}$$
        where $\La^Z$ is the limit set of $\Ga$ in the Gromov boundary $\partial Z$.
    \end{itemize}
\end{definition}

\begin{example} \label{ex.relanosov}
As mentioned before, any subgroup of an Anosov or a relatively Anosov group (Definition \ref{def.relAnosov}) is hypertransverse. Indeed, when $\Ga$ is a subgroup of an Anosov group $\Ga_0$, we can take $Z$ to be the Cayley graph of $\Ga_0$. For a subgroup $\Ga$ of a relatively Anosov group $\Ga_0$, we can set $Z$ to be the Groves-Manning cusp space of $\Ga_0$.

\end{example}

It seems that most transverse subgroups are hypertransverse. We do not know  of an example of a transverse subgroup which is  not hypertransverse.

 \subsection*{Rigidity theorems}

 Let $G_1, G_2$ be connected simple real algebraic groups. Let $\theta_1$ and $\theta_2$ be non-empty subsets of the set of simple roots of $G_1$ and $G_2$ respectively.
Here is our main rigidity theorem:

 \begin{theorem}[Conformal measure rigidity] \label{thm.main}
 
    Let $\Ga < G_1$ be a Zariski dense $\theta_1$-hypertransverse subgroup. 
    Let $\rho : \Ga \to G_2$ be a Zariski dense $\theta_2$-regular\footnote{i.e. $\rho(\Ga)$ is $\theta_2$-regular} faithful representation with a pair of $\rho$-equivariant continuous embeddings $f: \La^{\theta_1} \to \F_{\theta_2}$ and $f_{\i} : \La^{\i(\theta_1)} \to \F_{\i(\theta_2)}$. If there exists a $\Ga$-conformal measure $\nu_{\Ga}$ of divergence type such that $$\nu_{\rho(\Ga)} \ll f_* \nu_{\Ga}$$ for some $\rho(\Ga)$-conformal measure $\nu_{\rho(\Ga)}$, then $\rho$ extends to a Lie group isomorphism $G_1 \to G_2$.
 \end{theorem}

 \begin{remark}
    We emphasize that there is no additional assumption on the conformal measure $\nu_{\rho(\Ga)}$ and its associated linear form, such as $(\rho(\Ga), \theta_2)$-properness. Moreover, we do not assume that the image $\rho(\Ga)$ is transverse.
 \end{remark}

 We note that if a $\rho$-equivariant map $\La^{\theta_1} \to \F_{\theta_2}$ exists, then it is unique (Lemma \ref{ll}). We also note that if $\psi \in \fa_{\theta_1}^*$ is $(\Ga, \theta_1)$-proper and the associated Poincar\'e series diverges at $\delta_{\psi}$, then there exists a unique $(\Ga, \delta_{\psi}\psi)$-conformal measure on $\F_{\theta_1}$ and has support $\La^{\theta_1}$ (\cite{canary2023patterson}, \cite{kim2023growth}).
 
 When $\rank G_1 = 1$, a $\Ga$-conformal measure $\nu$ is of divergence type if and only if $\Ga$ is of divergence type and $\nu$ is the $\delta_{\Ga}$-dimensional $\Ga$-conformal measure. Hence Theorem \ref{thm.main} generalizes Theorem \ref{thm.sullivan} as well as the work of Tukia \cite{Tukia1989rigidity} and Yue \cite{Yue1996mostow} in the case that $\rank G_2 = 1$, and the work of Kim-Oh \cite{kim2023conformal} for general $G_2$.

 We also show the following:

 \begin{theorem}[Singularity between conformal measures] \label{thm.mainpsclass}
 Let $G$ be a semisimple real algebraic group and $\Ga < G$ be a Zariski dense $\theta$-hypertransverse subgroup. Let $\nu$ be a $(\Ga, \psi)$-conformal measure of divergence type. Then for any $(\Ga, \psi')$-conformal measure $\nu'$ with $\psi'\ne \psi$, we have  $$\nu' \not\ll \nu.$$
 In particular, if $\nu'$ is further assumed to be of divergence type, then   $$\nu' \perp \nu.$$
 \end{theorem}

 Theorem \ref{thm.mainpsclass} also follows from the work of Lee-Oh \cite{lee2020invariant} and of Sambarino \cite{sambarino2022report} when $\Ga$ is $\theta$-Anosov.
Recently, Blayac-Canary-Zhu-Zimmer \cite{BCZZ} showed a related result on singularity/absolute continuity between two Patterson-Sullivan measures in a more abstract setting, assuming that they are supported on the limit set.

 \subsection*{Relatively Anosov subgroups}
Theorem \ref{thm.main} and Theorem \ref{thm.mainpsclass} apply to any Zariski dense subgroup $\Ga$ of a relatively Anosov subgroup $\Delta$ (see Example \ref{ex.relanosov}).
 The notion of relatively Anosov subgroups (resp. Anosov subgroups) was introduced as a higher rank extension of geometrically finite subgroups (resp. convex cocompact subgroups); see (\cite{Labourie2006anosov}, \cite{Guichard2012anosov}, \cite{Kapovich2017anosov}, \cite{GGKW_gt}, \cite{CZZ_relative}, \cite{ZZ_relatively}).
\begin{definition} \label{def.relAnosov}
 Let $\Delta < G$ be a $\theta$-transverse subgroup and $\cal P$ a finite collection of subgroups in $\Ga$. 
     We say that $\Delta < G$ is $\theta$-Anosov relative to $\cal P$ if $\Delta$ is a hyperbolic group relative to $\cal P$, and there exists a $\Delta$-equivariant homeomorphism from the Bowditch boundary $\partial (\Delta, \cal P)$ to $\La^{\theta}_{\Delta}$. When $\cal P = \emptyset$, $\Delta$ is called $\theta$-Anosov.
 \end{definition}

 \subsection*{Rigidity of transverse representations} 
 We also consider a conjugate between two transverse representations, which can be regarded as a deformation between them. Let $(Z, d_Z)$ be a proper geodesic Gromov hyperbolic space. Let $\Delta < \Isom(Z)$ be a subgroup of isometries of $Z$ acting properly discontinuously on $Z$. We denote by $\La_{\Delta}^Z \subset \partial Z$ its limit set in the Gromov boundary. 
 
 For $i = 1, 2$, let $G_i$ be a simiple real algebraic group and $\rho_i : \Delta \to G_i$ a Zariski dense $\theta_i$-transverse representation, in the sense that $\rho_i(\Delta) < G_i$ is a $\theta_i$-transverse subgroup and that there exists a $\rho_i$-equivariant homeomorphism $f_i : \La^Z_{\Delta} \to \La^{\theta_i}_{\rho_i(\Delta)}$. We set $\Ga_i := \rho_i(\Delta)$. Then the isomorphism $\rho := \rho_2 \circ \rho_1|_{\Ga_1}^{-1}$ conjugates two representations $\rho_1$ and $\rho_2$:
$$\begin{tikzcd}[column sep = large, row sep = tiny]
    & \Ga_1  \arrow[dd, "\rho"]  \\
    \Delta \arrow[ru, "\rho_1"] \arrow[rd, "\rho_2"'] &  \\
    & \Ga_2
\end{tikzcd}$$
The following is an immediate consequence of Theorem \ref{thm.main} where $f$ is the $\rho$-boundary map:

\begin{theorem} \label{thm.main2}
Let $\rho : \Ga_1 \to \Ga_2$ be a conjugate between two Zariski dense transverse representations $\rho_1 : \Delta \to \Ga_1$ and $\rho_2 : \Delta \to \Ga_2$.
Suppose that $\rho$ does not extend to a Lie group isomorphism $G_1 \to G_2$. For any $\Ga_1$-conformal measure $\nu_1$ of divergence type and $\Ga_2$-conformal measure $\nu_2$, 
$$\nu_2 \not\ll f_* \nu_1.$$
In particular, if $\nu_2$ is further assumed to be of divergence type, then $$\nu_2 \perp f_* \nu_1.$$
\end{theorem}

In other words, we have the following rigidity theorem for transverse representations:

\begin{theorem} \label{thm.maintrans}
    If there exist a $\Ga_1$-conformal measure $\nu_1$ of divergence type such that $$\nu_2 \ll f_* \nu_1$$ for some $\Ga_2$-conformal measure $\nu_2$, then $\rho$ extends to a Lie group isomorphism $G_1 \to G_2$.
\end{theorem}

\subsection*{Ergodicity of horospherical foliations}

The horospherical foliation $\Hor$ of the unit tangent bundle $\T^1(\bH^2)$ is a collection of horospheres stable under the geodesic flow. This can be identified as follows: $$\Hor = \partial \H^2 \times \R = \PSL_2(\R)/N$$ where $N = \left\{ \begin{pmatrix} 1 & s \\
    0 & 1
\end{pmatrix} : s \in \R \right\}$.
Hedlund showed that when $\Ga$ is a lattice in $\PSL_2(\R)$, the $\Ga$-action on $\Hor$ is ergodic with respect to the Haar measure \cite{Hedlund}. Later, Burger \cite{Burger_horoc} proved that any convex cocompact $\Ga < \PSL_2(\R)$ with $\delta_{\Ga} > 1/2$ acts ergodically on $\Hor$ with respect to the measure
$$e^{\delta_{\Ga} t} d\nu_{\Ga} dt$$ where  $d\nu_{\Ga}$ is the unique $\delta_{\Ga}$-dimensional $\Ga$-conformal measure on $\partial \H^2$ and $dt$ is a Lebesgue measure on $\R$. This measure on $\Hor = \partial \H^2 \times \R$ is called the Burger-Roblin measure which we denote by $m_{\Ga}^{\BR}$.
Roblin extended these results to a more general setting: 
\begin{theorem} \cite[Corollary 2.3]{Roblin2003ergodicite} \label{thm.Roblin}
    Let $(X, d)$ be a proper $\op{CAT}(-1)$-space and $\Ga < \Isom(X)$ a discrete subgroup.
    Suppose that $\sum_{\ga \in \Ga} e^{-\delta_{\Ga} d(o, \ga o)} = \infty$ and the length spectrum of $\Ga$ is non-arithmetic\footnote{The non-arithmeticity of the length spectrum means that the set of all lengths of closed geodesics in $\Ga \ba X$ generates a dense subgroup of $(\R, +)$.}. Then the $\Ga$-action on the horospherical foliation of $X$ with respect to $m_{\Ga}^{\BR}$ is ergodic.
\end{theorem}

When $X$ is a rank one symmetric space, the length spectrum of any non-elementary discrete subgroup $\Ga < \Isom(X)$ is non-arithmetic \cite{IKim_nonarith}.
Hence Theorem \ref{thm.Roblin} implies that any discrete subgroup $\Ga < \Isom(X)$  of divergence type acts ergodically on the horospherical foliation of $\T^1 (X)$ with respect to the Burger-Roblin measure.

We extend Theorem \ref{thm.Roblin} to higher rank settings.
In the rest of the introduction, let $G$ be a connected semisimple real algebraic group and fix a non-empty $\theta \subset \Pi$. We then have the Langlands decomposition $P_{\theta} = N_{\theta}S_{\theta}A_{\theta}$ where $A_{\theta} = \exp \fa_{\theta}$, $S_\theta$ is an almost direct product of a semisimple algebraic group and a compact central torus, and $N_{\theta}$ is the unipotent radical of $P_{\theta}$. The $\theta$-horospherical foliation is defined as $$\cal H_{\theta} := \F_{\theta} \times \fa_{\theta} = G/N_{\theta}S_{\theta}$$ analogous to the rank one setting. Since $A_{\theta}$ normalizes $N_{\theta}S_{\theta}$, it acts on $\Hor_{\theta}$ on the right by multiplication (see \eqref{eqn.actions}).

Let $\Ga < G$ be a discrete subgroup and $\nu$ a $(\Ga, \psi)$-conformal measure on $\F_{\theta}$.
The Burger-Roblin measure associated to $\nu$ is a $\Ga$-invariant Radon measure on $\Hor_{\theta}$ defined as $$dm_{\nu}^{\BR}(\xi, u) := e^{\psi(u)} d\nu(\xi) du$$ where $du$ is the Lebesgue measure on $\fa_{\theta}$. Note that $$\supp m_{\nu}^{\BR} = \{ (\xi, u) \in \Hor_{\theta} : \xi \in \supp \nu\}.$$

Here is a higher rank version of Theorem \ref{thm.Roblin} relating the ergodicity of $(\Hor_{\theta}, \Ga, \nu_{\nu}^{\BR})$ with the divergence of the $\psi$-Poincar\'e series:
\begin{theorem} \label{thm.ergodichoro}
    Let $\Ga < G$ be a Zariski dense $\theta$-hypertransverse subgroup. For any $\Ga$-conformal measure $\nu$ of divergence type, $$\mbox{the }\Ga\mbox{-action on the horospherical foliation } (\Hor_{\theta}, m_{\nu}^{\BR}) \mbox{ is ergodic.}$$
\end{theorem}

\subsection*{Horospherical actions on $\Ga \ba G$}
Considering the case $\theta = \Pi$, we have $N_{\Pi} = N$ and $S_{\Pi} = M$, and hence $$\cal H_{\Pi} = G/NM.$$
By the Iwasawa decomposition $G = KAN = KP$, the Furstenberg boundary $\F$ is identified with $K/M$. For a $(\Ga, \psi)$-conformal measure $\nu$ on $\F$, we denote by $\hat \nu$ the $M$-invariant lift of $\nu$ to $K$. 
We then define the following $\Ga$-invariant measure on $G$: for $g = k (\exp u) n \in KAN$, 
\be \label{eqn.defBRGG}
d \hat m_{\nu}^{\BR}(g) := e^{\psi(u)} d \hat \nu(k) du dn
\ee where $dn$ is the Haar measure on $N$. The measure $\hat m_{\nu}^{\BR}$ is the $NM$-invariant lift of $m_{\nu}^{\BR}$ to $G$, and induces the measure on $\Ga \ba G$ which we also call the Burger-Roblin measure and denote by $\hat m^{\BR}_{\nu}$, abusing notations. 

We consider the horospherical action on $(\Ga \ba G, \hat m_{\nu}^{\BR})$, given as the right multiplication by $NM$.
Since any conformal measure of divergence type is supported on the limit set \cite[Theorem 1.5]{kim2023growth}, Theorem \ref{thm.ergodichoro} can be rephrased as follows in this case:

\begin{theorem} \label{thm.ergodichoroPi}
    Let $\Ga < G$ be a Zariski dense $\Pi$-hypertransverse subgroup. For any $\Ga$-conformal measure $\nu$ on $\F$ of divergence type,
    $$\text{the } NM\text{-action on } (\Ga \ba G, \hat m^{\BR}_{\nu}) \text{ is ergodic.}$$
    In particular, for $\hat{m}_{\nu}^{\BR}$-a.e. $x \in \Ga \ba G$, we have $$\overline{x N M} = \{ [g] \in \Ga \ba G : gP \in \La^{\Pi} \}.$$
\end{theorem}

Theorem \ref{thm.ergodichoroPi} applies to the images of cusped Hitchin representations, which are well-known examples of relatively $\Pi$-Anosov subgroups (see \cite{bray2021counting} and \cite{canary2021cusped}).

Theorem \ref{thm.ergodichoro} and Theorem \ref{thm.ergodichoroPi} extend the work of Lee-Oh \cite{lee2020invariant} on $\Pi$-Anosov subgroups. See also \cite{LLLO} for a certain unique ergodicity result of Burger-Roblin measures for special types of $\Pi$-Anosov subgroups.

\subsection*{Ergodic decomposition}

 Lee-Oh described the ergodic decomposition of the Burger-Roblin and Bowen-Margulis-Sullivan measures on $\Ga \ba G$ for a Zariski dense $\Pi$-Anosov subgroup $\Ga < G$ \cite{lee2020ergodic}. 
In view of Theorem \ref{thm.ergodichoroPi}, their argument applies
to $\Pi$-hypertransverse subgroups, and hence
yields similar ergodic decomposition theorems.

For a $\Ga$-conformal measure $\nu$ on $\F$ of divergence type, denote by $\hat m_{\nu}^{\BMS}$ the associated Bowen-Margulis-Sullivan measure on $\Ga \ba G$ (see \eqref{eqn.defbmsgabag} for the precise definition). To state the ergodic decompositions of $\hat m_{\nu}^{\BR}$ and $\hat m_{\nu}^{\BMS}$, let $\mathfrak{D}_{\Ga}$ be the (finite) collection of all $P^\circ$-minimal subsets of $\Ga \ba G$ where $P^{\circ}$ is the identity component of $P$.

\begin{theorem} \label{thm.ergdecompintro}
    Let $\Ga < G$ be a Zariski dense $\Pi$-hypertransverse subgroup. Let $\nu$ be a $\Ga$-conformal measure on $\F$ of divergence type. Then
    \begin{enumerate}
        \item $\hat m_{\nu}^{\BR} = \sum_{\cal E \in \frak D_{\Ga}} \hat m_{\nu}^{\BR}|_{\cal E}$ is an $N$-ergodic decomposition;
        \item $\hat m_{\nu}^{\BMS} = \sum_{\cal E \in \frak D_{\Ga}} \hat  m_{\nu}^{\BMS}|_{\cal E}$ is an $A$-ergodic decomposition.
    \end{enumerate}
    In particular, the number of $N$-ergodic components of $\hat m_{\nu}^{\BR}$ and the number of $A$-ergodic components of $\hat m_{\nu}^{\BMS}$ are given by $\# \frak D_{\Ga} = [P : P_{\Ga}]$, where $P_{\Ga} := \{p \in P : \cal E_0 p = \cal E_0\}$ for any $\cal E_0 \in \frak D_{\Ga}$.
\end{theorem}

As in Theorem \ref{thm.ergodichoroPi}, the above ergodic decomposition implies density of almost every $N$-orbit and $A$-orbit in each $\cal E \in \frak D_{\Ga}$. Moreover, together with the Hopf-Tsuji-Sullivan dichotomy for transverse subgroups by Canary-Zhang-Zimmer \cite{canary2023patterson} and by Kim-Oh-Wang \cite{kim2023growth}, we deduce the following from Theorem \ref{thm.ergdecompintro}.

\begin{theorem} \label{thm.denseAplusintro}
    Let $\Ga < G$ be a Zariski dense $\Pi$-hypertransverse subgroup. Let $\nu$ be a $\Ga$-conformal measure on $\F$ of divergence type. Then for any $\cal E \in \frak D_{\Ga}$ and $\hat m_{\nu}^{\BMS}$-a.e. $x \in \cal E$,
    $$
    \ov{x A^+} = \supp \hat m_{\nu}^{\BMS}|_{\cal E}.
    $$
    In particular, if $P$ is connected, then for $\hat m_{\nu}^{\BMS}$-a.e. $x \in \Ga \ba G$,
    $$
    \ov{x A^+} = \{ [g] \in \Ga \ba G : g P, g w_0 P \in \La^{\Pi} \}
    $$
    where $w_0 \in N_K(A)$ is the longest Weyl element.
\end{theorem}

 \subsection*{On the proofs}
As mentioned before, Theorem \ref{thm.main} was proved by Kim-Oh \cite{kim2023conformal} when $\Ga$ is a rank one discrete subgroup and $\rho$ is a certain representation into a higher rank simple group. The major point of this paper is to extend the proof of \cite{kim2023conformal} to higher rank. Under the rank one assumption on $\Ga$, the symmetric space is $\op{CAT}(-1)$, and hence Busemann functions and Gromov products behave nicely enough. However, when $\Ga$ is of higher rank, the symmetric space is neither negatively curved nor $\op{CAT}(-1)$, and hence it requires additional ideas to prove Theorem \ref{thm.main} in this generality. Even under the hypertransverse hypothesis on $\Ga$, it only admits an action on a Gromov hyperbolic space, and the coarse nature of the geometry of Gromov hyperbolic spaces still presents several non-trivial difficulties in extending previous works on the conformal measure rigidity \cite{kim2023conformal} and  the ergodicity of horospherical foliations (\cite{Roblin2003ergodicite}, \cite{lee2020invariant}).

In contrast to the proof of Sullivan, Tukia, and Yue, we consider the self-joining of $\Ga$ via  $\rho  : \Ga \to G_2$ to prove Theorem \ref{thm.main}: $$\Gr := (\id \times \rho)(\Ga) = \{ (\ga, \rho(\ga)) : \ga \in \Ga \},$$ a discrete subgroup of $G := G_1 \times G_2$.   As $G_1$ and $G_2$ are simple, $\rho$ extends to a Lie group isomorphism $G_1 \to G_2$ if and only if $\Gr$ is not Zariski dense (Lemma \ref{Zdense}). Hence we translate the rigidity question on $\rho$ to the Zariski density question on the self-joining $\Gr$. The idea relating dynamics of self-joinings to rigidity problems in terms of boundary maps originates from the work of Kim-Oh (\cite{kim2022rigidity}, \cite{kim2022rigidity2}, \cite{kim2023conformal}). 

As in \cite{kim2023conformal}, we let $A := A_1 \times A_2$ and $A^+ := A_1^+ \times A_2^+$ so that $\fa := \Lie A = \fa_1 \oplus \fa_2$ and $\fa^+ = \fa_1^+ \oplus \fa_2^+$. Denote by $\Pi = \Pi_1 \cup \Pi_2$ the set of all simple roots for $(\Lie G, \fa^+)$ and $\theta = \theta_1 \cup \theta_2$. We have the $\theta$-boundary $\F_{\theta} = \F_{\theta_1} \times \F_{\theta_2}$ and $\fa_{\theta} = \fa_{\theta_1} \oplus \fa_{\theta_2}$.

For a $(\Ga, \psi)$-conformal measure $\nu$ of divergence type for $\psi \in \fa_{\theta_1}^*$, recall the graph-conformal measure defined by Kim-Oh \cite{kim2023conformal} which is the push-forward $$\nu_\rho := (\id \times f)_*\nu.$$ As observed  in \cite{kim2023conformal}, $\nu_{\rho}$ is $(\Gr, \sigma_{\psi})$-conformal where $\sigma_{\psi} \in \fa_{\theta}^*$ is defined by $\sigma_{\psi}(u_1, u_2) = \psi(u_1)$ for $(u_1, u_2) \in \fa_{\theta_1} \oplus \fa_{\theta_2}$.

The main technical ingredient of the proof is to show that if $\Gr$ is Zariski dense, the essential subgroup $\ess_{\nu_{\rho}}^{\theta}(\Gr)$ is the whole $\fa_{\theta}$ (Theorem \ref{thm.ess}). The essential subgroup $\ess_{\nu_{\rho}}^{\theta}(\Gr) \subset \fa_{\theta}$ is defined as the set of $u \in \fa_{\theta}$ such that for any $\varepsilon > 0$ and a Borel subset $B \subset \F_{\theta}$ with $\nu_{\rho}(B) > 0$, the subset $$B \cap \ga B \cap \{\xi \in \F_{\theta} : \|\beta_{\xi}^{\theta}(e, \ga) - u \| < \varepsilon \}$$ has a positive $\nu_{\rho}$-measure for some $\ga \in \Gr$. That $\ess_{\nu_{\rho}}^{\theta}(\Gr) = \fa_{\theta}$ implies the singularity of $\nu_{\rho}$ among conformal measures of $\Gr$, and hence the singularity of $f_* \nu$ between $\rho(\Ga)$-conformal measures (Proposition \ref{prop.loabscont}). Therefore, the non-singularity between $f_* \nu$ and a $\rho(\Ga)$-conformal measure forbids $\Gr$ from being Zariski dense and hence $\rho$ extends to a Lie group isomorphism $G_1 \to G_2$.

To show $\ess_{\nu_{\rho}}^{\theta}(\Gr) = \fa_{\theta}$, we first prove that the Myrberg limit sets of $\Ga$ and of the self-joining $\Ga_{\rho}$ have full $\nu$ and $\nu_{\rho}$-measures respectively (Theorem \ref{thm.fullmyrberg}), only assuming that $\Ga$ is a transverse subgroup. This is based on the ergodicity of an appropriate one-dimensional flow obtained in \cite[Theorem 10.2]{kim2023growth} and is new even for a transverse subgroup $\Ga$.

We then deduce $\ess_{\nu_{\rho}}^{\theta}(\Gr) = \fa_{\theta}$ from the full $\nu_{\rho}$-mass of the Myrberg limit set of $\Ga_{\rho}$ under the additional hypothesis that $\Ga$ is hypertransverse. 
The idea of this deduction is influenced by Roblin \cite{Roblin2003ergodicite} in the $\op{CAT}(-1)$ setting, by Kim-Oh \cite{kim2023conformal} for self-joinings of rank one discrete subgroups, and by Lee-Oh \cite{lee2020invariant} dealing with Anosov subgroups with respect to minimal parabolic subgroups. The fact that the visual metrics behave nicely in the $\op{CAT}(-1)$ setting was crucial in (\cite{Roblin2003ergodicite}, \cite{kim2023conformal}). And the higher rank Morse Lemma was heavily used in \cite{lee2020invariant} to show that the visual metric defined in terms of the higher rank Gromov product has the desired properties. 

On the other hand, the coarse feature of a Gromov hyperbolic space does not make visual metrics as good as in the $\op{CAT}(-1)$ setting, and the higher rank Morse lemma is not available in the generality of our setting. To overcome this difficulty, we conduct a detailed investigation of the coarse geometry in Gromov hyperbolic spaces to make the visual metric defined on the Gromov boundary work in higher rank settings, together with a simultaneous sharp-control on the other Busemann function on the higher rank symmetric space (see Remark \ref{rmk.ess}).

Finally, we deduce the ergodicity of the horospherical foliation for hypertransverse subgroups with respect to Burger-Roblin measures based on our investigation on the essential subgroup.

\subsection*{Organization}

In Section \ref{pre}, we review basic notions and properties of semisimple real algebraic groups and their boundaries. In Section \ref{sec.confess} we introduce the notion of essential subgroups for conformal measures and discuss how the essential subgroup of a given conformal measure detects its singularity among conformal measures. In Section \ref{set}, we define self-joining subgroups and graph-conformal measures, and describe their key role in studying rigidity questions. We introduce the notion of Myrberg limit set in higher rank in Section \ref{sec.Myrberg}. We prove that the Myrberg limit set of a self-joining has full mass with respect to the graph-conformal measure. Section \ref{sec.hyperbolic} is devoted to the discussion on quantitative aspects of the geometry of Gromov hyperbolic spaces. In Section \ref{sec.ess}, we show that the essential subgroup for a graph-conformal measure is the whole $\fa_{\theta}$ under the Zariski dense hypothesis on the self-joining. In Section \ref{sec.proof}, we establish the singularity of graph-conformal measures among conformal measures, and provide the proof of Theorem \ref{thm.main}.
In Section \ref{sec.transrigidity}, we prove more stronger rigidity statements (Theorem \ref{thm.main2}, Theorem \ref{thm.maintrans}) for deformations of transverse representations. The ergodicity of horospherical foliations and ergodic decomposition results are proved in Section \ref{sec.BRmeas}.

\subsection*{Acknowledgements}
 I would like to thank my advisor Professor Hee Oh for her encouragement and many helpful conversations.

\section{Preliminaries} \label{sec.higherconf}\label{pre}
Let $G$ be a connected semisimple real algebraic group. We use the notations and terminology introduced in the introduction. We denote by $\cal W = N_K(A)/C_K(A)$ the Weyl group, where $N_K(A)$ and $C_K(A)$ are the normalizer and centralizer of $A$ in $K$ respectively.
Fixing a left $G$-invariant and right $K$-invariant Riemannian metric on $G$ induces a $\cal W$-invariant norm on $\mathfrak a$, which we denote by $\|\cdot\|$. We also denote by $d$ the induced left $G$-invariant metric on the symmetric space $X:=G/K$ and by $o\in X$ the point corresponding to the coset $[K]$. 

Recall that we choose a closed positive Weyl chamber $\fa^+$ of $\fa$ and set $A^+=\exp \mathfrak a^+$. The Cartan projection $\mu : G \to \fa^+$ is defined to be such that $g \in K \exp (\mu(g)) K$ for all $g \in G$.

\begin{lemma} \cite[Lemma 4.6]{Benoist1997proprietes} \label{lem.cptcartan}
For any compact subset $Q \subset G$, there exists $C=C(Q)>0$ such that for all $g \in G$, $$\sup_{q_1, q_2\in Q} \| \mu(q_1gq_2) -\mu(g)\| \le C .$$  
\end{lemma}

Let $\Phi^{+}$ be the set of all positive roots and $\Pi \subset \Phi^+$ the set of all simple roots for $(\fg, \fa^+)$. We fix an element $$w_0\in N_K(A) $$ representing the longest Weyl element. This induces an involution $$\i := - \op{Ad}_{w_0} : \fa \to \fa$$ preserving $\fa^+$, called the opposition involution. This also induces a map $\Phi\to \Phi$ preserving $\Pi$, for which we use the same notation $\i$. We have 
\be \label{mu}
\mu(g^{-1})=\i (\mu(g))\quad\text{ for all $g\in G$ }.
\ee

In the rest of the section,  fix a  non-empty subset $\theta\subset \Pi$. 
Let $ P_\theta$ denote a standard parabolic subgroup of $G$ corresponding to $\theta$; that is,
$P_{\theta}$ is generated by $MA$ and all root subgroups $U_\alpha$,
$\alpha\in \Phi^{+} \cup [\Pi-\theta]$ where $[\Pi-\theta]$ denotes the set of all roots in $\Phi$ which are $\mathbb Z$-linear combinations of $\Pi-\theta$. Hence $P_\Pi=P$.
The subgroup $P_\theta$ is equal to its own normalizer; for $g\in G$,
$gP_\theta g^{-1}=P_\theta$ if and only if $g\in P_\theta$.
Let 
$$
\begin{aligned}
    \mathfrak{a}_\theta &=\bigcap_{\alpha \in \Pi-\theta} \ker \alpha, & \fa_\theta^+ & =\fa_\theta\cap \fa^+, \\
    A_{\theta} & = \exp \fa_{\theta}, \mbox{ and} &    A_{\theta}^+ & = \exp \fa_{\theta}^+.
\end{aligned}
$$
Let $$ p_\theta:\mathfrak{a}\to\mathfrak{a}_\theta$$ denote  the projection invariant under $w\in \cal W$ fixing $\fa_\theta$ pointwise. We also write $\mu_{\theta} := p_{\theta} \circ \mu$.
We denote by $\fa_\theta^*=\op{Hom}(\fa_\theta, \br)$ the dual space of $\fa_\theta$. It can be identified with the subspace of $\fa^*$ which is $p_\theta$-invariant: $\fa_\theta^*=\{\psi\in \fa^*: \psi\circ p_\theta=\psi\}$; so for $\theta \subset \theta'$, we have $\fa_{\theta}^*\subset \fa_{\theta'}^*$.

Let $L_\theta$ be the centralizer of $A_{\theta}$;
it is a Levi subgroup of $P_\theta$ and $P_\theta=L_\theta N_\theta$ where $N_\theta=R_u(P_\theta)$ is the unipotent radical of $P_\theta$. We set $M_{\theta} = K \cap P_{\theta}\subset L_\theta$.
We may then write $L_{\theta} = A_{\theta}S_{\theta}$ where $S_{\theta}$ is an almost direct product of
 a connected semisimple real algebraic subgroup and a compact subgroup. We omit the subscript when $\theta = \Pi$.

\subsection*{The $\theta$-boundary $\F_\theta$} 
The Furstenberg boundary is defined as the quotient $\F = G/P = G/P_{\Pi}$.
The $\theta$-boundary is defined similarly: $$\F_\theta=G/P_{\theta}.$$
Let $$ \pi_\theta:\F\to \F_\theta$$ denote
 the canonical projection map given by $gP\mapsto gP_\theta$, $g\in G$. 
 We set $$\xi_\theta=[P_\theta] \in \F_{\theta}.$$
 By the Iwasawa decomposition $G=KP=KAN$, the subgroup $K$ acts transitively on $\F_\theta$, and hence
 $$\F_\theta\simeq K/ M_\theta.$$

\subsection*{Points in general position} Let $P_\theta^+$ be the
standard parabolic subgroup of $G$ opposite to $P_\theta$ such that $P_\theta\cap P_\theta^+=L_\theta$. We have $P_\theta^+ =w_0 P_{\i(\theta)}w_0^{-1}$ and hence 
$$\F_{\i(\theta)}=G/P_\theta^+.$$
In particular, if $\theta$ is symmetric in the sense that $\theta=\i(\theta)$, then $\F_\theta=G/P_\theta^+$. 
The $G$-orbit
of $(P_\theta, P_\theta^+)$ is the unique open $G$-orbit
in $G/P_\theta\times G/P_\theta^+$ under the diagonal $G$-action.

\begin{Def}
 Two elements
$\xi\in \F_\theta$ and $\eta\in \F_{\i(\theta)}$ are said to be in general position if $(\xi, \eta)\in G \cdot (P_\theta, w_0 P_{\i(\theta)})=G \cdot (P_\theta, P_\theta^+)$, i.e.,
$\xi=gP_\theta$ and $\eta=gw_0 P_{\i(\theta)}$ for some $g\in G$.
\end{Def}

 We set
$$ \F_{\theta}^{(2)}=\{(\xi, \eta) \in \F_{\theta} \times \F_{\i(\theta)}: \xi, \eta \text{ are in general position}\},
$$
which is
the unique open $G$-orbit in $ \F_{\theta} \times \F_{\i(\theta)}$.

\subsection*{Jordan projection}
An element $g\in G$ is
loxodromic if $$g=h a m h^{-1}$$ for some $ a\in \inte A^+$,
$m\in M{}$ and $h\in G$. The Jordan projection of $g$
is defined to be $$\lambda(g):=\log a \in \inte \fa^+.$$ The attracting fixed point of $g$ in $\F$ is given by $$y_{g} := h P \in \F;$$ for any $\xi\in \cal F$ in general position with $y_{g^{-1}}$, the sequence $g^{\ell} \xi$ converges to $y_{g}$ as $\ell \to \infty$. 
We also set $$\la_{\theta}(g) := p_{\theta}(\la(g)) \in \inte \fa_{\theta}^+ \quad \text{and} \quad y_g^{\theta} := \pi_{\theta}(y_g) \in \F_{\theta}.$$
Since $g^{-1} = (hw_0)(w_0^{-1} a^{-1} w_0) (w_0^{-1}m^{-1} w_0) (hw_0)^{-1}$ and $w_0^{-1} a w_0 \in \inte \fa^+$ and $w_0^{-1} m w_0 \in M$ as well, it follows that $y_g^{\theta}$ and $y_{g^{-1}}^{\i(\theta)}$ are in general position.

Let $\Delta<G$ be a discrete subgroup.  We write
$\lambda(\Delta)$ for the set of all Jordan projections of loxodromic elements of $\Delta$.
The following  result is due to Benoist \cite{Benoist2000proprietes}.
\begin{theorem}\label{t1}
\label{dense0} 
If $\Delta<G$ is Zariski dense, then
$\lambda(\Delta)$ generates a dense subgroup of  $\fa$.
In particular, $\la_{\theta}(\Delta)$ generates a dense subgroup of $\fa_{\theta}$.
\end{theorem}

\subsection*{Convergence in $G \cup \F_{\theta}$}

We consider the following notion of convergence of a sequence in $G$ to an element of $\F_\theta$. We say that for a sequence $g_i \in G$, $g_i \to \infty$ (or $g_i o \to \infty$) $\theta$-regularly if $\min_{\alpha \in \theta} \alpha(\mu(g_i)) \to \infty$ as $i \to \infty$.

\begin{definition} \label{conv} For a sequence $g_i \in G$  and $\xi\in \F_{\theta}$, we write $\lim_{i \to \infty} g_i =\lim_{i \to \infty} g_i o =\xi$ and
 say $g_i $ (or $g_i o \in X$) converges to $\xi$ if \begin{itemize}
     \item $g_i \to \infty$ $\theta$-regularly; and
\item $\lim_{i \to\infty} \kappa_{g_i }P_\theta= \xi$ in $\F_\theta$ for some $\kappa_{g_i}\in K$ such that $g_i \in \kappa_{g_i} A^+ K$.
 \end{itemize}         
\end{definition}

We recall the lemma that we will use in later sections.

\begin{lemma} \cite[Lemma 2.4]{kim2023growth} \label{lem.29inv}
    Consider a sequence $g_i = k_i a_i h_i^{-1}$ where $k_i \in K$, $a_i \in A^+$, and $h_i \in G$. Suppose that $k_i \to k_0\in K$, $h_i \to h_0\in G$, and $\min_{\alpha\in \theta} \alpha(\log a_i) \to \infty$, as $i\to \infty$. Then for any 
    $\xi$ in general position with $h_0P_\theta^+$,  we have $$\lim_{i \to \infty} g_i \xi = k_0 \xi_\theta.$$
\end{lemma}

Let $p, q \in X$ and $R > 0$.
The shadow of the ball $$B(q, R)=\{z\in X: d(z,q)<R\}$$ viewed from $p$ is defined as follows: $$O_R^{\theta}(p, q) := \{ gP_{\theta} \in \F_{\theta} : 
g A^+ o\cap B(q, r)\ne \emptyset  \}$$ where $g \in G$ satisfies $p = g o$. The shadow of $B(q,R)$
viewed from $\eta \in \F_{\i(\theta)}$ can also be defined: $$O_R^{\theta}(\eta, q) := \{ hP_{\theta} \in \F_{\theta} : h w_0 P_{\i(\theta)} = \eta, ho\in B(q,r) \}.$$

 We say that a sequence $g_i\in G$ (or $g_i o \in X$) converges to $\xi\in \F_\theta$
conically if $g_i \to \xi$ in the sense of Definition \ref{conv}
and there exists $R>0$ such that $\xi\in O_R^{\theta}(o, g_io)$ for all $i\ge 1$. 

 \begin{lemma}[{\cite[Lemma 5.35]{Kapovich2017anosov} (see also \cite[Lemma 9.8]{kim2023growth})}] \label{lem.klpconical} 
    Let $g_i \in G$ be a sequence such that $g_i \to \xi \in \F_{\theta}$. Then the following are equivalent:
    \begin{enumerate}
        \item The convergence $g_i \to \xi$ is conical.
        \item For any $\eta \in \F_{\i(\theta)}$ such that $(\xi, \eta) \in \F_{\theta}^{(2)}$, the sequence $g_i^{-1}(\xi, \eta)$ is precompact in $\F_{\theta}^{(2)}$.
        \item For some $\eta \in \F_{\i(\theta)}$ such that $(\xi, \eta) \in \F_{\theta}^{(2)}$, the sequence $g_i^{-1}(\xi, \eta)$ is precompact in $\F_{\theta}^{(2)}$.
     \end{enumerate}
\end{lemma}

The shadows vary continuously:
\begin{lemma} \cite[Lemma 3.3]{kim2023ergodic} \label{lem.approxshadowpi}
    Let $p \in X$, $q_i \in X$ a sequence converging to $ \eta \in \F_{\i(\theta)}$ and $r > 0$. Then for any $p \in X$ and $0 < \varepsilon < r$, we have for all large enough $i$ that $$O_{r-\varepsilon}^{\theta}(q_i, p) \subset O_r^{\theta}(\eta, p) \subset O_{r + \varepsilon}^{\theta}(q_i, p).$$
\end{lemma}

\subsection*{Limit set}

For a discrete subgroup $\Delta < G$, its limit set is defined as follows:

 \begin{Def}[Limit set] \label{def.limitset} 
 The limit set $\La_{\Delta}^{\theta}\subset \F_{\theta}$ is defined as the set of all accumulation points
 of $\Delta (o)$ in $\cal F_{\theta}$, that is,
 $$\La_{\Delta}^{\theta}=\{\lim_{i \to \infty} g_i o \in \F_{\theta}:
 g_i \in \Delta\} .$$
 \end{Def}
This may be an empty set for a general discrete subgroup. However,
we have the following result of Benoist for Zariski dense subgroups (\cite[Section 3.6]{Benoist1997proprietes}, \cite[Lemma 2.15]{lee2020invariant}):

\begin{theorem}\label{t2}
\label{dense}
If $\Delta<G$ is Zariski dense,
 then $\Lambda_\Delta^{\theta}$ is the unique $\Delta$-minimal subset of $\F_{\theta}$ and
the set of all attracting fixed points of loxodromic elements of $ \Delta$ is dense in
$\Lambda_\Delta^{\theta}$.
\end{theorem}

  \section{Busemann maps, conformal measures and essential subgroups} \label{sec.confess}
Let $G$ be a connected semisimple real algebraic group and fix a non-empty subset $\theta \subset \Pi$. We continue notations from Section \ref{sec.higherconf}. In this section, we introduce conformal measures and essential subgroups, and see how they are related.

\subsection*{Busemann maps} 

For $g \in G$ and $\xi = [k] \in K/M = \F$, the Iwasawa cocycle $H(g, \xi)$ is defined as the $\fa$-component of the Iwasawa decomposition of $gk$ so that $$gk \in K \exp (H(g, \xi)) N.$$ The higher rank Busemann maps are defined as follows:

\begin{Def}[Busemann map] \label{buse} 
The $\fa$-valued Busemann map $\beta: \F\times G \times G \to\mathfrak a $ is now defined as follows: for $\xi\in \F$ and $g, h \in G$,
 $$\beta_\xi ( g, h):=H (g^{-1}, \xi)- H(h^{-1}, \xi).$$
The $\fa_{\theta}$-valued Busemann map $\beta^{\theta} : \F_{\theta} \times G \times G \to \fa_{\theta}$ is defined as follows: for $\xi \in \F_{\theta}$ and $g, h \in G$, $$\beta_{\xi}^{\theta}(g, h) := p_{\theta}(\beta_{\tilde{\xi}}(g, h))$$ 
  where $\tilde{\xi} \in \pi_{\theta}^{-1}(\xi) \in \F$. This is well-defined \cite[Section 6]{Quint2002Mesures}.
\end{Def}
Observe that the Busemann map is continuous in all three variables. For $\xi \in \F$, $g \in G$ and $k \in K$, we have $H( (gk)^{-1}, \xi) = H(g^{-1}, \xi)$. Hence we can also define the Busemann map $\beta^{\theta} : \F_{\theta} \times X \times X \to \fa$ as $$\beta_{\xi}^{\theta}(g o, h o) := \beta_{\xi}^{\theta}(g, h) \quad \text{for } \xi \in \F_{\theta}, g, h \in G.$$

The following was proved in \cite{lee2020invariant} for $\theta = \Pi$. Since the $\fa_{\theta}$-valued Busemann map is defined as the $p_{\theta}$-image of the $\fa$-valued Busemann map, the same is true for general $\theta$:

\begin{lemma} \cite[Lemma 3.5]{lee2020invariant} \label{lem.buseisjordan}
    For a loxodromic $g \in G$, we have $$\beta_{y_g^{\theta}}^{\theta}(p, gp) = \la_{\theta}(g) \quad \text{for all } p \in X.$$
\end{lemma}

Busemann maps are compatible to Cartan projections in shadows:

\begin{lemma} \cite[Lemma 5.7]{lee2020invariant} \label{lem.shadow}
There exists $\kappa > 0$ such that for any $g, h \in G$ and $r > 0$, we have $$\sup_{\xi \in O_r^{\theta}(go, ho)} \| \beta_{\xi}^{\theta}(go, ho) - \mu_{\theta}(g^{-1}h) \| \le \kappa r.$$
\end{lemma}

\begin{corollary} \label{cor.busemannapprox}
Let $g_i \in G$ be a sequence such that $g_i o \to \eta \in \F_{\i(\theta)}$. For any $p \in X$ and $r, \varepsilon > 0$, we have $$\sup_{\xi, \xi' \in O_r^{\theta}(\eta, p)} \| \beta_{\xi}^{\theta}(g_i o, p) - \beta_{\xi'}^{\theta}(g_i o, p) \|  \le 2 \kappa (r + \varepsilon) \quad \text{for all large } i $$
where $\kappa$ is given by Lemma \ref{lem.shadow}.
\end{corollary}

\begin{proof}
    By Lemma \ref{lem.approxshadowpi}, we have $O_{r}^{\theta}(\eta, p) \subset O_{r + \varepsilon}^{\theta}(g_i o, p) $ for all large $i$. Letting $h \in G$ be such that $p = ho$, we have that for any $ \xi' \in O_{r}^{\theta}(\eta, p)$, both $\| \beta_{\xi}^{\theta}(g_i o, p) - \mu_{\theta}(g_i^{-1}h) \|$ and $\| \beta_{\xi'}^{\theta}(g_i o, p) - \mu_{\theta}(g_i^{-1}h) \|$ are bounded by $\kappa(r + \varepsilon)$ by Lemma \ref{lem.shadow}. Now the claim follows from the triangle inequality.
\end{proof}

\subsection*{Conformal measures} 
Let $\Delta < G$ be a discrete subgroup. The notion of higher rank conformal measures for $\Delta$ is defined in terms of $\fa_{\theta}$-valued Busemann maps and linear forms on $\fa_{\theta}$.

\begin{Def}[Conformal measures] \label{ps} 
A Borel probability measure $\nu_o$ on $\F_{\theta}$ is called a $\Delta$-conformal measure (with respect to $o$)
if there exists a linear form $\psi\in \fa_{\theta}^*$ such that for all $\eta\in \F_{\theta}$ and $g\in \Delta$,
$$
{ dg_* \nu_o \over d \nu_o}(\eta) = e^{\psi( \beta_{\eta}^{\theta}(o, g o))} .
$$
In this case, we say $\nu_o$ is a $(\Delta,\psi)$-conformal measure.
For $p \in X$, $d\nu_p (\eta)= e^{\psi( \beta_{\eta}^{\theta}(o, p))}d\nu_o (\eta) $
is a $(\Delta, \psi)$-conformal measure with respect to $p$.
\end{Def}

The set of values $\{\beta_{\eta}^{\theta}(o, go)\in \fa_{\theta}: g\in \Delta, \eta\in \supp \nu_o\} $ may not be large enough to distinguish $\Delta$-conformal measure classes by determining the linear form to which $\nu_o$ is associated; in general, there may be multiple linear forms associated to the same conformal measure class.

\begin{definition}[Divergence type]
    We say that a $(\Delta, \psi)$-conformal measure $\nu$ is of \emph{divergence type} if $\psi \in \fa_{\theta}^*$ is $(\Ga, \theta)$-proper and $\sum_{g \in \Delta} e^{-\psi(\mu_{\theta}(g))} = \infty$.
\end{definition}

\subsection*{Essential subgroups and Singularity of conformal measures}
The notion of essential subgroups was introduced by Schmidt \cite{Schmidt1977cocycles} (see also \cite{Roblin2003ergodicite}) in order to study the ergodic properties of horospherical actions. Its higher rank analogue, when $\theta = \Pi$, was studied in \cite{lee2020invariant} to show the ergodicity of horospherical foliations for Anosov subgroups with respect to a minimal parabolic subgroup.
We consider essential subgroups for general $\theta$, and also use them as tools to detect the singularity between two conformal measures.

\begin{definition}[Essential subgroup for $\nu$]\label{ess}
   For a $\Delta$-conformal measure $\nu$ on $\F_{\theta}$ with respect to $o$, we define
   the subset $\mathsf E_{\nu}^{\theta}(\Delta) \subset \fa_{\theta}$ as follows:
  $u\in \mathsf E_{\nu}^{\theta}(\Delta)$  if for any Borel subset $B \subset \F_{\theta}$ with $\nu(B) > 0$ and any $\varepsilon > 0$, there exists $g\in \Delta$ such that 
$$
\nu(B \cap g B \cap \{ \xi \in \F_{\theta} : \| \beta_{\xi}^{\theta}(o, g o) - u\| < \varepsilon \} ) > 0.
$$
It is easy to see that $\ess_{\nu}^{\theta}(\Delta)$ is a closed subgroup of $\fa_{\theta}$.
We call $\ess_{\nu}^{\theta}(\Delta)$ the essential subgroup for $\nu$.
\end{definition}

The following proposition is one of key ingredients of this paper. Although it was proved in \cite[Lemma 10.21]{lee2020invariant} (see also \cite[Proposition 3.6]{kim2023conformal}) for $\theta = \Pi$, the same proof works for general $\theta$:
\begin{prop} \label{prop.loabscont} For $i=1,2$, let $\nu_i$ be a
   $(\Delta, \psi_i)$-conformal measure on $\F_{\theta}$ for some $\psi_i \in \fa_{\theta}^*$.  If $\nu_2\ll \nu_1$,
   then 
   $$\psi_1(w)=\psi_2(w)\quad  \text{  for all $w\in \ess_{\nu_1}^{\theta}(\Delta)$}.$$
In particular, 
   if $\ess_{\nu_1}^{\theta}(\Delta) = \fa_{\theta}$,
   then $\nu_2 \ll \nu_1 $ implies
   $\psi_1=\psi_2. $
\end{prop}

\section{Graph-conformal measures for self-joinings}\label{set} \label{sec.tangent}

In this section, we review the notion of self-joinings and graph-conformal measures, introduced by Kim-Oh (\cite{kim2022rigidity}, \cite{kim2022rigidity2}, \cite{kim2023conformal}), which play key roles in studying rigidity problems.
For $i = 1, 2$, let $G_i$ be a connected semisimple real algebraic group with the associated Riemannian symmetric space $(X_i, d_i)$, and write $\fg_i := \Lie G_i$. 
Let $(X, d)$ be the Riemannian product $(X_1\times X_2, \sqrt{d_1^2 + d_2^2})$. 
Set $$G=G_1\times G_2$$ so that its Lie algebra is $\fg := \fg_1 \oplus \fg_2$.
Then $G$ acts by isometries on $X$.
 For $i = 1, 2$, we use the same notations for $G_i$ as we did for $G$ but with a subscript $i$. For $\square\in \{ A, M, N, P,K, o\}$, we consider the corresponding objects for $G$ by setting $$\square= \square_1\times \square_2.$$
In particular, $A=A_1\times A_2$.
Let
$A^+=A_1^+ \times A_2^+$. Let $\fa$ denote the Lie algebra of $A$, and $\fa^+=\log A^+$. We note that $$\fa = \fa_1 \oplus \fa_2 \quad \mbox{and} \quad \fa^+ = \fa_1^+ \oplus \fa_2^+,$$
where $\fa_i = \Lie A_i$ and $\fa_i^+ = \Lie A_i^+$ for $i = 1, 2$.

For each $i = 1, 2$, let $\Pi_i$ be the set of all simple roots for $(\fg_i, \fa_i^+)$ and fix a non-empty subset $\theta_i \subset \Pi_i$ in the rest of the paper. We set $\Pi := \Pi_1 \cup \Pi_2$ which is the set of all simple roots for $(\fg, \fa^+)$ and $$\theta := \theta_1 \cup \theta_2.$$ Then we have $$\fa_{\theta} = \fa_{\theta_1} \oplus \fa_{\theta_2}, \quad P_{\theta} = P_{\theta_1} \times P_{\theta_2} \quad \text{and} \quad \F_{\theta} = \F_{\theta_1} \times \F_{\theta_2}.$$

Let $\Ga < G_1$ be a Zariski dense discrete subgroup with the limit set $\La^{\theta_1} \subset \F_{\theta_1}$.
Let $\rho:\Ga \to G_2$ be a discrete faithful Zariski dense representation. 

\begin{Def}[Self-joining] We define the self-joining of $\Ga$ via $\rho$ as  
$$\Ga_\rho:=(\id\times \rho)(\Ga) = \{ (g, \rho(g))\in G : g \in \Ga \},$$
which is a discrete subgroup of $G$.
\end{Def}

One key feature of a self-joining subgroup is that the rigidity question on $\rho$ can be translated to a Zariski density question on the self-joining.

\begin{lem} \cite{DalboKim_criterion} \label{Zdense}
Suppose that $G_1$ and $G_2$ are simple. Then the self-joining $\Ga_\rho < G$ is not Zariski dense  if and only if $\rho$ extends to a
Lie group isomorphism $G_1\to G_2$.
\end{lem}

\subsection*{Boundary map}

In the rest of this section, we assume that there exists a $\rho$-equivariant continuous map
$$f :\La^{\theta_1} \to \F_{\theta_2}.$$
We will not assume that $f$ is injective,
mentioned otherwise. When it is injective, we call it a $\rho$-boundary map.

A $\rho$-boundary map is unique when it exists.
First observe that, since  $\La^{\theta_1}$ (resp. $\La_{\rho(\Ga)}^{\theta_2}$) is 
the unique $\Ga$ (resp. $\rho(\Ga)$) minimal subset of $\F_{\theta_1}$ (resp. $\F_{\theta_2}$) (Theorem \ref{dense}), it follows from the equivariance of $f$ that
$$ f (\La^{\theta_1})=\La_{\rho(\Ga)}^{\theta_2}.$$
The uniqueness of a boundary map was proved in \cite[Lemma 4.5]{kim2023conformal} for $\theta = \Pi$ and the same proof works for general $\theta$.

\begin{lem}[Uniqueness] \label{ll}
If $ g \in \Ga$ and $\rho(g)$ are both loxodromic, then
  $$f (y_{g}^{\theta_1})=y_{\rho(g)}^{\theta_2}.$$ 
  In particular, when $G_1$ and $G_2$ are simple, $f$ is the unique $\rho$-equivariant continuous map $\La^{\theta_1}\to \F_{\theta_2}$.
\end{lem} 

In terms of the boundary map, the limit set of the self-joining $\Ga_\rho$ in $\F_{\theta}$ is as follows:
$$\La_\rho^{\theta}: =(\id \times f)(\Lambda^{\theta_1}) =\{(\xi, f (\xi)) \in  \F_{\theta} :\xi\in \La^{\theta_1} \}.$$

\subsection*{Graph-conformal measures} 

Recall that for each $i=1,2$,
$$
p_{\theta_i} : \fa \to \fa_{\theta_i}
$$ is the projection invariant under all Weyl elements fixing $\fa_{\theta_i}$ pointwise.
Restricting on $\fa_{\theta}$, we may regard $p_{\theta_i}$ as the projection from $\fa_{\theta}$ as well.
Abusing notations, we also denote the restriction $\fa_{\theta} \to \fa_{\theta_i}$ by $p_{\theta_i}$ for $i = 1, 2$.
For a linear form $\psi_i$ on $\fa_{\theta_i}^*$, we define
a linear form $\sigma_{ \psi_i}\in \fa_{\theta}^*$ by 
$$\sigma_{\psi_i}:=\psi_i\circ p_{\theta_i};$$
in other words, $\sigma_{\psi_i}(u_1,u_2)= \psi_i(u_i)$ for all $(u_1, u_2)\in \fa_{\theta_1}\oplus \fa_{\theta_2}$. 
Recall that $o_i = [K_i] \in X_i$ for $i = 1, 2$ and $o = (o_1, o_2) \in X$.

The following is a key observation on the relation between $\Ga$-conformal measures and $\Gr$-conformal measures. It was proved in \cite[Proposition 4.6, Corollary 4.7, Lemma 4.10]{kim2023conformal} for $\theta = \Pi$; the same proof works for general $\theta$:

\begin{prop} \label{lem.pushforward} \label{cor.uniqueconformal} \label{s2} Let $\psi_i \in \fa_{\theta_i}^*$ for $i = 1, 2$.
\begin{enumerate}
    \item If $\nu_{\psi_1}$ is a $(\Ga,\psi_1)$-conformal measure on $\La^{\theta_1}$ with respect to $o_1$,
then $(\id \times f)_*\nu_{\psi_1} $ is a $(\Ga_\rho, \sigma_{\psi_1})$-conformal measure on $\La_{\rho}^{\theta}$ with respect to $o$.
    \item Any  $(\Ga_\rho, \sigma_{\psi_1})$-conformal measure on $\La_\rho^{\theta}$ with respect to $o$ is of the form $(\id \times f)_*\nu_{\psi_1}$ 
for some  $(\Ga, \psi_1)$-conformal measure $\nu_{\psi_1}$ on $\La^{\theta_1}$.
    \item If $\nu_{\psi_1}$ is the unique $(\Ga, \psi_1)$-conformal measure on $\La^{\theta_1}$, then $(\id \times f )_*\nu_{\psi_1}$ is the unique $(\Ga_\rho, \sigma_{\psi_1})$-conformal measure on $\La_\rho^{\theta}$
    with respect to $o$; in particular,
    $(\id \times f)_*\nu_{\psi_1}$ is $\Ga_\rho$-ergodic.

    \item Let  $\nu_{\psi_1}$ and
    $\nu_{\psi_2}$ be $\Ga$-conformal and $\rho(\Ga)$-conformal measures on $\La^{\theta_1}$ and $\La_{\rho(\Ga)}^{\theta_2}$ respectively. If $f$ is injective, then $(f^{-1}\times \id )_*\nu_{\psi_2} $ is a $(\Ga_\rho, \sigma_{\psi_2})$-conformal measure, and we have
   $$\text{$(f^{-1}\times \id)_*\nu_{\psi_2}\ll(\id\times  f)_*\nu_{\psi_1}$ if and only if $\nu_{\psi_2}\ll f_*\nu_{\psi_1} $}.$$

\end{enumerate}
\end{prop}

The notion of graph-conformal measure was first introduced in our earlier work with Oh \cite{kim2023conformal}:

\begin{Def}[Graph-conformal measures] By
a graph-conformal measure of $\Ga_\rho$, we mean a (conformal) measure
of the form $$\nu_{\rho} := (\id \times f)_*\nu$$ for some $\Ga$-conformal measure
$\nu$ on $\La^{\theta_1}$.
\end{Def}

Note that we used the notation $\nu_{\mathrm{graph}}$ for the graph-conformal measure in \cite{kim2023conformal}. Using this terminology, Proposition \ref{lem.pushforward}(1)-(2) can be reformulated as follows:
\begin{prop}\label{gc}
 Let $\sigma\in \fa_{\theta}^*$ be a linear form which factors through $\fa_{\theta_1}$.
A
$(\Ga_\rho, \sigma)$-conformal measure on $\La_\rho^{\theta}$  is
a graph-conformal measure of $\Ga_\rho$, and conversely, any graph-conformal measure of $\Ga_\rho$ is of such a form.
    \end{prop}

\section{Myrberg limit sets}\label{sec.Myrberg}

Let $G$ be a connected semisimple real algebraic group and $X = G/K$ the associated symmetric space. Recall the choice of the basepoint $o = [K] \in X$.
For a discrete subgroup $\Delta < G$, the Myrberg limit set is defined as follows:

\begin{definition}[Myrberg limit set]
    We say that $\xi \in \La_{\Delta}^{\theta}$ is a Myrberg limit point of $\Delta$ if for any $\xi_0 \in \La_{\Delta}^{\theta}$ and $\eta_0 \in \La_{\Delta}^{\i(\theta)}$ in general position, there exists a sequence $g_i \in \Delta$ such that $g_i \xi \to \xi_0$ and $g_i o \to \eta_0$ as $i \to \infty$.

    We call the set of all Myrberg limit points of $\Delta$ the \emph{Myrberg limit set} of $\Delta$ and denote by $\La_{\Delta, M}^{\theta} \subset \F_{\theta}$.
\end{definition}

In this section, we show that the Myrberg limit sets of transverse subgroups and their self-joinings have full measures with respect to conformal measures of divergence type and their associated graph-conformal measures.

\begin{definition}
    
A discrete subgroup $\Delta < G$ is called \emph{$\theta$-transverse} if
\begin{itemize}
    \item   it is \emph{$\theta$-regular}, i.e.,
$ \liminf_{g\in \Delta} \alpha(\mu({g}))=\infty $ for all $\alpha\in \theta$; and
\item it is  \emph{$\theta$-antipodal}, i.e., any two distinct $\xi, \eta\in \La^{\theta\cup \i(\theta)}$ 
are in general position.
\end{itemize}

\end{definition}

Since $\mu(g^{-1}) = \i(\mu(g))$ for all $g \in G$, the $\theta$-regularity is equivalent to the $\i(\theta)$-regularity, and hence $\Delta$ is $\theta$-regular if and only if it is $\theta \cup \i(\theta)$-regular. Moreover, by \cite[Lemma 9.5]{kim2023growth}, the $\theta$-antipodality implies that the canonical projections $\La^{\theta \cup \i(\theta)}_{\Delta} \to \La_{\Delta}^{\theta}$ and $\La^{\theta \cup \i(\theta)}_{\Delta} \to \La_{\Delta}^{\i(\theta)}$ are $\Delta$-equivariant homeomorphisms.

\subsection*{Myrberg limit sets of transverse subgroups and their self-joinings}
For $i = 1, 2$, let $G_i$ be a connected semisimple real algebraic group and $X_i$ the associated symmetric space and fix a non-empty subset $\theta_i \subset \Pi_i$. Let $$\Ga < G_1$$ be a Zariski dense $\theta_1$-transverse subgroup with limit set $$\La^{\theta_1} := \La^{\theta_1}_{\Ga} \subset \F_{\theta_1}.$$ Let $$\rho : \Ga \to G_2$$ be a Zariski dense $\theta_2$-regular faithful representation with a pair of $\rho$-equivariant continuous maps $f: \La^{\theta_1} \to \La^{\theta_2}_{\rho(\Ga)}$ and $f_{\i} : \La^{\i(\theta_1)} \to \La^{\i(\theta_2)}_{\rho(\Ga)}$.

Set 
$$G = G_1 \times G_2, \quad X = X_1 \times X_2, \quad \text{and} \quad \theta = \theta_1 \cup \theta_2.$$ We keep the notations introduced in Section \ref{set}.  We denote by $$\La_{\rho}^{\theta} \subset \F_{\theta}$$ the limit set of the self-joining $\Ga_{\rho}$. 
We simply write $$\La_{M}^{\theta_1} := \La_{\Ga, M}^{\theta_1} \subset \F_{\theta_1} \quad \mbox{and} \quad \La_{\rho, M}^{\theta} := \La_{\Gr, M}^{\theta} \subset \F_{\theta}$$ for the Myrberg limit sets of $\Ga$ and $\Ga_{\rho}$ respectively. 
The main goal of this section is to prove the following:
\begin{theorem} \label{thm.fullmyrberg}
    Let $\nu$ be a $\Ga$-conformal measure on $\F_{\theta_1}$ of divergence type, and $\nu_{\rho} = (\id \times f)_* \nu$ the associated graph-conformal measure for $\Ga_{\rho}$ on $\F_{\theta}$. Then $$\nu(\La^{\theta_1}_{M}) = 1 \quad \text{and} \quad \nu_{\rho} (\La_{\rho, M}^{\theta}) = 1.$$
\end{theorem}

\begin{remark}
The claim $\nu( \La_M^{\theta_1}) = 1$ was proved by Tukia \cite[Theorem 4A]{Tukia1994poincare} when $\rank G_1 = 1$, and by Lee-Oh \cite{lee2020invariant} when $\Ga$ is $\Pi$-Anosov.
\end{remark}

\subsection*{Ergodic properties of divergence-type conformal measures}
To prove Theorem \ref{thm.fullmyrberg}, we use the ergodic properties of conformal measures of divergence type.
Recall that a linear form $\psi \in \fa_{\theta_1}^*$ is called $(\Ga, \theta_1)$-proper if the map $ \psi \circ \mu_{\theta_1} |_{\Ga} : \Ga \to [-\varepsilon, \infty)$ is proper for some $\varepsilon > 0$. 
For a $(\Ga, \theta_1)$-proper $\psi \in \fa_{\theta_1}^*$, the abscissa of convergence $\delta_{\psi} \in (0, \infty]$ of the $\psi$-Poincar\'e series $\sum_{g \in \Ga} e^{-s\psi(\mu_{\theta_1}(g))}$ is well-defined \cite[Lemma 4.2]{kim2023growth}. Moreover, if there exists a $(\Ga, \psi)$-conformal measure, then $\delta_{\psi} \le 1$ \cite[Theorem 1.5]{kim2023growth}. Whether or not the $\psi$-Poincar\'e series converges or diverges at $s = 1$ plays a crucial role:

\begin{theorem} \cite[Theorem 1.5]{kim2023growth} \label{thm.uniquepsdiv}
Let $\psi \in \fa_{\theta_1}^*$ be a $(\Ga, \theta_1)$-proper form.
If $\delta_{\psi} \le 1$ and $\sum_{g \in \Ga} e^{-\psi(\mu_{\theta_1}(g))} = \infty$, then there exists a unique $(\Ga, \psi)$-conformal measure on $\F_{\theta_1}$. Moreover, the unique $(\Ga, \psi)$-conformal measure has the support $\La^{\theta_1}$.
 In particular, any $\Ga$-conformal measure of divergence type is $\Ga$-ergodic.
\end{theorem} 
Under the additional assumption on the support, the above statement was also proved in \cite{canary2023patterson}.

Recall that $\F_{\theta_1}^{(2)} \subset \F_{\theta_1} \times \F_{\i(\theta_1)}$  is the set of all pairs in general position. Fixing a $(\Ga,\theta_1)$-proper form $\varphi\in \fa_{\theta_1}^*$,
we consider the $G_1$-action 
\be\label{hopf3} g \cdot (\xi, \eta, s) = (g \xi, g \eta, s + \varphi(\beta_{\xi}^{\theta_1}(g^{-1}, e))) \ee 
for all $g\in G_1$ and $(\xi, \eta, s)\in \F_{\theta_1}^{(2)} \times \br$. 
Setting $$\La^{(2)} := (\La^{\theta_1} \times \La^{\i(\theta_1)}) \cap \F_{\theta_1}^{(2)},$$
the subspace $\La^{(2)} \times \R \subset \F_{\theta_1}^{(2)} \times \R$ is invariant under the $\Ga$-action.

\begin{theorem} \cite[Theorem 9.2]{kim2023growth} \label{tproper2}
    The action $\Ga $ on $\La^{(2)} \times \br$ given by \eqref{hopf3} is properly discontinuous and hence $$\Omega_{\varphi}:=\Ga \ba \La^{(2)} \times \br$$ is a locally compact second-countable Hausdorff space.
\end{theorem}

Let $\psi \in \fa_{\theta_1}^*$. For a pair of a $(\Ga, \psi)$-conformal measure $\nu$ on $\La^{\theta_1}$ and  a $(\Ga, \psi \circ \i)$-conformal measure 
$\nu_{\i}$ on $\La^{\i(\theta_1)}$, we  define a Radon measure $\tilde{\mathsf m}^{\varphi}_{\nu, \nu_{\i}}$ on $\La^{(2)} \times \R$ as follows: 
$$
d\tilde{\mathsf m}^{\varphi}_{\nu, \nu_{\i}} (\xi, \eta, t) = e^{\psi \left( \beta_{\xi}^{\theta_1}(e, g) + \i(\beta_{\eta}^{\i(\theta_1)}(e, g))\right)} d\nu(\xi) d \nu_{ \i}(\eta) dt
$$
where $g \in G_1$ is such that $(\xi, \eta) = (gP_{\theta_1}, g w_0 P_{\i(\theta_1)})$ and $dt$ is the Lebesgue measure on $\R$. This is well-defined \cite[Lemma 9.13]{kim2023growth}.
The measure $d\tilde{\mathsf m}^{\varphi}_{\nu, \nu_{\i}} $
is left $\Ga$-invariant and invariant under the translation on $\R$.
Hence it descends to the $\R$-invariant Radon measure $$\mathsf m_{\nu, \nu_{\i}}^{\varphi}$$ on $\Omega_{\varphi}$, which we call the Bowen-Margulis-Sullivan measure associated to the pair $(\nu, \nu_{\i})$.

Let $\cal M_{\psi}^{\theta_1}$ (resp. $\cal M_{\psi \circ \i}^{\i(\theta_1)}$) be the set of all $(\Ga, \psi)$-conformal measures on $\F_{\theta_1}$ (resp. $(\Ga, \psi \circ \i)$-conformal measures on $\F_{\i(\theta_1)}$). The following is the Hopf-Tsuji-Sullivan dichotomy for transverse subgroups: we also denote by $\La^{\theta_1}_{\mathsf{c}} \subset \F_{\theta_1}$ the conical set of $\Ga$.

\begin{theorem}  [{\cite[Theorem 10.2]{kim2023growth}, see also \cite{canary2023patterson} and \cite{Sullivan1979density}}] \label{ceq}
Let $\Ga < G_1$ be a Zariski dense $\theta_1$-transverse subgroup.
   The following are equivalent to each other.
\begin{enumerate}

\item $\sum_{g \in \Ga}e^{-\psi(\mu_{\theta_1}(g))}=\infty$
(resp. $\sum_{g \in \Ga}e^{-\psi(\mu_{\theta_1}(g))}<\infty$);

\item $\nu(\La^{\theta_1}_{\mathsf{c}}) = 1$ (resp. $\nu(\La^{\theta_1}_{\mathsf{c}}) = 0$);

\item For any $(\nu, \nu_{\i}) \in \mathcal{M}_{\psi}^{\theta_1} \times \mathcal{M}_{\psi \circ \i}^{\i(\theta_1)}$ and any $(\Ga, \theta_1)$-proper $\varphi \in \fa_{\theta_1}^*$, the $\br$-action on $(\Omega_\varphi, {\mathsf m}^\varphi_{\nu,\nu_{\i}})$  is completely conservative and ergodic (resp. completely dissipative and non-ergodic).
\end{enumerate}
 \end{theorem}

Indeed, Theorem \ref{thm.fullmyrberg} adds one more item $\nu(\La_M^{\theta_1}) = 1$ (resp. $\nu(\La_M^{\theta_1}) = 0$) to the above dichotomy, since $\La_{M}^{\theta_1} \subset \La_{\mathsf{c}}^{\theta_1}$ (Lemma \ref{lem.klpconical}).

\subsection*{Proof of Theorem \ref{thm.fullmyrberg}} Note that we can regard $\psi$ as a linear form on $\fa_{\theta_1 \cup \i(\theta_1)}$ by precomposing with the projection $\fa_{\theta_1 \cup \i(\theta_1)} \to \fa_{\theta_1}$, which is $(\Ga, \theta_1 \cup \i(\theta_1))$-proper since $\psi(\mu_{\theta_1}(g)) = \psi(\mu_{\theta_1 \cup \i(\theta_1)}(g))$ for all $g \in \Ga$. By Theorem \ref{ceq}, $\nu$ is supported on the conical set of $\Ga$, in particular, on $\La^{\theta_1}$. The canonical projection $\La^{\theta_1 \cup \i(\theta_1)} \to \La^{\theta_1}$ is a $\Ga$-equivariant homeomorphism \cite[Lemma 9.5]{kim2023growth} and hence we can pull-back $\nu$ to $\La^{\theta_1 \cup \i(\theta_1)}$ so that $\nu$ can be considered as a $(\Ga, \psi)$-conformal measure on $\La^{\theta_1 \cup \i(\theta_1)}$. Since $\theta_1$-transverse subgroups are $\theta_1 \cup \i(\theta_1)$-transverse, we may assume without loss of generality that $\theta_1 = \i(\theta)$ by replacing $\theta_1$ with $\theta_1 \cup \i(\theta_1)$. 

\medskip
{\bf \noindent Myrberg limit set of $\Ga$ is $\nu$-full.} 
By \cite[Theorem 1.5]{kim2023growth}, it follows from the existence of $\nu$ that $\delta_{\psi} \le 1$. Since $\sum_{g \in \Ga} e^{-\psi(\mu_{\theta_1}(g))} = \infty$, we have $\delta_{\psi} = 1$. Hence $\nu$ is the unique $(\Ga, \psi)$-conformal measure on $\F_{\theta_1}$ by Theorem \ref{thm.uniquepsdiv}, and is supported on $\La^{\theta_1}$ as mentioned above. Moreover, since $\mu_{\theta_1}(g^{-1}) = \i(\mu_{\theta_1}(g^{-1}))$, $\psi \circ \i$ is also $(\Ga, \theta_1)$-proper, $\delta_{\psi \circ \i} = 1$ and $\sum_{g \in \Ga} e^{-(\psi \circ \i)(\mu_{\theta_1}(g))} = \infty$. Hence by Theorem \ref{thm.uniquepsdiv} and Theorem \ref{ceq}, there exists a unique $(\Ga, \psi \circ \i)$-conformal measure $\nu_{\i}$ on $\F_{\theta_1}$ and is supported on $\La^{\theta_1}$ as well.

Now we are able to consider the measure space $(\Omega_\varphi, \mathsf{m}_{\nu, \nu_{\i}}^{\varphi})$ by fixing a $(\Ga, \theta_1)$-proper $\varphi \in \fa_{\theta_1}^*$. By Theorem \ref{ceq}, the $\R$-action on $(\Omega_\varphi, \mathsf{m}_{\nu, \nu_{\i}}^{\varphi})$ is completely conservative and ergodic, and hence $\mathsf{m}_{\nu, \nu_{\i}}^{\varphi}$-a.e. $\R_{+}$-orbit is dense.

In other words, for $\nu \otimes \nu_{\i} \otimes dt$-a.e. $(\xi, \eta, t) \in \La^{(2)} \times \R$, its $\R_+$-orbit is dense in $\Omega_{\varphi}$. Fix one such element $(\xi, \eta, t) \in \La^{(2)} \times \R$. Hence for any $(\xi_0, \eta_0) \in \La^{(2)}$, there exist sequences $g_i \in \Ga$ and $t_i \to + \infty$ such that $$g_i (\xi, \eta, t + t_i) \to (\xi_0, \eta_0, 0) \quad \text{as } i \to \infty.$$ In particular, we have 
\be \label{eqn.busemanninfty}
g_i (\xi, \eta) \to (\xi_0, \eta_0) \quad \text{and} \quad \varphi(\beta_{\xi}^{\theta_1}(g_i^{-1}, e)) \to - \infty.
\ee

Since the action of $\Ga$ on $\La^{\theta_1}$ is a convergence group action \cite[Theorem 4.16]{Kapovich2017anosov}, after passing to a subsequence, there exist $a, b \in \La^{\theta_1}$ such that as $i \to \infty$, $$
g_i|_{\La^{\theta_1} - \{b\}} \to a \quad \text{uniformly on compact subsets.}$$
That is, for any compact subsets $C_a \subset \La^{\theta_1} - \{a\}$ and $C_b \subset \La^{\theta_1} - \{b\}$, 
    $$\# \{ g_i : g_i C_b \cap C_a \neq \emptyset\} < \infty,$$ 
   or equivalently $\# \{ g_i^{-1} : g_i^{-1} C_a \cap C_b \neq \emptyset\} < \infty$. Therefore we have, as $i\to \infty$,  $$
   g_i^{-1} |_{\La^{\theta_1} - \{a\}} \to b \quad \text{uniformly on compact subsets.}
   $$
   Since $g_i(\xi, \eta) \to (\xi_0, \eta_0)$, we have either 
   \be \label{eqn.resultofconvaction}
   (a, b) = (\xi_0, \eta) \quad \text{or} \quad (a, b) = (\eta_0, \xi).
   \ee
   
   By the $\theta_1$-regularity of $\Ga$, we may assume by passing to a subsequence that the sequence $g_i o_1$ (resp. $g_i^{-1} o_1$) converges to some point, say $z \in \La^{\theta_1}$ (resp. $z' \in \La^{\theta_1}$). We claim that $z = a$ and $z' = b$. Write $g_i = k_i b_i \ell_i^{-1} \in KA^+K$ using the Cartan decomposition. By passing to a subsequence, we may assume that $k_i \to k_0 \in K$ and $\ell_i \to \ell_0 \in K$. Choose $x \in \La^{\theta_1} - \{\xi, \eta, \xi_0, \eta_0\}$ which is in general position with $\ell_0 w_0 P_{\theta_1}$ and $k_0 P_{\theta_1}$; this is possible by the Zariski density of $\Ga$. Since $\Ga$ is $\theta_1$-regular, we have $\min_{\alpha \in \theta_1} \alpha(\log b_i) \to \infty$. Hence, by Lemma \ref{lem.29inv}, we have $$g_i x \to k_0 P_{\theta_1} = z.$$ Since $x \neq b$, we must have $z = a$.

   Similarly, the Cartan decomposition $g_i^{-1} = (\ell_i w_0)(w_0^{-1} b_i^{-1} w_0)(w_0^{-1} k_i^{-1}) \in KA^+ K$ and the $\theta_1$-regularity of $\Ga$ imply $\min_{\alpha \in \theta_1} \alpha(\log (w_0^{-1} b_i^{-1} w_0)) \to \infty$. Hence it follows from Lemma \ref{lem.29inv} that $$g_i^{-1} x \to \ell_0 w_0 P_{\theta_1} = z'.$$ Since $x \neq a$, we must have $z' = b$, which shows the claim.

   Therefore, it suffices to show that $(a, b) = (\eta_0, \xi)$ since we already know that $g_i \xi \to \xi_0$ and $g_i o_1 \to a$. Suppose not. Then by \eqref{eqn.resultofconvaction}, $(a, b) = (\xi_0, \eta)$, and hence $g_i^{-1} o_1 \to \eta$. Since $g_i (\xi, \eta) \to (\xi_0, \eta_0)$ in $\La^{(2)}$, we have $g_i^{-1} o_1 \to \eta$ conically by Lemma \ref{lem.klpconical}.
   
Choose $g\in G_1$ so that $\xi=gP_{\theta_1}$ and $\eta= gw_0P_{\theta_1}$, noting that we are assuming that $\theta_1 = \i(\theta_1)$. 
That $g_i^{-1}$ conically converges to $\eta$ means that there exist a sequence $k_i \in K$ and a sequence $a_i \to \infty$ in $A^+$ such that $\eta = k_i P_{\theta_1}$ for all $i$ and the sequence $g_i k_i a_i$ is bounded. Since $\eta = g w_0 P_{\theta_1} = k_i P_{\theta_1}$, we have for each $i$, $g w_0 m_i' p_i = k_i$ for some $m_i' \in M_{\theta_1}$ and $p_i \in P$, using $P_{\theta_1} = M_{\theta_1} P$. Since both $k_i$ and $m_i'$ are bounded sequences, the sequence $p_i \in P$ is bounded as well. It implies that the sequence $a_i^{-1} p_i a_i$ is bounded since $a_i \in A^+$. Hence it follows from the boundedness of the sequence  $g_i k_i a_i = g_i g w_0 m_i' p_i a_i = g_i g w_0 m_i' a_i (a_i^{-1} p_i a_i)$ that  $$\text{the sequence } h_i := g_i g w_0 m_i' a_i \text{ is bounded.}$$

For each $i$, set $m_i=w_0m_i'w_0^{-1}\in M_{\theta_1} $.
Then $$\eta = gw_0 P_{\theta_1} = g w_0 m_i' P_{\theta_1} = g m_i w_0 P_{\theta_1}, \quad \xi = gP_{\theta_1} = gm_iP_{\theta_1}$$ and
$$h_i = g_i g w_0 m_i' a_i = g_i g m_i w_0 a_i.$$

Using $\xi=gm_i P_{\theta_1}$, we have
$$\begin{aligned}
    \beta_{\xi}^{\theta_1}(g_i^{-1}, e) & = \beta_{g_i \xi}^{\theta_1}(e, g_i) = \beta_{g_i \xi}^{\theta_1}(e, h_i) + \beta_{g_i \xi}^{\theta_1}(h_i, g_i) \\
    & = \beta_{g_i \xi}^{\theta_1}(e, h_i) + \beta_{\xi}^{\theta_1}(g m_i w_0 a_i, e)
    \\ &= \beta_{g_i \xi}^{\theta_1}(e, h_i) + \beta_{P_{\theta_1}}^{\theta_1}( w_0 a_i, e) 
   + \beta_{P_{\theta_1}}^{\theta_1}(e, m_i^{-1} g^{-1}).
\end{aligned} $$

Since $h_i$ is a bounded sequence,
the sequence $\beta_{g_i \xi}^{\theta_1}(e, h_i)$ is bounded by \cite[Lemma 5.1]{lee2020invariant}. Similarly, $\beta_{P_{\theta_1}}^{\theta_1}(e, m_i^{-1} g^{-1})$ is bounded. Hence it suffices to show that  as $i\to \infty$, 
\be\label{betaf} 
\varphi(\beta_{P_{\theta_1}}^{\theta_1}( w_0 a_i, e)) \to \infty,
\ee  
which yields a contradiction to \eqref{eqn.busemanninfty}.
Note that $$\beta_{P}( w_0 a_i, e)=
\beta_{P}( w_0 a_iw_0^{-1}, e) =\i (\log a_i).$$
Since $h_i = g_i g m_i w_0 a_i$ is bounded and $g_i^{-1} h_i= g m_i w_0 a_i $, we have 
$\|\mu(g_i^{-1}) - \log a_i\| = \|\mu(g_i) - \i(\log a_i)\|$ is uniformly bounded by Lemma \ref{lem.cptcartan} and the identity \eqref{mu}.
Therefore
$$\sup_i |\varphi(\mu_{\theta_1}(g_i) - p_{\theta_1}(\i(\log a_i)))| <\infty .$$

By the $\theta_1$-regularity of $\Gamma$ and the $(\Ga, \theta_1)$-properness of $\varphi$, $\varphi(\mu_{\theta_1}(g_i)) \to \infty$ and hence
$\varphi(p_{\theta_1}(\i(\log a_i)))\to \infty$ as $i \to \infty$, which implies \eqref{betaf}. Hence
 $\varphi(\beta_{\xi}^{\theta_1}(g_i^{-1}, e))\to \infty$, yielding the desired contradiction to \eqref{eqn.busemanninfty}.
This shows that $(a, b) = (\eta_0, \xi)$.

Consequently, we have $\xi \in \La^{\theta_1}_{M}$. Since this holds for $\nu \otimes \nu_{\i} \otimes dt$-a.e. $(\xi, \eta, t) \in \La^{(2)} \times \R$, we have $\nu(\La^{\theta_1}_{M}) = 1$.

\medskip
{\noindent \bf Myrberg limit set of $\Gr$ is $\nu_{\rho}$-full.}
Note that we have $\rho$-equivariant continuous maps $f : \La^{\theta_1} \to \La^{\theta_2}_{\rho(\Ga)}$ and $f_{\i} : \La^{\i(\theta_1)} \to \La^{\i(\theta_2)}_{\rho(\Ga)}$. As in the previous argument, take $(\xi, \eta, t) \in \La^{(2)} \times \R$ with a dense $\R_+$-orbit in $\Omega_{\varphi}$. Setting $\La_{\rho}^{(2)} := (\La_{\rho}^{\theta} \times \La_{\rho}^{\i(\theta)}) \cap \F_{\theta}^{(2)}$, every element of $\La_{\rho}^{(2)}$ is a pair of $(\xi_0, f (\xi_0)) \in \La_{\rho}^{\theta}$ and $(\eta_0, f_{\i}(\eta_0)) \in \La_{\rho}^{\i(\theta)}$ in general position for some $\xi_0, \eta_0 \in \La^{\theta_1}$. Take such elements $\xi_0, \eta_0 \in \La^{\theta_1}$; then $(\xi_0, \eta_0) \in \La^{(2)}$ and hence we again have sequences $g_i \in \Ga$ and $t_i \to \infty$ such that $g_i (\xi, \eta, t + t_i) \to (\xi_0, \eta_0, 0)$ as $i \to \infty$ as above. As we have shown, it follows that $$\quad g_i|_{\La^{\theta_1} - \{ \xi \}} \to \eta_0 \quad \text{and} \quad g_i o_1 \to \eta_0 \quad \text{as } i \to \infty.$$

For each $i$, we denote by $\ga_i = (g_i, \rho(g_i)) \in \Ga_{\rho}$. As $f$ and $f_{\i}$ are $\rho$-equivariant continuous maps, we have $$\ga_i (\xi, f (\xi)) \to (\xi_0, f(\xi_0)) \quad  \text{and} \quad \rho(g_i)|_{\La^{\i(\theta_2)}_{\rho(\Ga)} - \{ f_{\i}(\xi)\}} \to f_{\i}(\eta_0).$$
Hence it remains to show that $\rho(g_i) o_2 \to f_{\i}(\eta_0)$ as this, together with $g_i o_1 \to \eta_0$, implies $\ga_i o \to (\eta_0, f_{\i}(\eta_0))$.

Since $\rho(\Ga)$ is $\theta_2$-regular, it is $\i(\theta_2)$-regular.
Write the Cartan decomposition $\rho(g_i) = \tilde{k}_i \tilde{a}_i \tilde{\ell}_i^{-1} \in KA^+K$. By passing to a subsequence, we may assume that $\tilde{k}_i \to \tilde{k}_0$ and $\tilde{\ell}_i \to \tilde{\ell}_0$. Since $\rho(\Ga)$ is Zariski dense, we can choose $\tilde{x} \in \La^{\i(\theta_2)}_{\rho(\Ga)} - \{ f_{\i}(\xi)\}$ which is in general position with $\tilde{\ell}_0 w_0 P_{\i(\theta_2)}$. The $\i(\theta_2)$-regularity of $\rho(\Ga)$ implies that $\min_{\alpha \in \i(\theta_2)} \alpha(\log \tilde{a}_i) \to \infty$. Hence it follows by Lemma \ref{lem.29inv} that $$\rho(g_i) \tilde{x} \to \tilde{k}_0 P_{\i(\theta_2)} = \lim \rho(g_i) o_2.$$ Since $\tilde{x} \neq f_{\i}(\xi)$, we have $\lim \rho(g_i) o_2 = f_{\i}(\eta_0)$, as desired.

Therefore, we have $(\xi, f(\xi)) \in \La_{\rho, M}^{\theta}$. Since it holds for $\nu \otimes \nu_{\i} \otimes dt$-a.e. $(\xi, \eta, t) \in \La^{(2)} \times \R$ and $\nu_{\rho} = (\id \times f)_*\nu$, we have $\nu_{\rho}(\La_{\rho, M}^{\theta}) = 1$, finishing the proof.
\qed

\section{$\delta$-hyperbolic spaces} \label{sec.hyperbolic}
In this section, we present certain results on geometric properties of a $\delta$-hyperbolic space; while the qualitative statements in this section is known to experts, it is important for us to express every constant purely in terms of $\delta$.
We refer to (\cite[Part III]{Bridson1999metric}, \cite{Bowditch1999convergence}, \cite{Kapovich_boundary}, \cite[Chapter 1]{CP_symbolic}) for comprehensive expositions.

Let $(Z, d_Z)$ be a proper geodesic metric space. The Gromov product of $y, z \in Z$ with respect to $x \in Z$ is defined as follows:
$$\langle y, z \rangle_x := \frac{1}{2}\left( d_Z(x, y) + d_Z(x, z) - d_Z(y, z) \right).$$
It is straightforward to see that for all $x, y, z \in Z$,
$$
    \langle y, z \rangle_x  = \langle z, y \rangle_x \quad \text{and} \quad 
    0 \le \langle y, z \rangle_x  \le d_Z(x, y).
$$ 

For $\delta  \ge 0$, we call that $Z$ is $\delta$-hyperbolic if $$\langle w, z \rangle_x \ge \min \{ \langle w, y \rangle_x, \langle y, z\rangle_x \} - \delta$$
for all $w, x, y, z \in Z$. The metric space $Z$ is called Gromov hyperbolic if it is $\delta$-hyperbolic for some $\delta \ge 0$.

In the rest of this section, let $Z$ be a proper geodesic  $\delta$-hyperbolic space for $\delta \ge 0$. We fix the constant $\delta$ and keep track of other constants in terms of $\delta$ in the following discussion.

\subsection*{Basic geometry}
We first discuss some basic geometry of $Z$. The following standard lemma says that the Gromov product gives the length of the initial segments of two geodesics from a common point with uniformly bounded Hausdorff distance:

\begin{lemma} \label{lem.fellowtravel}
    Let $x, y, z \in Z$ and fix geodesic segments $[x, y], [x, z] \subset Z$ between $x$ and $y$, and $x$ and $z$, respectively. If $y' \in [x, y]$ and $z' \in [x, z]$ are such that $d_Z(x, y') = d_Z(x, z') \le \langle y, z \rangle_x$, then $d_Z(y', z') \le 4 \delta$.
\end{lemma}

\begin{proof}
    Since $Z$ is $\delta$-hyperbolic, we have $$\begin{aligned}
        \langle y', z' \rangle_x & \ge \min \{ \langle y', z \rangle_x, \langle z, z' \rangle_x \} - \delta \\
        & \ge \min \{ \min\{ \langle y', y \rangle_x, \langle y, z \rangle_x \} - \delta, \langle z, z' \rangle_x \} - \delta \\
        & = \min \{ \min\{ d_Z(x, y'), \langle y, z \rangle_x \} - \delta, d_Z(x, z') \} - \delta \\
        & = d_Z(x, y') - 2 \delta.
    \end{aligned}$$ 
    On the other hand, $\langle y', z'\rangle_x = d_Z(x, y') - \frac{1}{2} d_Z(y', z')$, from which the claim follows.
\end{proof}

As a corollary, we deduce that every geodesic triangle in $Z$ is uniformly thin:

\begin{corollary}
    Let $x, y, z \in Z$ and fix geodesic segments $[x, y], [y, z], [x, z] \subset Z$ between $x$ and $y$, $y$ and $z$, and $x$ and $z$, respectively. Then $[x, y]$ is contained in the $4\delta$-neighborhood of $[x, z] \cup [y, z]$.
\end{corollary}

We also obtain the interpretation that the Gromov product roughly measures a distance between a point and a geodesic segment.

\begin{corollary} \label{cor.Gromovproductasdistance}
    Let $x, y, z \in Z$ and fix a geodesic segment $[y, z] \subset Z$ between $y$ and $z$. Then $$d_Z(x, [y, z]) - 4 \delta \le \langle y, z \rangle_x \le d_Z(x, [y, z]).$$
\end{corollary}

\begin{proof}
    Let $w \in [y, z]$ be such that $d_Z(x, w) = d_Z(x, [y, z])$. Then $$\begin{aligned}
        \langle y, z \rangle_x & = \frac{1}{2} \left( d_Z(x, y) + d_Z(x, z) - d_Z(y, z) \right) \\
        & = \frac{1}{2} \left( d_Z(x, y)  - d_Z(y, w) \right) + \frac{1}{2} \left( d_Z(x, z) - d_Z(w, z) \right) \\
        & \le d_Z(x, w) = d_Z(x, [y, z]).
    \end{aligned}$$
    To see the lower bound, fix a geodesic segment $[x, y] \subset Z$ between $x$ and $y$ and let $y' \in [x, y]$ be the point such that $d_Z(x, y') = \langle y, z \rangle_x$. Since $d_Z(x, y) = \langle y, z \rangle_x + \langle x, z \rangle_y$, we have $d_Z(y, y') = \langle x, z \rangle_y$. Let $z' \in [y, z]$ be the point such that $d_Z(y, z') = \langle x, z \rangle_y$. By Lemma \ref{lem.fellowtravel}, $d_Z(y', z') \le 4 \delta$, and hence $$d_Z(x, z') \le d_Z(x, y') + d_Z(y', z') \le \langle y, z \rangle_x + 4 \delta.$$
    Since $d_Z(x, [y, z]) \le d_Z(x, z')$, this finishes the proof.
\end{proof}

Note that in Corollary \ref{cor.Gromovproductasdistance}, the choice of a geodesic segment was made while the Gromov product does not involve any choice of a geodesic segment. Indeed, geodesics between two points are stable:

\begin{corollary} \label{lem.stablesegment}
    Let $y, z \in Z$. Then two geodesic segments in $Z$ between $y$ and $z$ have Hausdorff distance at most $4 \delta$.
\end{corollary}

\begin{proof}
    Let $\sigma_1, \sigma_2 \subset Z$ be two geodesics between $y$ and $z$. Fix any $x \in \sigma_1$. Then $\langle y, z \rangle_x = 0$. By Corollary \ref{cor.Gromovproductasdistance}, this implies $d_Z(x, \sigma_1) \le 4 \delta$. Since $x$ is arbitrary, $\sigma_2$ is contained in the $4\delta$-neighborhood of $\sigma_1$. The same argument switching $\sigma_1$ and $\sigma_2$ finishes the proof.
\end{proof}

The following will be a useful observation:

\begin{corollary} \label{cor.projection}
    Let $x \in Z$ and $\sigma$ be a geodesic segment in $Z$. Let $y \in \sigma$ be such that $d_Z(x, y) = d_Z(x, \sigma)$. Then for any $z \in \sigma$, we have $$d_Z(x, y) + d_Z(y, z) - 8 \delta \le d_Z(x, z) \le d_Z(x, y) + d_Z(y, z).$$
\end{corollary}

\begin{proof}
    The upper bound is straightforward. By Corollary \ref{cor.Gromovproductasdistance}, we have $$\begin{aligned}
        d_Z(x, y) & \le \langle y, z \rangle_x + 4 \delta \\
        & = \frac{1}{2} ( d_Z(x, y) + d_Z(x, z) - d_Z(y, z)) + 4 \delta.
    \end{aligned}$$
    This implies the lower bound.
\end{proof}

\subsection*{Gromov boundary}
An isometric embedding $\sigma: [0, \infty) \to Z$, or its image in $Z$, is called a geodesic ray in $Z$. The Gromov boundary of $Z$ is defined as the set of all equivalence classes of geodesic rays in $Z$:
$$\partial Z := \{ \sigma : [0, \infty) \to Z, \text{ a geodesic ray} \} / \sim$$ where $\sigma \sim \sigma'$ if the Hausdorff distance between  two geodesic rays $\sigma([0, \infty))$ and $\sigma'([0, \infty))$ is finite. We denote by $\sigma(\infty) \in \partial Z$ the equivalence class of the geodesic ray $\sigma : [0, \infty) \to Z$. Fixing a basepoint in $Z$, the Gromov boundary $\partial Z$ is visible from the basepoint:

\begin{lemma} \cite[Lemma III.3.1]{Bridson1999metric} \label{lem.visibility}
    Let $x \in Z$ and $\xi \in \partial Z$. Then there exists a geodesic ray $\sigma : [0, \infty) \to Z$ such that $\sigma(0) = x$ and $\sigma(\infty) = \xi$.
\end{lemma}

Moreover, this visualization is stable under the choice of the basepoint:

\begin{lemma} \cite[Lemma III.3.3]{Bridson1999metric} \label{lem.raystopoint}
    Let $\sigma_1, \sigma_2 : [0, \infty) \to Z$ be geodesic rays with $\sigma_1(\infty) = \sigma_2(\infty)$.
    \begin{enumerate}
        \item If $\sigma_1(0) = \sigma_2(0)$, then $d_Z(\sigma_1(t), \sigma_2(t)) \le 8 \delta$ for all $t \ge 0$.
        \item In general, there exist $T_1, T_2 \ge 0$ such that $d_Z(\sigma_1(t + T_1), \sigma_2(t + T_2)) \le 20 \delta$ for all $t \ge 0$.
    \end{enumerate}
\end{lemma}

Hence, fixing a basepoint $x \in Z$, we can identify the Gromov boundary of $Z$ with the set of all equivalence classes of geodesic rays in $Z$ based at $x$:
$$\partial Z = \{ \sigma : [0, \infty) \to Z, \text{ a geodesic ray with } \sigma(0) = x \} / \sim$$ where $\sigma \sim \sigma'$ if $\sigma(t) \sim \sigma'(t) \le 8 \delta$ for all $t \ge 0$. Under this identification, a natural topology on $\partial Z$ is given as follows: for $\xi \in \partial Z$ and $r \ge 0$, we set $$V(\xi, r) := \left\{ \eta \in \partial Z : \begin{matrix}
    \text{for some geodesic rays } \sigma, \sigma' \text{ from } x\\
     \text{with }
     \sigma(\infty) = \xi \text{ and } \sigma'(\infty) = \eta,\\ \text{we have } \liminf_{t \to \infty} \langle \sigma(t), \sigma(t') \rangle_x \ge r 
\end{matrix} \right\}$$
 which consists of the geodesic rays from $x$ that are $8\delta$-close to $\sigma$ for a long time. We topologize $\partial Z$ by setting $\{ V(\xi, r) : \xi \in \partial Z, r \ge 0\}$ to be the basis.

We now consider $\bar Z = Z \cup \partial Z$ and give it a natural topology. We say that a sequence $(x_i)$ converges to infinity if $ \liminf_{i, j \to \infty} \langle x_i, x_j \rangle_x = \infty$. To topologize $\bar Z$, it is useful to associate a geodesic ray $\sigma : [0, \infty) \to Z$ with a sequence $(\sigma(i))_{i \in \N}$ which converges to infinity. This gives a map
$$ \partial Z \to \{ (x_i) \subset Z, \text{ a sequence converging to infinity}\} / \sim$$ where $(x_i) \sim (y_i)$ if $\liminf_{i, j \to \infty} \langle x_i, y_j \rangle_x = \infty$. The above map is indeed a bijection, and hence we identify them as well. We denote by $[(x_i)]$ the equivalence class of the sequence $(x_i)$. Similar to the above, for $\xi \in \partial Z$ and $r \ge 0$, we set $$U(\xi, r) := \left\{ \eta \in \partial Z : \begin{matrix}
    \text{for some sequences } (x_i), (y_i) \\
    \text{with } [(x_i)] = \xi \text{ and } [(y_i)] = \eta,\\
    \text{we have } \liminf_{i, j \to \infty} \langle x_i, y_j \rangle_x \ge r
\end{matrix} \right\}.$$
The topology on $\partial Z$ given by setting $\{U(\xi, r) : \xi\in \partial Z, r \ge 0\}$ as a basis is equivalent to the one defined in terms of $V(\xi, r)$. To obtain a basis for $\bar Z$, we also consider for $\xi \in \partial Z$ and $r \ge 0$ $$U'(\xi, r) := U(\xi, r) \cup \left\{ y \in Z : \begin{matrix}
     \text{for some sequence }(x_i) \text{ with } [(x_i)] =\xi,\\
     \text{we have } \liminf_{i \to \infty} \langle x_i, y \rangle_x \ge r
\end{matrix} \right\}.$$
Then setting $\{U'(\xi, r) : \xi \in \partial Z, r \ge 0\}$ and metric balls in $Z$ to be the basis, $\bar Z$ is equipped with the topology.
In this topology, a sequence $x_i$ in $Z$ converges to $\xi \in \partial Z$ if and only if $\xi = [(x_i)]$.
The spaces $\partial Z$ and $\bar Z$ equipped with these topologies are compact, and the topologies do not depend on the choice of the basepoint $x$.
We refer to (\cite{Bridson1999metric}, \cite{Kapovich_boundary}) for details.

\subsection*{Extended Gromov product} We extend the notion of the Gromov product to $\bar Z$: for $y, z \in \bar Z$ and $x \in Z$, the Gromov product of $y$ and $z$ with respect to $x$ is defined as $$\langle y, z \rangle_x := \sup \liminf_{i, j \to \infty} \langle y_i, z_j \rangle_x$$
where the supremum is taken over all sequences $(y_i)$ and $(z_j)$ in $Z$ such that $y = \lim_i y_i$ and $z = \lim_j z_j$. We note the following properties of the extended Gromov product:

\begin{lemma} \cite[Remark III.3.17]{Bridson1999metric} \label{lem.basicextendedgp}
    Fix $x \in Z$. \begin{enumerate}
        \item For $y, z \in \partial Z$, $\langle y, z \rangle_x = \infty$ if and only if $y = z$.
        \item For $w, y, z \in \bar Z$, we have $$\langle w, z \rangle_x \ge \min \{ \langle w, y \rangle_x, \langle y, z \rangle_x \} - 2 \delta.$$
        \item For $y, z \in \bar Z$ and sequences $(y_i), (z_j)$ in $Z$ with $\lim_i y_i = y$ and $\lim_j z_j = z$, we have $$\langle y, z \rangle_x - 2 \delta \le \liminf_{i, j} \langle y_i, z_j \rangle_x \le \langle y, z \rangle_x.$$
        \item For $z \in \partial Z$ and a sequence $(z_i)$ in $\partial Z$, $z_i \to z$ as $i \to \infty$ if and only if $\langle z, z_i \rangle_x \to \infty$ as $i \to \infty$.
    \end{enumerate}
\end{lemma}

Given two distinct points $y, z \in \partial Z$, there exists a bi-infinite geodesic $\sigma : \R \to Z$ connecting $y$ and $z$, i.e., $\sigma(-\infty) = y$ and $\sigma(\infty) = z$ \cite[Lemma III.3.2]{Bridson1999metric}. Hence in general we can consider a geodesic between two points in $\bar Z$. The extended Gromov product $\langle y, z \rangle_x$ also measures the crude distance from $x$ to a geodesic between $y, z \in \bar Z$.

\begin{corollary} \label{cor.extendedgpdistance}
    Let $x \in Z$ and $y, z \in \bar Z$ be distinct points. Let $[y, z]$ be a geodesic connecting $y$ and $z$. Then we have $$d_Z(x, [y, z]) - 4 \delta \le \langle y, z \rangle_x \le d_Z(x, [y, z]) + 2\delta.$$
\end{corollary}

\begin{proof}
    Let $w \in [y, z]$ be such that $d_Z(x, w) = d_Z(x, [y, z])$. Let $(y_i)$ and $(z_j)$ be sequences of points on $[y, z]$ such that $\lim_i y_i = y$ and $\lim_j z_j = z$. For each $i$ and $j$, let $[y_i, z_j] \subset [y, z]$ be the segment between $y_i$ and $z_j$. Then for large enough $i$ and $j$, we have $w \in [y_i, z_j]$ and hence by Corollary \ref{cor.Gromovproductasdistance}, $$d_Z(x, [y, z]) - 4 \delta \le \langle y_i, z_j \rangle_x \le  d_Z(x, [y, z])$$ since $d_Z(x, [y_i, z_j]) = d_Z(x, [y, z])$. Applying Lemma \ref{lem.basicextendedgp}(3) finishes the proof.
\end{proof}

As in Corollary \ref{lem.stablesegment}, we also obtain the stability of geodesics between two points in $\bar Z$.

\begin{corollary} \label{lem.stablebiinfinite}
    Let $y, z \in \bar Z$. Then two geodesics between $y$ and $z$ have Hausdorff distance at most $6 \delta$. 
\end{corollary}

\begin{proof}
        Suppose first that $y, z \in \partial Z$. Let $\sigma_1, \sigma_2 : \R \to Z$ be two bi-infinite geodesics between $y$ and $z$. Let $x \in \sigma_1(\R)$. Then by Corollary \ref{cor.extendedgpdistance}, we have $$d_Z(x, \sigma_2(\R)) \le \langle y, z \rangle_x + 4 \delta.$$
        On the other hand, $0 = \liminf_{t \to \infty} \langle \sigma_1(-t), \sigma_1(t) \rangle_x \ge \langle y, z \rangle_x - 2 \delta$ by Lemma \ref{lem.basicextendedgp}(3). Therefore we have $$d_Z(x, \sigma_2(\R)) \le 6 \delta.$$ Since $x$ is arbitrary, this finishes the proof in this case. The case when one of $y$ and $z$ is in $Z$ can be handled similarly.
\end{proof}

\subsection*{Visual metric}  Indeed, $\partial Z$ can be equipped with a natural metric-like function.
From the above observation, it is natural to consider the following function which plays a role of metric on $\partial Z$, which we call the visual metric on $\partial Z$, although it may not satisfy the triangle inequality in general:
\begin{definition}
    Let $x \in Z$. We define a function $d_x : \partial Z \times \partial Z \to \R$ as $$d_x(y, z) := e^{-2 \langle y, z \rangle_x}.$$
    For $y \in \partial Z$ and $r > 0$ we consider the $d_x$-ball $$B_x(y, r) := \{ z \in \partial Z : d_x(y, z) < r\}.$$
\end{definition}

Usually the visual metric is defined without the multiplication by $2$. However, we defined it as above in order to simplify the later computation. The visual metric is compatible to a genuine metric on $\partial Z$ after taking a suitable power:

\begin{proposition} \cite[Proposition III.3.21]{Bridson1999metric}
    Let $x \in Z$. For any small enough $\varepsilon > 0$, there exists a constant $c_{\varepsilon}$ and a metric $d_{\varepsilon}$ on $\partial Z$ such that $$d_{\varepsilon}(y, z) \le d_x(y, z)^{\varepsilon} \le c_{\varepsilon} d_{\varepsilon}(y, z)$$ for all $y, z \in \partial Z$.
\end{proposition}

It follows from Lemma \ref{lem.basicextendedgp}(2) that for any $w, y, z \in \partial Z$, we have \be \label{eqn.almosttriangle}
d_x(w, z) \le e^{4\delta} (d_x(w, y) + d_x(y, z))
\ee
From this we deduce the following Vitali-type covering lemma:

\begin{lemma} \label{lem.Vcoveringtrans}
    Let $x \in Z$ and $B_x(y_1, r_1), \cdots, B_x(y_n, r_n)$ a finite collection of $d_x$-balls for $y_i \in \partial Z$ and $r_i > 0$. Then there a subcollection of disjoint balls $B_x(y_{i_1}, r_{i_1}), \cdots, B_x(y_{i_k}, r_{i_k})$ such that $$\bigcup_{i=1}^n B_x(y_i, r_i) \subset \bigcup_{j=1}^k B_x(y_{i_j}, 3 e^{8 \delta} r_{i_j}).$$
\end{lemma}

\begin{proof}
    Given a finite collection $B_x(y_1, r_1), \cdots, B_x(y_n, r_n)$ of $d_x$-balls, we rearrange them so that we may assume $r_1 \ge \cdots \ge r_n$. Let $i_1 = 1$ and for each $j \ge 2$, we set $i_j = \min \{ i > i_{j-1} : B_x(y_i, r_i) \cap \bigcup_{\ell = 1}^{i_{j-1}} B_x(y_{\ell}, r_{\ell}) = \emptyset\}$. Then we obtain a subcollection $B_x(y_{i_1}, r_{i_1}), \cdots, B_x(y_{i_k}, r_{i_k})$ consisting of disjoint balls.

    For each $i$, $B_x(y_i, r_i)$ intersects $B_x(y_{i_j}, r_{i_j})$ for some $j$ such that $r_{i_j} \ge r_i$. Choosing a point $y \in B_x(y_i, r_i) \cap B_x(y_{i_j}, r_{i_j})$, it follows from \eqref{eqn.almosttriangle} that for any $z \in B_x(y_i, r_i)$, $$\begin{aligned}
        d_x(z, y_{i_j}) & \le e^{4\delta} (d_x(z, y_i) + d_x(y_i, y_{i_j})) \\
        & \le e^{4\delta}(r_i + e^{4\delta}(d_x(y_i, y) + d_x(y, y_{i_j}))) \\
        & \le e^{4\delta}(r_i + e^{4\delta}(r_i + r_{i_j})) \le 3e^{8\delta}r_{i_j}.
    \end{aligned}$$ Hence $B_x(y_i, r_i) \subset B_{x}(y_{i_j}, 3e^{8\delta} r_{i_j})$. This finishes the proof.
\end{proof}

\subsection*{Busemann functions}
Let $\sigma : [0, \infty) \to Z$ be a geodesic ray and $y, z \in Z$. Then the following limit is well-defined and satisfies the following inequality: 
    \be \label{lem.busedef}
    - d_Z(y, z) \le \lim_{t \to \infty} d_Z(y, \sigma(t)) - d_Z(z, \sigma(t)) \le d_Z(y, z).
    \ee
Therefore, we define the Busemann function as follows:
$$\beta_{\sigma}(y, z) := \lim_{t \to \infty} d_Z(y, \sigma(t)) - d_Z(z, \sigma(t)).$$ 
Observe that: for $w, y, z \in Z$,
    \begin{enumerate}
        \item we have $|\beta_{\sigma}(y, z)| \le d_Z(y, z)$;
        \item we have $\beta_{\sigma}(y, z) = - \beta_{\sigma}(z, y)$;
        \item we have $\beta_{\sigma}(w, z) = \beta_{\sigma}(w, y) + \beta_{\sigma}(y, z)$.
    \end{enumerate}
The Busemann function depends only on the endpoint at $\partial Z$, independent of a choice of a geodesic ray, up to a uniform error.

\begin{lemma} \label{lem.buseatpoint}
    Let $\sigma, \sigma': [0, \infty) \to Z$ be geodesic rays such that $\sigma(\infty) = \sigma'(\infty)$. Then for any $y, z \in Z$, we have $$|\beta_{\sigma}(y, z) - \beta_{\sigma'}(y, z)| \le 40 \delta.$$
\end{lemma}

\begin{proof}
    By Lemma \ref{lem.raystopoint}, there exists $T, T' > 0$ such that $$d_{Z}(\sigma(t + T), \sigma'(t + T')) \le 20 \delta$$ for all $t \ge 0$. This implies the desired inequality.
\end{proof}

Moreover, the Busemann function is stable under the change of the endpoint.

\begin{lemma} \label{lem.stablebuse0415}
    Let $x \in Z$ and $r > 0$. Let $y, z \in Z$ be such that $d_Z(x, y) < r - 10 \delta$ and $d_Z(x, z) < r - 10 \delta$. Let $\sigma, \sigma' : [0, \infty) \to Z$ be geodesic rays with $\langle \sigma(\infty), \sigma'(\infty) \rangle_x > r$. If $\sigma(0) = \sigma'(0) = x$, then $$|\beta_{\sigma}(y, z) - \beta_{\sigma'}(y, z)| \le 72 \delta.$$
    In general, $$|\beta_{\sigma}(y, z) - \beta_{\sigma'}(y, z)| \le 152 \delta.$$
\end{lemma}

\begin{proof}
    Suppose first that $\sigma(0) = \sigma'(0) = x$. By Lemma \ref{lem.basicextendedgp}(3), we have for all large $t > 0$ that $$\langle \sigma(t), \sigma'(t) \rangle_x > r - 2\delta.$$ Fix such $t > 0$. Let $t_0 \ge 0$ be such that $d_Z(y, \sigma([0, \infty))) = d_Z(y, \sigma(t_0))$. By Corollary \ref{cor.projection}, we have 
    \be \label{eqn.busestable1}
    t_0 + d_Z(\sigma(t_0), y) - 8 \delta \le d_Z(x, y) \le t_0 + d_Z(\sigma(t_0), y).
    \ee
    In particular, we have $t_0 \le d_Z(x, y) + 8 \delta < r - 2 \delta < \langle \sigma(t), \sigma'(t) \rangle_x$. Hence by Lemma \ref{lem.fellowtravel}, we have 
    \be \label{eqn.busestable2}
    d_Z(\sigma(t_0), \sigma'(t_0)) \le 4 \delta.
    \ee
    Similarly, letting $t_0' \ge 0$ be such that $d_Z(y, \sigma'([0, \infty))) = d_Z(y, \sigma'(t_0'))$, it follows from Corollary \ref{cor.projection} that 
    \be \label{eqn.busestable3}
    t_0' + d_Z(\sigma'(t_0'), y) - 8 \delta \le d_Z(x, y) \le t_0' + d_Z(\sigma'(t_0'), y).
    \ee
    Similarly, by Lemma \ref{lem.fellowtravel}, we have
    \be \label{eqn.busestable20}
    d_Z(\sigma(t_0'), \sigma'(t_0')) \le 4 \delta.
    \ee

    Combining \eqref{eqn.busestable1} and \eqref{eqn.busestable3}, we have $$d_Z(\sigma'(t_0'), y) - d_Z(\sigma(t_0), y) - 8 \delta \le t_0 - t_0' \le d_Z(\sigma'(t_0'), y) - d_Z(\sigma(t_0), y) + 8 \delta.$$
    Since $d_Z(y, \sigma'([0, \infty)) = d_Z(y, \sigma'(t_0'))$, we have
    $$\begin{aligned}
        t_0 - t_0' & \le d_Z(\sigma'(t_0'), y) - d_Z(\sigma(t_0), y) + 8 \delta \\
        & \le d_Z(\sigma'(t_0), y) - d_Z(\sigma(t_0), y) + 8 \delta\\
        & \le d_Z(\sigma'(t_0), \sigma(t_0)) + 8 \delta \\
        & \le 12 \delta
    \end{aligned}$$
    where the last inequality is by \eqref{eqn.busestable2}. Similarly, we have
    $$\begin{aligned}
        t_0 - t_0' & \ge d_Z(\sigma'(t_0'), y) - d_Z(\sigma(t_0), y) - 8 \delta \\
        & \ge d_Z(\sigma'(t_0'), y) - d_Z(\sigma(t_0'), y) - 8 \delta\\
        & \le - d_Z(\sigma'(t_0'), \sigma(t_0')) - 8 \delta \\
        & \le -12 \delta
    \end{aligned}$$
    where the last inequality is by \eqref{eqn.busestable20}.
    Therefore we obtain $$|t_0 - t_0'| \le 12 \delta,$$ and hence $$d_Z(\sigma(t_0), \sigma'(t_0')) \le d_Z(\sigma(t_0), \sigma(t_0')) + d_Z(\sigma(t_0'), \sigma'(t_0')) \le 16 \delta.$$

    Now for all large $t > 0$, it follows from Corollary \ref{cor.projection} that $$\begin{aligned}
        d_Z(x, \sigma(t)) & - d_Z(y, \sigma(t_0)) - d_Z(\sigma(t_0), \sigma(t)) \\
        & \le d_Z(x, \sigma(t)) - d_Z(y, \sigma(t)) \\
        & \le d_Z(x, \sigma(t))  - d_Z(y, \sigma(t_0)) - d_Z(\sigma(t_0), \sigma(t)) + 8 \delta.
    \end{aligned}$$
    This implies $$t_0 - d_Z(y, \sigma(t_0)) \le d_Z(x, \sigma(t)) - d_Z(y, \sigma(t)) \le t_0 - d_Z(y, \sigma(t_0)) + 8 \delta$$
    for all large enough $t > 0$, and therefore we have $$t_0 - d_Z(y, \sigma(t_0)) \le \beta_{\sigma}(x, y) \le t_0 - d_Z(y, \sigma(t_0)) + 8 \delta.$$
    Similarly, we also have $$t_0' - d_Z(y, \sigma'(t_0')) \le \beta_{\sigma'}(x, y) \le t_0' - d_Z(y, \sigma'(t_0')) + 8 \delta.$$
    Hence, we have $$\begin{aligned}
        |\beta_{\sigma}(x, y) - \beta_{\sigma'}(x, y)| & \le |t_0 - t_0'| + |d_Z(y, \sigma(t_0)) - d_Z(y, \sigma'(t_0'))| + 8 \delta \\
        & \le 12 \delta + 16 \delta + 8 \delta = 36 \delta.
    \end{aligned}$$

    By the same argument replacing $y$ with $z$, we also have $$|\beta_{\sigma}(x, z) - \beta_{\sigma'}(x, z) | \le 36 \delta.$$
    Therefore, it follows that $$|\beta_{\sigma}(y, z) - \beta_{\sigma'}(y, z)| \le 72 \delta,$$ proving the first claim.

    The last claim follows from the first claim, by applying Lemma \ref{lem.visibility} and Lemma \ref{lem.buseatpoint}.
\end{proof}

\begin{lemma} \label{lem.extgpbuse}
    Let $\sigma_1, \sigma_2 : [0, \infty) \to Z$ be geodesic rays from $x \in Z$. For any $w \in Z$ on a bi-infinite geodesic between $\sigma_1(\infty)$ and $\sigma_2(\infty)$, we have $$  \langle \sigma_1(\infty), \sigma_2(\infty) \rangle_x - 42 \delta \le \frac{1}{2} \left( \beta_{\sigma_1}(x, w) + \beta_{\sigma_2}(x, w) \right) \le \langle \sigma_1(\infty), \sigma_2(\infty) \rangle_x.$$
\end{lemma}

\begin{proof}
    Let $[\sigma_1(\infty), \sigma_2(\infty)]$ be a bi-infinite geodesic in $Z$ between $\sigma_1(\infty)$ and $\sigma_2(\infty)$ and $w \in [\sigma_1(\infty), \sigma_2(\infty)]$. We then have $$\begin{aligned}
        \beta_{\sigma_1}(x, w) & + \beta_{\sigma_2}(x, w) \\ 
        & = \lim_{t \to \infty} d_Z(x, \sigma_1(t)) - d_Z(w, \sigma_1(t)) + d_Z(x, \sigma_2(t)) - d_Z(w, \sigma_2(t)) \\
        & \le \liminf_{t \to \infty} d_Z(x, \sigma_1(t)) + d_Z(x, \sigma_2(t)) - d_Z(\sigma_1(t), \sigma_2(t))\\
        & = 2 \liminf_{t \to \infty} \langle \sigma_1(t), \sigma_2(t) \rangle_x.
    \end{aligned}$$ Hence, by Lemma \ref{lem.basicextendedgp}(3), the upper bound follows.

    To see the lower bound, let $\sigma_1', \sigma_2' : [0, \infty) \to Z$ be geodesic rays that parametrize the segments of $[\sigma_1(\infty), \sigma_2(\infty)]$ from $w$ to $\sigma_1(\infty)$, from $w$ to $\sigma_2(\infty)$, respectively. Then $$\begin{aligned}
        \beta_{\sigma_1'}(x, w) & + \beta_{\sigma_2'}(x, w) \\ 
        & = \lim_{t \to \infty} d_Z(x, \sigma_1'(t)) - d_Z(w, \sigma_1'(t)) + d_Z(x, \sigma_2'(t)) - d_Z(w, \sigma_2'(t)) \\
        & = \lim_{t \to \infty} d_Z(x, \sigma_1'(t)) + d_Z(x, \sigma_2'(t)) - d_Z(\sigma_1'(t), \sigma_2'(t))\\
        & = 2 \lim_{t \to \infty} \langle \sigma_1'(t), \sigma_2'(t) \rangle_x.
    \end{aligned}$$
    Therefore, by Lemma \ref{lem.buseatpoint} and Lemma \ref{lem.basicextendedgp}(3), we have $$\begin{aligned}
        \beta_{\sigma_1}(x, w) + \beta_{\sigma_2}(x, w) & \ge \beta_{\sigma_1'}(x, w)  + \beta_{\sigma_2'}(x, w) - 80 \delta\\ 
        & = 2 \lim_{t \to \infty} \langle \sigma_1'(t), \sigma_2'(t) \rangle_x - 80 \delta \\
        & \ge 2 \langle \sigma_1(\infty), \sigma_2(\infty) \rangle_x - 84 \delta.
    \end{aligned}$$
    This implies the desired lower bound.
\end{proof}

We can now compare two visual metrics in terms of Busemann functions:

\begin{lemma} \label{lem.vmbpchange}
    Let $y, z \in \partial Z$ and $x, x' \in Z$. Let $\sigma_1', \sigma_2' : [0, \infty) \to Z$ be geodesic rays from $x'$ to $y'$ and $z'$ respectively. Let $\sigma_1, \sigma_2 : [0, \infty) \to Z$ be any geodesic rays with $\sigma_1(\infty) = y$ and $\sigma_2(\infty) = z$.  Then $$\begin{aligned}
        d_{x'}(y, z) & \le e^{164 \delta} e^{\beta_{\sigma_1'}(x, x') + \beta_{\sigma_2'}(x, x')} d_x(y, z) \\
        &  \le e^{244 \delta} e^{\beta_{\sigma_1}(x, x') + \beta_{\sigma_2}(x, x')} d_x(y, z).
    \end{aligned}$$
\end{lemma}

\begin{proof}
    Let $[y, z]$ be a bi-infinite geodesic between $y$ and $z$ and let $w \in [y, z]$. By Lemma \ref{lem.extgpbuse} and Lemma \ref{lem.buseatpoint}, we have $$\begin{aligned}
        2 \langle y, z \rangle_{x'} & \ge \beta_{\sigma_1'}(x', w) + \beta_{\sigma_2'}(x', w) \\
            & = \beta_{\sigma_1'}(x, w) + \beta_{\sigma_2'}(x, w) + \beta_{\sigma_1'}(x', x) + \beta_{\sigma_2'}(x', x) \\
            & \ge \beta_{\sigma_1}(x, w) + \beta_{\sigma_2}(x, w) + \beta_{\sigma_1'}(x', x) + \beta_{\sigma_2'}(x', x) - 80 \delta \\
            & \ge 2 \langle y, z \rangle_x + \beta_{\sigma_1'}(x', x) + \beta_{\sigma_2'}(x', x) - 164 \delta
    \end{aligned}$$
    Therefore, $$d_{x'}(y, z) \le e^{164 \delta} e^{\beta_{\sigma_1'}(x, x') + \beta_{\sigma_2'}(x, x')} d_x(y, z).$$ Applying Lemma \ref{lem.buseatpoint} again, we deduce $$e^{164 \delta} e^{\beta_{\sigma_1'}(x, x') + \beta_{\sigma_2'}(x, x')} d_x(y, z) \le e^{244 \delta} e^{\beta_{\sigma_1}(x, x') + \beta_{\sigma_2}(x, x')} d_x(y, z).$$
\end{proof}

\subsection*{Shadows}
Let $x \in Z$, $y \in \bar Z$, and $R > 0$. The shadow of $R$-ball at $x$ viewed from $y$ is defined as $$O_R^Z(y, x) = \{ z \in \partial Z : \langle y, z \rangle_x < R \}.$$
Busemann function is comparable to the distance in a shadow:
\begin{lemma} \label{lem.buseinshadow}
    Let $x, y \in Z$ and $R > 0$. Let $\sigma : [0, \infty) \to Z$ be a geodesic ray such that $\sigma(\infty) \in O_R^Z(y, x)$. Then $$|\beta_{\sigma}(y, x) - d_Z(y, x)| < 2 R + 48 \delta.$$
\end{lemma}

\begin{proof}
    Fix any geodesic $\sigma' : [0, \infty) \to Z$ from $y$ to $\sigma(\infty)$. By Corollary \ref{cor.extendedgpdistance}, there exists $t_0 \ge 0$ such that $d_Z(x, \sigma'(t_0)) < R + 4 \delta$. We then have $$ \begin{aligned}
        | \beta_{\sigma'}(y, x) - d_Z(y, \sigma'(t_0)) | & = \left| \lim_{t \to \infty} d_Z(\sigma'(t_0), \sigma'(t)) - d_Z(x, \sigma'(t))\right| \\
        & \le d_Z(x, \sigma'(t_0)) < R + 4 \delta.
    \end{aligned}$$
    By Lemma \ref{lem.buseatpoint}, we have $$\begin{aligned}
        |\beta_{\sigma}(y, x) - d_Z(y, x)| & \le 40 \delta + |\beta_{\sigma'}(y, x) - d_{Z}(y, \sigma'(t_0))| \\
        & \quad + | d_{Z}(y, \sigma'(t_0)) - d_Z(y, x) | \\
        & < 40 \delta + (R + 4 \delta) + (R + 4 \delta) = 2R + 48 \delta.
    \end{aligned}$$
\end{proof}

Shadows viewed from $\partial Z$ can be approximated by shadows viewed from~$Z$.

\begin{lemma} \label{lem.shadowapprox}
    Let $x \in Z$ and $(y_i)$ be a sequence in $Z$ such that $\lim_i y_i = y \in \partial Z$. Then for any $R > 0$, we have $$O_R^Z(y, x) \subset O_{R + 2\delta}^Z(y_i, x)$$ for all $i \ge 1$ large enough.
\end{lemma}

\begin{proof}
    Let $z \in O_R^{Z}(y, x)$. Then by Lemma \ref{lem.basicextendedgp}(2), we have for each $i \ge 1$ that $$R > \langle y, z \rangle_x \ge \min \{ \langle y, y_i \rangle_x , \langle y_i, z \rangle_x \} - 2 \delta.$$ Since $y_i \to y$ as $i \to \infty$, for $i$ large enough so that $\langle y, y_i \rangle_x > R + 2 \delta$, we have $$\langle y_i, z \rangle_x < R + 2 \delta.$$ This shows the claim.
\end{proof}

\subsection*{Isometries}
Let $g \in \Isom(Z)$ be an isometry of $Z$. Then $g : Z \to Z$ extends to a homeomorphism $g : \bar Z \to \bar Z$. One can see that for $x, w \in Z$, $y, z \in \bar Z$, and a geodesic ray $\sigma : [0, \infty) \to Z$, $$\langle g y, g z \rangle_{gx} = \langle y, z \rangle_x \quad \text{and} \quad \beta_{g \sigma}(gx, gw) = \beta_{\sigma}(x, w).$$

Isometries of $Z$ are classified into three categories. Let $g \in \Isom(Z)$. Then either one of the following holds:
\begin{enumerate}
    \item $g$ is elliptic, i.e., $\{g^n x : n \in \Z \}$ is bounded for any $x \in Z$;
    \item $g$ is parabolic, i.e., $g$ is not elliptic and has exactly one fixed point in $\partial Z$; or
    \item $g$ is loxodromic, i.e., $g$ is not elliptic and has exactly two fixed points in $\partial Z$.
\end{enumerate}
If $g \in \Isom(X)$ is loxodromic, we can denote two fixed points by $y_g, y_{g^{-1}} \in \partial Z$ so that $g^n x \to y_g$ as $n \to \infty$ for all $x \neq y_{g^{-1}}$ and $g^{-n} x \to y_{g^{-1}}$ as $n \to \infty$ for all $x \neq y_g$. We call $y_g$ and $y_{g^{-1}}$ the attracting and repelling fixed points of $g$ respectively.

For $g \in \Isom(X)$, we define its asymptotic translation length by $$\ell(g) := \lim_{n \to \infty} \frac{d_Z(x, g^n x)}{n}$$ for $x \in Z$. This does not depend on the choice of $x$. It is clear that $\ell(h g h^{-1}) = \ell(g)$ and $\ell(g^{n}) = |n| \ell(g)$ for all $g, h \in \Isom(X)$ and $n \in \Z$.

\begin{lemma} \label{lem.trlength}
    Let $g \in \Isom(X)$ be loxodromic and $[y_{g^{-1}}, y_g]$ a geodesic between $y_{g^{-1}}$ and $y_g$. Let $x \in [y_{g^{-1}}, y_g]$ and $\sigma_0 : [0, \infty) \to [y_{g^{-1}}, y_g]$ be a parametrization from $x$ to $y_g$. Then $$| \beta_{\sigma_0}(x, gx) - \ell(g)| \le 12 \delta.$$
    Moreover, if $\sigma : [0, \infty) \to Z$ is a geodesic ray with $\sigma(\infty) = y_g$ and $w \in Z$, then $$| \beta_{\sigma}(w, gw) - \ell(g)| \le 92 \delta.$$
\end{lemma}

\begin{proof}
    Note that $g^n x \to y_g$ as $n \to \infty$ and $g^n x$ always belongs to a geodesic $g^n[y_{g^{-1}}, y_g]$ between $y_{g^{-1}}$ and $y_g$. Hence for all large $n$, there exists $t_n \ge 0$ such that $d_Z(g^n x, \sigma_0(t_n)) \le 6 \delta$ by Lemma \ref{lem.stablebiinfinite}. We then have for each $n \ge 1$ that 
    $$\begin{aligned}
        | \beta_{\sigma_0}(x, g^n x) & - d_Z(x, \sigma_0(t_n)) | \\
        & = \lim_{t \to \infty} |d_Z(x, \sigma_0(t)) - d_Z(x, \sigma_0(t_n)) - d_Z(g^n x, \sigma_0(t))| \\
        & = \lim_{t \to \infty} |d_Z(\sigma_0(t_n), \sigma_0(t)) - d_Z(g^n x, \sigma_0(t))| \le 6 \delta.
    \end{aligned}$$
    Together with $d_Z(g^n x, \sigma_0(t_n)) \le 6 \delta$, this implies $$| \beta_{\sigma_0}(x, g^n x) - d_Z(x, g^n x) | \le 12 \delta.$$

    On the other hand, we have $$\beta_{\sigma_0}(x, g^n x) = \sum_{i = 0}^{n-1} \beta_{\sigma_0}(g^ix, g^{i+1} x) =  \sum_{i = 0}^{n-1} \beta_{g^{-i}\sigma_0}(x, g x).$$
    Since each $g^{-i}\sigma_0$ is the geodesic ray in a geodesic $g^{-i}[y_{g^{-1}}, y_g]$ between $y_{g^{-1}}$ and $y_g$, it follows from Lemma \ref{lem.stablebiinfinite} that $|\beta_{g^{-i}\sigma_0}(x, gx) - \beta_{\sigma_0}(x, gx)| \le 12 \delta$. This implies $$|\beta_{\sigma_0}(x, g^n x) - n \beta_{\sigma_0}(x, gx) | \le 12 \delta (n-1).$$
    Hence we have $$\left| \beta_{\sigma_0}(x, gx) - \frac{d_Z(x, g^n x)}{n} \right| \le 12 \delta.$$ Since this holds for all large $n \ge 1$, taking $n \to \infty$ yields the first claim.

    Let us now show the last claim. We have $$\beta_{\sigma}(w, gw) = \beta_{\sigma}(x, gx) + \beta_{\sigma}(w, x) - \beta_{g^{-1}\sigma}(w, x).$$ Since $g^{-1}\sigma(\infty) = y_g$ as well, it follows from Lemma \ref{lem.buseatpoint} that
    $$|\beta_{\sigma}(w, g w) - \beta_{\sigma}(x, gx)| \le 40 \delta.$$
    Applying Lemma \ref{lem.buseatpoint} again, we obtain $$|\beta_{\sigma}(w, gw) - \beta_{\sigma_0}(x, gx) | \le 80 \delta.$$
    By the first claim, we obtain $$|\beta_{\sigma}(w, gw) - \ell(g)| \le 92 \delta$$ as desired.
\end{proof}

If a subgroup $\Ga < \Isom(Z)$ acts properly discontinuously on $Z$, then the actions of $\Ga$ on $\partial Z$ and $\bar Z$ are convergence group actions \cite[Lemma 2.11]{Bowditch1999convergence}. We denote by $\La^Z = \La^Z_{\Ga}$ the limit set of $\Ga$, which is the set of accumulation points of the $\Ga$-orbit in $\partial Z$. We say that $\Ga$ is non-elementary if $\# \La^Z \ge 3$.

\section{Essential subgroups for graph-conformal measures} \label{sec.ess}

Let $G_1$ be a connected semisimple real algebraic group and $\Ga < G_1$ be a Zariski dense $\theta_1$-hypertransverse subgroup.
Let $(Z, d_Z)$ be a proper geodesic $\delta$-hyperbolic space on which $\Ga$ acts properly discontinuously by isometries, with the $\Ga$-equivariant homeomorphism $\iota: \La^Z \to \La^{\theta_1}$.

We keep the same notations as in Section \ref{sec.Myrberg}. Let $\rho : \Ga \to G_2$ be a Zariski dense $\theta_2$-regular faithful representation with a pair of $\rho$-equivariant continuous maps $f : \La^{\theta_1} \to \La^{\theta_2}_{\rho(\Ga)}$ and $f_{\i} : \La^{\i(\theta_1)} \to \La^{\i(\theta_2)}_{\rho(\Ga)}$. Let $G = G_1 \times G_2$ and consider the self-joining $$\Ga_{\rho} = (\id \times \rho)(\Ga) < G.$$ Its limit set in $\F_{\theta}$ is the graph $\La_{\rho}^{\theta} = (\id \times f)(\La^{\theta_1})$. Recall that for a $\Ga$-conformal measure $\nu$ on $\La^{\theta_1}$, the graph-conformal measure is defined as follows: $$\nu_{\rho} = (\id \times f)_* \nu.$$

The main goal of this section is to prove: 
\begin{theorem} \label{thm.ess}
Let $\nu$ be a $\Ga$-conformal measure of divergence type and $\nu_{\rho}$ the associated graph-conformal measure of $\Ga_{\rho}$.
If $\Gr$ is Zariski dense, then $$\ess_{\nu_{\rho}}^{\theta} = \fa_{\theta}.$$

\end{theorem}

\subsection*{Visual balls in $\La_{\rho}^{\theta}$}
Via the equivariant homeomorphisms $\iota: \La^Z \to \La^{\theta_1}$ and $(\id \times f) : \La^{\theta_1} \to \La^{\theta}_{\rho}$, we identify $\La^Z$, $\La^{\theta_1}$ and $\La^{\theta}_{\rho}$.
In particular, we can use the notion of visual balls in Section \ref{sec.hyperbolic} on all three spaces. More precisely, setting $f_0 := (\id \times f) \circ \iota$, we can define the visual metric on $\La^{\theta}_{\rho}$ as follows: for $x \in Z$ and $\xi, \eta \in \La^{\theta}_{\rho}$, $$d_x(\xi, \eta) := e^{-2 \langle f_0^{-1}(\xi), f_0^{-1}(\eta) \rangle_x}.$$
We also use the same notation for the $d_x$-balls: for $\xi \in \La^{\theta}_{\theta}$ and $r > 0$, $B_x(\xi, r) := \{ \eta \in \La_{\rho}^{\theta} : d_x(\xi, \eta) < r \}$. This allows us to regard $\La_{\rho}^{\theta}$ as the limit set of $\Ga$ in $\partial Z$, and to employ the properties of visual metrics discussed in Section \ref{sec.hyperbolic}.

\subsection*{Main proposition}

Since $(\id \times \rho) : \Ga \to \Gr$ is an isomorphism, we can regard the $\Ga$-action on $Z$ as the $\Gr$-action on $Z$: for $g \in \Ga$ and $x \in Z$, $(g, \rho(g)) \cdot x = g \cdot x$. We will keep using this identification to ease the notations.

\begin{prop} \label{prop.ess} \label{prop.jordaness}  
Let $\nu$ be a $\Ga$-conformal measure of divergence type, and $\nu_{\rho}$ the associated graph-conformal measure of $\Ga_{\rho}$.
Let $\ga_0 \in \Gr$ be a loxodromic element such that $\ell(\ga_0) > \frac{1}{2} \left( 344 \delta + 10^{100}\delta + \log 3 \right)$. For any $\varepsilon > 0$ and a Borel subset $B \subset \F_{\theta}$ with $\nu_{\rho}(B) > 0$, there exists $\ga \in \Gr$ such that $$
B \cap \ga \ga_0 \ga^{-1} B \cap \left\{ \xi \in \Lambda_{\rho}^{\theta} : \begin{matrix}
\| \beta_{\xi}^{\theta}(o, \ga \ga_0 \ga^{-1} o) - \la_{\theta}(\ga_0) \| < \varepsilon
\end{matrix} \right\}$$
has a positive $\nu_{\rho}$-measure. In particular,
$$\lambda_{\theta}(\ga_0)\in \ess_{\nu_{\rho}}^{\theta}(\Ga_{\rho}).$$

\end{prop}

\begin{remark} \label{rmk.ess}
    The proof of Proposition \ref{prop.ess} is motivated by Robin \cite{Roblin2003ergodicite} and Lee-Oh \cite{lee2020invariant}. In both (\cite{Roblin2003ergodicite}, \cite{lee2020invariant}) there exists a nice Busemann function due to the $\op{CAT}(-1)$ and the higher rank Morse lemma respectively. In contrast, in the generality of our setting, the Busemann function on a Gromov hyperbolic space $(Z, d_Z)$ is not as good as the one in $\op{CAT}(-1)$ spaces, and the higher rank Morse lemma is not available. We overcome this difficulty by simultaneously controlling both the coarsely defined Busemann function on $(Z, d_Z)$ and the $\fa_{\theta}$-valued Busemann map on the higher rank symmetric space $X$ to make the modified argument work. 
    Our arguments are  based on the dynamical properties of transverse subgroups and the uniformity and stability results on Busemann functions on $(Z, d_Z)$ obtained in Section \ref{sec.hyperbolic}.
\end{remark}

In the rest of this section, we fix a loxodromic element $\ga_0 \in \Gr$ and assume that $\ell(\ga_0) > \frac{1}{2} \left( 344 \delta + 10^{100}\delta + \log 3 \right)$. We denote by $\xi_0 \in \La_{\rho}^{\theta}$ and $\eta \in \La_{\rho}^{\i(\theta)}$ the attracting and repelling fixed points of $\ga_0$ respectively, which are identified with the attracting and repelling fixed points $y_{\ga_0}, y_{\ga_0^{-1}} \in \La^Z$ respectively. Since $\xi_0$ and $\eta$ are in general position, we can choose $p = go \in X$ where $g \in G$ is such that $\xi_0 = gP_{\theta}$ and $\eta = g w_0 P_{\i(\theta)}$.
We also fix a point $x \in Z$ on a geodesic $[\xi_0, \eta]$ between $\xi_0$ and $\eta$ in $Z$ and a geodesic ray $\sigma_0 : [0, \infty) \to [\xi_0, \eta]$ with $\sigma_0(0) = x$ and $\sigma_0(\infty) = \xi_0$. Finally, we fix $0 < \varepsilon < 1/2$.

\subsection*{Covering the Myrberg limit set}
We first make the choice of two constants: $$C_1 = 10^{100} \delta \quad \text{and} \quad C_2 = 10^{10} \delta.$$
We then have $$C_1 \ge 5C_2 + 104 \delta  \quad \text{and} \quad \ell(\ga_0) > \frac{1}{2} \left( 344 \delta + C_1 + \log 3 \right).$$ We only need these two properties; one can choose different $C_1$ and $C_2$ as long as they satisfy the above inequalities.

For each $\ga \in \Gr$, let $r_0(\ga) > 0$ be the supremum of $r \ge 0$ such that 
\begin{itemize}
    \item we have $$\sup_{\xi \in B_x(\ga \xi_0, 3 e^{8 \delta} r)} \| \beta_{\xi}^{\theta}(p, \ga \ga_0^{\pm 1} \ga^{-1} p) \mp \la_{\theta}(\ga_0) \| < \varepsilon; \text{ and}$$

    \item for any geodesic ray $\sigma : [0, \infty) \to Z$ with $\sigma(\infty) \in B_x(\ga \xi_0, 3 e^{8 \delta} r)$, $$ | \beta_{\sigma}(x, \ga \ga_0^{\pm 1} \ga^{-1} x) \mp \ell(\ga_0) | < C_1.$$
\end{itemize}
Such an $r$ exists by Lemma \ref{lem.buseisjordan}, Lemma \ref{lem.trlength}, and Lemma \ref{lem.stablebuse0415}.

For each $R > 0$, we define $$\B_R (\ga_0, \varepsilon) = \{ B_x(\ga \xi_0, r) : \ga \in \Gr, 0 < r \le \min(R, r_0(\ga))\}.$$ Choose $0 < s = s(\ga_0) < R$ small enough so that
\begin{itemize}
    \item we have $$\sup_{\xi \in B_x(\xi_0, s)} \| \beta_{\xi}^{\theta}(p, \ga_0^{\pm 1} p) \mp \la_{\theta}(\ga_0) \| < \frac{\varepsilon}{4};$$

    \item for any geodesic ray $\sigma : [0, \infty) \to Z$ with $\sigma(\infty) \in B_x(\xi_0, s)$, $$ | \beta_{\sigma}(x, \ga_0^{\pm 1} x) \mp \ell(\ga_0)| < C_2; \text{ and}$$ 

    \item we have $$B_x(\xi_0, e^{2 \ell(\ga_0) + C_2} s) \subset O_{\varepsilon / (8 \kappa)}^{\theta}(\eta, p) \cap  O_{C_2}^{Z}(\eta, x)$$
where $\kappa >  0$ is the constant given in Lemma \ref{lem.shadow}.
\end{itemize}
For each $\ga \in \Gr$ and $r > 0$, we set $$D(\ga \xi_0, r) = B_x\left(\ga \xi_0, \frac{1}{3 e^{8 \delta}} e^{-2d_{Z}(x, \ga x)} r \right).$$

\begin{proposition} \label{prop.coveringmyrberg}
    Fix $R > 0$. Let $\xi \in \La_{\rho}^{\theta}$ and $\ga_i \in \Gr$ be a sequence such that $\ga_i^{-1} p \to \eta$ and $\ga_i^{-1} \xi \to \xi_0$ as $i \to \infty$. Then for any $0 < r \le e^{-500 \delta} s(\ga_0)$, there exists $i_0$ such that for all $i \ge i_0$, we have $$D(\ga_i \xi_0, r) \in \B_R(\ga_0, \varepsilon) \quad \text{and} \quad \xi \in D(\ga_i \xi_0, r).$$

    In particular, for any $R > 0$, we have $$\La_{\rho, M}^{\theta} \subset \bigcup_{D \in \B_R(\ga_0, \varepsilon)} D.$$
\end{proposition}

\begin{proof}
    We first claim that $D(\ga_i \xi_0, r) \in \B_R(\ga_0, \varepsilon)$ for all large $i$. By Lemma \ref{lem.29inv} and the equivariance of the homeomorphism $\La^Z \to \La^{\i(\theta_1)}$, we have that $\ga_i^{-1}x \to y_{\ga_0^{-1}}$ in $\bar Z$, noting that $\eta = y_{\ga_0^{-1}}^{\i(\theta)}$. Hence $O_{C_2}^{Z}(y_{\ga_0^{-1}}, x) \subset O_{C_2 + 2 \delta}^{Z}(\ga_i^{-1}x, x) $ for all $i$ by Lemma \ref{lem.shadowapprox}.

    For each $i \ge 1$, we set $s_i = \frac{1}{3e^{8 \delta}} e^{-2 d_{Z}(\ga_i x, x)} r$. We need to show that $$\sup_{\xi' \in B_x(\ga_i \xi_0, 3 e^{8 \delta} s_i)} \| \beta_{\xi'}^{\theta}(p, \ga_i \ga_0^{\pm 1} \ga_i^{-1} p) \mp \la_{\theta}(\ga_0) \| < \varepsilon$$ and for any geodesic ray $\sigma:[0, \infty) \to Z$ with $\sigma(\infty) \in B_x(\ga_i \xi_0, 3e^{8 \delta} s_i)$, 
    $$| \beta_{\sigma}(x, \ga_i \ga_0^{\pm 1} \ga_i^{-1} x) \mp \ell(\ga_0) | < C_1.$$

    Let $\xi' \in B_x(\ga_i \xi_0, 3 e^{8 \delta} s_i)$ and $\sigma : [0, \infty) \to Z$ a geodesic ray with $\sigma(\infty) = \xi'$. We have from Lemma \ref{lem.vmbpchange} that
    $$\begin{aligned}
        d_x(\xi_0, \ga_i^{-1} \xi') & = d_{\ga_i x}(\ga_i \xi_0, \xi') \le e^{244\delta} e^{\beta_{\ga_i \sigma_0}(x, \ga_i x) + \beta_{\sigma}(x, \ga_i x)} d_x(\ga_i \xi_0, \xi') \\
        & < e^{244 \delta} e^{ 2 d_Z(x, \ga_i x)} e^{-2 d_Z(\ga_i x, x)} r = e^{244 \delta} r.
    \end{aligned}$$

    Since $e^{244\delta}r \le s(\ga_0)$, we have $$\| \beta_{\ga_i^{-1} \xi'} ^{\theta} (p, \ga_0 p) - \la_{\theta}(\ga_0) \| < \frac{\varepsilon}{4} \quad \mbox{and} \quad | \beta_{\ga_i^{-1} \sigma}(x, \ga_0 x) - \ell(\ga_0)| < C_2.$$
    Hence we have 
    $$\begin{aligned}
        d_x(\xi_0, \ga_0^{-1}\ga_i^{-1}\xi') & = d_{\ga_0 x} (\xi_0, \ga_i^{-1}\xi') \le e^{244\delta} e^{\beta_{\sigma_0}(x, \ga_0 x) + \beta_{\ga_i^{-1}\sigma}(x, \ga_0 x)} d_x(\xi_0, \ga_i^{-1}\xi') \\
        & \le e^{244 \delta} e^{12 \delta} e^{2\ell(\ga_0) + C_2} e^{244 \delta} r
    \end{aligned}$$
    by Lemma \ref{lem.trlength}. Since $e^{500\delta}r < s(\ga_0)$, this implies that $\ga_i^{-1} \xi', \ga_0^{-1}\ga_i^{-1} \xi' \in B_x(\xi_0, e^{2 \ell(\ga_0) + C_2} s)$. Since $B_x(\xi_0, e^{2\ell(\ga_0) + C_2}s) \subset O^{\theta}_{\varepsilon/(8\kappa)}(\eta, p)$ and $\ga_i^{-1}p \to \eta$,  we obtain from Corollary \ref{cor.busemannapprox} that $$\| \beta_{\ga_i^{-1} \xi'}^{\theta}(\ga_i^{-1} p, p) - \beta_{\ga_0^{-1} \ga_i^{-1} \xi'}^{\theta}(\ga_i^{-1} p, p) \| < \frac{\varepsilon}{2} \quad \text{for all but finitely many } i.$$

    Now we have $$\begin{aligned}
        & \| \beta_{\xi'}^{\theta}(p, \ga_i \ga_0 \ga_i^{-1} p) - \la_{\theta}(\ga_0) \| \\
        & = \| \beta_{\xi'}^{\theta}(p, \ga_i p) + \beta_{\xi'}^{\theta}(\ga_i p, \ga_i \ga_0  p) + \beta_{\xi'}^{\theta}(\ga_i \ga_0 p, \ga_i \ga_0 \ga_i^{-1} p) - \la_{\theta}(\ga_0) \|  \\
        & \le \| \beta_{\xi'}^{\theta}(p, \ga_i p) - \beta_{\xi'}^{\theta}(\ga_i \ga_0 \ga_i^{-1} p, \ga_i \ga_0 p) \| + \| \beta_{\xi'}^{\theta}(\ga_i p, \ga_i \ga_0 p) - \la_{\theta}(\ga_0) \| \\
        & = \| \beta_{\ga_i^{-1} \xi'}^{\theta}(\ga_i^{-1} p, p) - \beta_{\ga_0^{-1} \ga_i^{-1} \xi'}^{\theta}(\ga_i^{-1} p, p) \| + \| \beta_{\ga_i^{-1} \xi'}^{\theta}(p, \ga_0 p) - \la_{\theta}(\ga_0) \| \\
        & < \varepsilon/ 2 + \varepsilon / 4 < \varepsilon.
    \end{aligned}$$

    Similarly, as $B_x(\xi_0, e^{2 \ell(\ga_0) + C_2}s) \subset O_{C_2}^{Z}(y_{\ga_0^{-1}}, x)$, it follows form Lemma \ref{lem.shadowapprox} and Lemma \ref{lem.buseinshadow} that  $$| \beta_{\ga_i^{-1} \sigma} (\ga_i^{-1} x, x) - \beta_{\ga_0^{-1} \ga_i^{-1} \sigma} (\ga_i^{-1} x, x) | <  4 C_2 + 104 \delta \quad \text{for all large } i.$$
    Hence we have $$\begin{aligned}
        & | \beta_{\sigma}(x, \ga_i \ga_0 \ga_i^{-1} x) - \ell(\ga_0)| \\
        & \le | \beta_{\ga_i^{-1} \sigma}(\ga_i^{-1} x, x) - \beta_{\ga_0^{-1} \ga_i^{-1} \sigma}(\ga_i^{-1} x, x) | + | \beta_{\ga_i^{-1} \sigma}(x, \ga_0 x) - \ell(\ga_0) | \\
        & < 4 C_2 + 104 \delta + C_2 \le C_1.
        \end{aligned}$$

    By the same argument, we also have $$\| \beta_{\xi'}^{\theta}(p, \ga_i \ga_0^{-1} \ga_i^{-1} p) + \la_{\theta}(\ga_0) \| < \varepsilon \quad \text{and} \quad | \beta_{\sigma}(x, \ga_i \ga_0^{-1} \ga_i^{-1} x) + \ell(\ga_0) | < C_1.$$ Since $\xi' \in B_x(\ga_i \xi_0, 3 e^{8 \delta} s_i)$ and the geodesic ray $\sigma$ were arbitrary, it shows $D(\ga_i \xi_0, r) \in \B_R(\ga_0, \varepsilon)$ for all large $i$.

    \medskip
    
    We now prove the second claim that $\xi \in D(\ga_i \xi_0, r)$ for all large $i$. Since $\ga_i^{-1} \xi \to \xi_0$, we may assume that $$\ga_i^{-1} \xi \in B_x(\xi_0, e^{2\ell(\ga_0) + C_2 }s) \subset O_{C_2}^Z( y_{\ga_0^{-1}}, x) \quad \text{for all } i \ge 1.$$ By Lemma \ref{lem.shadowapprox}, we have  $$\ga_i^{-1} \xi \in  O_{C_2 + 2 \delta}^Z(\ga_i^{-1} x, x) \quad \text{for all large }i.$$ Note that $\xi_0 \in   O_{C_2 + 2 \delta}^Z(\ga_i^{-1} x, x)$ as well. Let $\sigma$ be a geodesic ray with $\sigma(\infty) = \xi$. It follows from Lemma \ref{lem.buseinshadow} that
     $$ \begin{aligned}
        | \beta_{\ga_i^{-1} \sigma}(\ga_i^{-1} x, x) - d_Z(\ga_i^{-1}x, x) | < 2C_2 + 52 \delta; \\
        | \beta_{\sigma_0}(\ga_i^{-1} x, x) - d_Z(\ga_i^{-1}x, x) | < 2C_2 + 52 \delta.
     \end{aligned}$$ Therefore, we have $$\begin{aligned}
        d_x(\ga_i \xi_0, \xi) & = d_{\ga_i^{-1} x}(\xi_0, \ga_i^{-1} \xi) \\
        & \le e^{244 \delta} e^{-(\beta_{\sigma_0}(\ga_i^{-1}x, x) + \beta_{\ga_i^{-1} \sigma}(\ga_i^{-1} x, x))} d_x(\xi_0, \ga_i^{-1} \xi) \\
        & < e^{244 \delta} e^{-2 d_{Z}(\ga_i^{-1} x, x)} e^{4C_2 + 104 \delta} d_x(\xi_0, \ga_i^{-1} \xi).
    \end{aligned}$$
    Since $\ga_i^{-1} \xi \to \xi_0$, we have $d_x(\xi_0, \ga_i^{-1}\xi) < e^{-(4C_2 + 348 \delta)} \frac{1}{3 e^{8 \delta}}r$ for all large $i$, and hence $\xi \in D(\ga_i \xi_0, r)$, completing the proof.
\end{proof}

\subsection*{Approximation by $d_x$-balls}
From now on, let $\nu$ be a $(\Ga, \psi)$-conformal measure of divergence type and $\nu_{\rho}$ the associated graph-conformal measure of $\Ga_{\rho}$. Note that $\nu_{\rho}$ is a $(\Ga_{\rho}, \sigma_{\psi})$-conformal measure where $\sigma_{\psi} \in \fa_{\theta}^*$ is the composition $\psi \circ p_{\theta_1}$ (Proposition \ref{lem.pushforward}).
It is more convenient to use
the following conformal measure $\nu_p$ (with respect to the basepoint $p$) as $\ess_{\nu_p}^{\theta}(\Gr) = \ess_{\nu_{\rho}}^{\theta}(\Gr)$: 
$$
d \nu_p(\xi) = e^{\sigma_{\psi}(\beta_{\xi}^{\theta}(o, p))} d \nu_{\rho}(\xi).
$$

\begin{proposition} \label{prop.approxbycovers}
    Let $B \subset \F_{\theta}$ be a Borel subset with $\nu_{p}(B) > 0$. Then for $\nu_{p}$-a.e. $\xi \in B$, we have $$\lim_{R \to 0} \sup_{\xi \in D, D \in \B_R(\ga_0, \varepsilon)} \frac{\nu_{p}(B \cap D)}{\nu_{p}(D)} = 1.$$
\end{proposition}

\begin{proof}
    For a Borel function $h : \F_{\theta} \to \R$, define $h^* : \F_{\theta} \to \R$ as $$h^*(\xi) := \lim_{R \to 0} \sup_{\xi \in D, D \in \B_R(\ga_0, \varepsilon)}  \frac{1}{\nu_{p}(D)} \int_D h d\nu_{p}.$$
    By Proposition \ref{prop.coveringmyrberg}, $h^*$ is well-defined on $\La_{\rho, M}^{\theta}$. Since $\La_{\rho, M}^{\theta}$ has full $\nu_p$-measure by Theorem \ref{thm.fullmyrberg}, $h^*$ is well-defined for $\nu_{p}$-a.e. $\xi \in \F_{\theta}$. It suffices to show that $h(\xi)= h^*(\xi)$ for $\nu_{p}$-a.e. $\xi \in \F_{\theta}$; taking $h = \mathbbm{1}_{B}$ implies the desired identity. Note that $h = h^*$ if $h$ is continuous; we now consider the general case.

    \medskip
    {\bf \noindent Claim.} We claim that for any $\alpha > 0$, we have $$\nu_{p}(h^* > \alpha) \le \frac{e^{\sigma_\psi(\la_{\theta}(\ga_0)) + \|\sigma_\psi\| \varepsilon}}{\alpha} \int_{\F_{\theta}} | h| d \nu_{p}.$$ To see this, it suffices to show that for any compact $Q \subset \{ h^*  > \alpha \}$, we have $$\nu_{p}(Q) \le \frac{e^{\sigma_\psi(\la_{\theta}(\ga_0)) + \|\sigma_\psi\| \varepsilon}}{\alpha} \int_{\F_{\theta}} | h| d \nu_{p}.$$ Fix $R > 0$ and a compact subset $Q \subset \{ h^*  > \alpha \}$. By the definition of $h^*$, for each $q \in Q$, there exists $D_q \in \B_R(\ga_0, \varepsilon)$ containing $q$ such that $$\frac{1}{\nu_{p}(D_q)} \int_{D_q} h d \nu_{p} > \alpha.$$ Since $Q$ is compact, we have a finite subcover $\{D_i = B_x(\ga_i \xi_0, s_i) \}$ of $\{D_q : q \in Q\}$, where $\ga_i \in \Gr$ and $s_i = \frac{1}{3 e^{8 \delta}} e^{-2 d_{Z}(\ga_i^{-1} x, x)} r_i$ for some $ r_i > 0$.

    By Lemma \ref{lem.Vcoveringtrans}, there exists a subcollection $D_{i_1}, \cdots, D_{i_k}$ of disjoint subsets such that $$\bigcup_i D_i \subset \bigcup_{j = 1}^k 3e^{8 \delta} D_{i_j}$$ where $3 e^{8 \delta} D_{i_j} = B_x(\ga_{i_j} \xi_0, 3e^{8 \delta} s_{i_j})$.

    We observe that for each $j$, $3 e^{8 \delta} D_{i_j} \subset \ga_{i_j} \ga_0^{-1} \ga_{i_j}^{-1} D_{i_j}$. Indeed, for $\xi \in 3 e^{8 \delta} D_{i_j}$ and a geodesic ray $\sigma$ with $\sigma(\infty) = \xi$, we have from Lemma \ref{lem.vmbpchange}, $\xi \in 3 e^{8 \delta} D_{i_j}$, and $D_{i_j} \in \B_R(\ga_0, \varepsilon)$ that $$\begin{aligned}
        d_x(\ga_{i_j} \xi_0, & \ga_{i_j}\ga_0 \ga_{i_j}^{-1} \xi)
         \\ & = d_{\ga_{i_j} \ga_0^{-1} \ga_{i_j}^{-1} x}(\ga_{i_j} \xi_0, \xi) \\
        & \le e^{244 \delta} e^{-( \beta_{\ga_{i_j} \sigma_0}(\ga_{i_j} \ga_0^{-1} \ga_{i_j}^{-1} x, x) + \beta_{\sigma}(\ga_{i_j} \ga_0^{-1} \ga_{i_j}^{-1}x, x))} d_x(\ga_{i_j} \xi_0, \xi) \\
        & \le e^{244 \delta} e^{92 \delta} e^{-2 \ell(\ga_0) + C_1} 3 e^{8 \delta} s_{i_j} < s_{i_j}.
    \end{aligned}$$ This shows $\ga_{i_j} \ga_0 \ga_{i_j}^{-1} \xi \in D_{i_j}$, and hence $\xi \in \ga_{i_j} \ga_0^{-1} \ga_{i_j}^{-1} D_{i_j}$.

    Therefore, we have $$\begin{aligned}
    \nu_{p}(3 e^{8 \delta} D_{i_j}) & \le \nu_{p}(\ga_{i_j} \ga_0^{-1} \ga_{i_j}^{-1} D_{i_j}) = \int_{D_{i_j}} e^{\sigma_\psi (\beta_{\xi}^{\theta}(p, \ga_{i_j}\ga_0 \ga_{i_j}^{-1}p))} d \nu_{p}(\xi) \\
    & \le e^{\sigma_\psi(\la_{\theta}(\ga_0)) + \| \sigma_\psi \| \varepsilon} \nu_{p}(D_{i_j}).
    \end{aligned}
    $$
    Now it follows that $$\begin{aligned}
        \nu_{p}(Q) & \le \sum_{j = 1}^k \nu_{p}(3 e^{8 \delta} D_{i_j})\\
        & \le \sum_{j = 1}^k \frac{e^{\sigma_\psi(\la_{\theta}(\ga_0)) + \| \sigma_\psi \| \varepsilon}}{\alpha} \int_{D_{i_j}} h d\nu_{p} \le \frac{e^{\sigma_\psi(\la_{\theta}(\ga_0)) + \| \sigma_\psi \| \varepsilon}}{\alpha} \int_{\F_{\theta}} |h| d\nu_{p},
    \end{aligned}$$ as desired.

    \medskip
    We now finish the proof of the proposition by showing that $h(\xi) = h^*(\xi)$ for $\nu_{p}$-a.e. $\xi$. We first show that $h(\xi) \le h^*(\xi)$ for $\nu_{p}$-a.e. $\xi$. Let $\alpha > 0$ and a sequence of continuous functions $h_n \to h$ in $L^1(\nu_{p})$. Since $h_n$ is continuous, $h_n^* = h_n$. Now we have $$\begin{aligned}
        \nu_{p}(h - h^* > \alpha) & \le \nu_{p}(h - h_n > \alpha/2) + \nu_{p}(h_n^* - h^* > \alpha/2) \\
        & \le \frac{2}{\alpha} \| h - h_n \|_{L^1} + \frac{2}{\alpha} e^{\sigma_\psi(\la_{\theta}(\ga_0)) + \| \sigma_\psi \| \varepsilon} \| h - h_n \|_{L^1}.
    \end{aligned}$$ Since $\|h - h_n \|_{L^1} \to 0$ as $n \to \infty$, we have $\nu_{p}(h - h^* > \alpha) = 0$. Since $\alpha > 0$ is arbitrary, $h(\xi) \le h^*(\xi)$ for $\nu_{p}$-a.e. $\xi$. The similar argument shows $h^*(\xi) \le h(\xi)$ for $\nu_{p}$-a.e. $\xi$, and it completes the proof.
\end{proof}

\subsection*{Proof of Proposition \ref{prop.ess}}

Let $B \subset \F_{\theta}$ be a Borel subset with $\nu_{p}(B) > 0$. It suffices to show that for some $\ga \in \Gr$, the set $$B \cap \ga \ga_0 \ga^{-1} B \cap \{ \xi : \| \beta_{\xi}^{\theta}(p, \ga \ga_0 \ga^{-1} p) - \la_{\theta}(\ga_0)\| < \varepsilon\}$$ has positive $\nu_{p}$-measure.

By Proposition \ref{prop.approxbycovers}, there exists $D = B_x(\ga \xi_0, r) \in \B_R(\ga_0, \varepsilon)$ such that 
\be \label{eqn.approxchoice}
\nu_{p}(D \cap B) > (1 + e^{-\sigma_\psi(\la_{\theta}(\ga_0)) - \| \sigma_\psi \| \varepsilon})^{-1} \nu_{p}(D).
\ee
Since $r < r_0(\ga)$, we have $$D \subset \{ \xi : \| \beta_{\xi}^{\theta}(p, \ga \ga_0^{\pm} \ga^{-1} p) \mp \la_{\theta}(\ga_0) \|< \varepsilon\} $$
and for any geodesic ray $\sigma$ with $\sigma(\infty) \in D$, we have $$| \beta_{\sigma}(x, \ga \ga_0^{\pm} \ga^{-1} x) \mp \ell(\ga_0) | < C_1 .$$
This implies $$B \cap \ga \ga_0 \ga^{-1} B \cap \{ \xi : \| \beta_{\xi}^{\theta}(p, \ga \ga_0 \ga^{-1} p) - \la_{\theta}(\ga_0)\| < \varepsilon\} \supset (D \cap B) \cap \ga \ga_0 \ga^{-1}(D \cap B).$$
Hence it suffices to show 
\be \label{eqn.finishessproof}
\nu_{p}((D \cap B) \cap \ga \ga_0 \ga^{-1}(D \cap B)) > 0.
\ee

By the conformality, we have $$\begin{aligned}
    \nu_{p}(\ga \ga_0 \ga^{-1}(D \cap B)) & = \int_{D \cap B} e^{\sigma_\psi(\beta_{\xi}^{\theta}(p, \ga \ga_0^{-1} \ga^{-1} p))} d \nu_p(\xi) \\
    & > e^{- \sigma_\psi(\la_{\theta}(\ga_0)) - \| \sigma_\psi \| \varepsilon} \nu_p(D \cap B).\end{aligned}$$
    Hence we have $$\nu_p(D \cap B) + \nu_p(\ga \ga_0 \ga^{-1} (D \cap B)) > (1 + e^{- \sigma_\psi(\la_{\theta}(\ga_0)) - \| \sigma_\psi \| \varepsilon}) \nu_p(D \cap B).$$ Together with the choice \eqref{eqn.approxchoice} of $D$, we obtain 
    \be \label{eqn.approxresult}
    \nu_p(D \cap B) + \nu_p(\ga \ga_0 \ga^{-1} (D \cap B)) > \nu_p(D).
    \ee

We claim that $\ga \ga_0 \ga^{-1}D \subset D$. Indeed, if $\xi \in D$ and $\sigma$ is a geodesic ray with $\sigma(\infty) = \xi$, then by Lemma \ref{lem.vmbpchange} and Lemma \ref{lem.trlength}, $$\begin{aligned}
    d_x(\ga \xi_0, \ga \ga_0 \ga^{-1} \xi) & = d_{\ga \ga_0^{-1} \ga^{-1} x}(\ga \xi_0, \xi) \\
    & \le e^{244 \delta} e^{\beta_{\ga \sigma_0}(x, \ga \ga_0^{-1} \ga^{-1} x) + \beta_{\sigma}(x, \ga \ga_0^{-1} \ga^{-1} x)} d_x(\ga \xi_0, \xi) \\
    & < e^{244\delta} e^{92 \delta} e^{-2 \ell(\ga_0) + C_1} r < r.
\end{aligned}$$
Hence the claim follows.

Now both $D \cap B$ and $\ga \ga_0 \ga^{-1} (D \cap B)$ are subsets of $D$. Therefore, \eqref{eqn.approxresult} must imply \eqref{eqn.finishessproof}, completing the proof of Proposition \ref{prop.ess}.
\qed
\medskip

\begin{cor} \label{cor.loxess}
    For any loxodromic $\ga_0 \in \Gr$, we have $$\la_{\theta}(\ga_0) \in \ess_{\nu_{\rho}}^{\theta}(\Gr).$$
\end{cor}

\begin{proof}
    Let $\ga_0 \in \Ga$ be a loxodromic element. For sufficiently large $n$, both $\ga_0^n$ and $\ga_0^{n+1}$ satisfy the condition of Proposition \ref{prop.ess}. Hence we have $$n\la_{\theta}(\ga_0), (n+1) \la_{\theta}(\ga_0) \in \ess_{\nu_{\rho}}^{\theta}(\Gr).$$ Since $\ess_{\nu_{\rho}}^{\theta}(\Gr)$ is a subgroup of $\fa_{\theta}$, we have $$\la_{\theta}(\ga_0) = (n + 1) \la_{\theta}(\ga_0) - n \la_{\theta}(\ga_0) \in \ess_{\nu_{\rho}}^{\theta}(\Gr).$$
\end{proof}

\subsection*{Proof of Theorem \ref{thm.ess}}
By Corollary \ref{cor.loxess}, $\la_{\theta}(\ga_0) \in \ess_{\nu_{\rho}}^{\theta}(\Gr)$ for all loxodromic $\ga_0 \in \Gr$.
Since $\ess_{\nu_{\rho}}^{\theta}(\Gr)$ is a closed subgroup of $\fa_{\theta}$, it follows from Theorem \ref{t1} that $\ess_{\nu_{\rho}}^{\theta}(\Ga_{\rho}) = \fa_{\theta}$ if $\Gr$ is Zariski dense.
\qed

\subsection*{Essential subgroups for hypertransverse subgroups}
The same argument applies to a Zariski dense $\theta$-hypertransverse subgroup $\Ga < G$, which is not necessarily a self-joining. Therefore we deduce:

\begin{theorem}  \label{thm.fullessdiv}
    Let $G$ be a semisimple real algebraic group and $\Ga < G$ a Zariski dense $\theta$-hypertransverse subgroup. For a $\Ga$-conformal measure $\nu$ of divergence type, we have $$\ess_{\nu}^{\theta}(\Ga) = \fa_{\theta}.$$
\end{theorem}

\section{Singularity of the graph-conformal measure} \label{sec.proof}

We are finally ready to prove our main rigidity theorems. We recall the setting: let $G_1$ and $G_2$ be connected simple real algebraic groups and $\Ga < G_1$ be a Zariski dense $\theta_1$-hypertransverse subgroup with the limit set $\La^{\theta_1} \subset \F_{\theta_1}$.
Let $\rho : \Ga \to G_2$ be a Zariski dense $\theta_2$-regular representation with  $\rho$-equivariant continuous maps $f : \La^{\theta_1} \to \F_{\theta_2}$ and $f_{\i} : \La^{\i(\theta_1)} \to \F_{\i(\theta_2)}$. Let $\nu$ be a $(\Ga, \psi)$-conformal measure of divergence type, for $\psi \in \fa_{\theta_1}^*$.

Recall that $\Gr = (\id \times \rho)(\Ga)$ is the self-joining of $\Ga$ via $\rho$ which is a discrete subgroup of $G = G_1 \times G_2$.
The graph-conformal measure $\nu_{\rho} = (\id \times f)_*\nu$ is the unique $(\Gr, \sigma_{\psi})$-conformal measure on $\La_{\rho}^{\theta}$ where $\sigma_{\psi}$ is the composition of $\psi$ with the projection $\fa_{\theta} \to \fa_{\theta_1}$ (Proposition \ref{lem.pushforward}).

\begin{theorem} \label{thm.singulargraphconformal}
    If $\Gr$ is Zariski dense, then $$\nu_{\phi} \not\ll \nu_{\rho}$$ for all $(\Gr, \phi)$-conformal measure $\nu_{\phi}$ on $\F_{\theta}$ with $\phi \neq \sigma_{\psi}$.
\end{theorem}

\begin{proof}
Let $\nu_{\phi}$ be a $(\Gr, \phi)$-conformal measure on $\F_{\theta}$ for some $\phi \in \fa_{\theta}^*$. Suppose that $\nu_{\phi} \ll \nu_{\rho}$.
By Theorem \ref{thm.ess}, we have $\ess_{\nu_{\rho}}^{\theta}(\Gr) = \fa_{\theta}$. Hence it follows from Proposition \ref{prop.loabscont} that $\phi = \sigma_{\psi}$ on $\ess_{\nu_{\rho}}^{\theta}(\Gr) = \fa_{\theta}$. Therefore, $\phi = \sigma_{\psi}$ and the theorem follows.
\end{proof}

\subsection*{Proof of Theorem \ref{thm.main}}
Suppose that $\Gr$ is Zariski dense and there exists a $(\rho(\Ga), \varphi)$-conformal measure $\nu_{\varphi}$ on $\La_{\rho(\Ga)}^{\theta_2}$ for some $\varphi \in \fa_{\theta_2}^*$ such that $$\nu_{\varphi} \ll f_*\nu.$$ Then by Proposition \ref{lem.pushforward}, we have $$(f^{-1} \times \id)_* \nu_{\varphi} \ll \nu_{\rho}$$ and $(f^{-1} \times \id)_* \nu_{\varphi}$ is a $(\Gr, \sigma_{\varphi})$-conformal measure where $\sigma_{\varphi}$ is the composition of the projection $\fa_{\theta} \to \fa_{\theta_2}$ and $\varphi \in \fa_{\theta_2}^*$. By Theorem \ref{thm.singulargraphconformal}, we must have $\sigma_{\psi} = \sigma_{\varphi}$.

On the other hand, $\fa_{\theta_1} < \fa_{\theta} = \fa_{\theta_1} \oplus \fa_{\theta_2}$ is contained in $\ker \sigma_{\varphi}$ while $\sigma_{\psi}(u) = \psi(u) \neq 0$ for some $u \in \fa_{\theta_1}$, which is a contradiction. Therefore, $\Gr$ is not Zariski dense and hence $\rho$ extends to a Lie group isomorphism $G_1 \to G_2$ by Lemma \ref{Zdense}.
\qed

\subsection*{Proof of Theorem \ref{thm.mainpsclass}}
By Theorem \ref{thm.fullessdiv}, we have $\ess_{\nu}^{\theta}(\Ga) = \fa_{\theta}$. Hence Theorem \ref{thm.mainpsclass} follows by the same argument as in the proof of Theorem \ref{thm.singulargraphconformal}.
\qed

\section{Deformations of transverse representations} \label{sec.transrigidity}

In this  section, we consider deformations of transverse representations to which Theorem \ref{thm.singulargraphconformal} can be applied. We keep the same notations from previous sections.
Let $(Z, d_Z)$ be a proper geodesic $\delta$-hyperbolic space and $\Delta < \Isom(Z)$ a non-elementary subgroup acting properly discontinuously on $Z$. For $i = 1, 2$, we consider $\theta_i$-transverse representations $\rho_i : \Delta \to G_i$ and write $\Ga_i := \rho_i(\Delta)$. The conjugate $\rho = \rho_2 \circ \rho_1|_{\Delta}^{-1}$ between two representations is referred to as a deformation from $\rho_1$ to $\rho_2$:
$$\begin{tikzcd}[column sep = large, row sep = tiny]
    & \Ga_1  \arrow[dd, "\rho"]  \\
    \Delta \arrow[ru, "\rho_1"] \arrow[rd, "\rho_2"'] &  \\
    & \Ga_2
\end{tikzcd}$$

In this setting, we obtain the following stronger form of the conformal measure rigidity theorem which was stated as Theorem \ref{thm.main2} in the introduction:

\begin{theorem}
    There exists a pair of $\rho$-boundary maps $f : \La^{\theta_1}_{\Ga_1} \to \F_{\theta_2}$ and $f_{\i} : \La^{\i(\theta_1)}_{\Ga_1} \to \F_{\i(\theta_2)}$. 
    Moreover, unless 
    $\rho  : \Ga_1 \to \Ga_2$ does not extend to a Lie group isomorphism $G_1 \to G_2$, $$\nu_2 \not \ll f_* \nu_1$$
    for any $\Ga_1$-conformal measure $\nu_1$ of divergence type and $\Ga_2$-conformal measure $\nu_2$. In particular, if $\nu_2$ is further assumed to be of divergence type, then $$\nu_2 \perp f_* \nu_1.$$
\end{theorem}

\begin{proof}
By definition of the transverse representation, for $i = 1, 2$, we have a $\rho_i$-equivariant homeomorphism $f_i : \La_{\Delta}^{Z} \to \La^{\theta_i \cup \i(\theta_i)}_{\Ga_i}$. Together with the canonical projections $\La^{\theta_i \cup \i(\theta_i)}_{\Ga_i} \to \La^{\theta_i}_{\Ga_i}$ and $\La^{\theta_i \cup \i(\theta_i)}_{\Ga_i} \to \La^{\i(\theta_i)}_{\Ga_i}$, we have the following commutative diagram:
$$\begin{tikzcd}[row sep = large]
\La^{\i(\theta_1)}_{\Ga_1} \arrow[dd, dashed, "f_{\i}"'] & \La^{\theta_1 \cup \i(\theta_1)}_{\Ga_1} \arrow[l, "\sim"'] \arrow[r, "\sim"] & \La^{\theta_1}_{\Ga_1} \arrow[dd, dashed, "f"] \\
& \La^{Z}_{\Delta} \arrow[u, "f_1"], \arrow[d, "f_2"'] & \\
\La^{\i(\theta_1)}_{\Ga_2} & \La^{\theta_1 \cup \i(\theta_1)}_{\Ga_2} \arrow[l, "\sim"'] \arrow[r, "\sim"] & \La^{\theta_1}_{\Ga_2}
\end{tikzcd}$$
As indicated in the above diagram, the projections $\La^{\theta_i \cup \i(\theta_i)}_{\Ga_i} \to \La^{\theta_i}_{\Ga_i}$ and $\La^{\theta_i \cup \i(\theta_i)}_{\Ga_i} \to \La^{\i(\theta_i)}_{\Ga_i}$, $i = 1, 2$, are homeomorphisms due to the $\theta_i$-antipodality of $\Ga_i$ \cite[Lemma 9.5]{kim2023growth}. Hence the maps $f$ and $f_{\i}$ are well-defined as above, and are homeomorphisms. Moreover, since all maps in the diagram are equivariant under the actions of the corresponding groups, $f$ and $f_{\i}$ are $\rho$-equivariant. Therefore, they form a pair of $\rho$-boundary maps.

This allows us to apply Theorem \ref{thm.main}, finishing the proof.
\end{proof}

\section{Horospherical foliations and Burger-Roblin measures} \label{sec.BRmeas}
In the rest of the paper, let $G$ be a connected semisimple real algebraic group and fix a non-empty $\theta \subset \Pi$. In this section, we discuss  ergodic properties  of Burger-Roblin measures on horospherical foliations.

Recall the space $$\Hor_{\theta} := \F_{\theta} \times \fa_{\theta}$$ and the actions of $G$ and $A_{\theta}$ on $\Hor_{\theta}$ given as follows: for $(\xi, u) \in \Hor_{\theta}$, $g \in G$ and $a \in A_{\theta}$, \be \label{eqn.actions}
\begin{aligned}
    g\cdot (\xi, u) &= (g \xi, u + \beta_{\xi}^{\theta}(g^{-1}, e));\\
    (\xi, u) \cdot a &= (\xi, u + \log a).
\end{aligned}\ee
Denoting by $g^+ = gP_{\theta} \in F_{\theta}$, the map $g \mapsto (g^+, \beta_{g^+}^{\theta}(e, g))$ induces a homeomorphism $$G/N_{\theta}S_{\theta} \simeq \Hor_{\theta}.$$
 Hence the space $\Hor_{\theta}$ can be considered as the $\theta$-horospherical foliation. Indeed, when $G$ is of rank one, $\Hor_{\theta}$ is the horospherical foliation of the unit tangent bundle of $G/K$.
 
Since $A_{\theta}$ normalizes $N_{\theta}S_{\theta}$, the quotient $G / N_\theta S_{\theta}$ admits both left $G$-action and right $A_{\theta}$-action, and the above homeomorphism is $(G, A_{\theta})$-equivariant.
A Radon measure $m$ on $\Hor_{\theta}$ is $A_{\theta}$-semi-invariant if there exists a linear form $\chi_m \in \fa_{\theta}^*$ such that for all $a \in A_{\theta}$, we have $$a_* m = e^{\chi_m(\log a)}m.$$
We define a $\Ga$-invariant $A_{\theta}$-semi-invariant Radon measure on $\Hor_{\theta}$, called Burger-Roblin measure.
\begin{definition}[Burger-Roblin measures]
    Let $\Ga < G$ be a discrete subgroup and $\nu$ a $(\Ga, \psi)$-conformal measure on $\F_{\theta}$ for some $\psi \in \fa_{\theta}^*$. The \emph{Burger-Roblin measure} $m_{\nu}^{\BR}$ on $\Hor_{\theta}$ associated to $\nu$ is defined by $$dm_{\nu}^{\BR}(\xi, u) := e^{\psi(u)}d\nu(\xi)du$$ where $du$ is the Lebesgue measure on $\fa_{\theta}$.
\end{definition}

In fact, all $\Ga$-invariant $A_{\theta}$-semi-invariant measures arise as Burger-Roblin measures. See (\cite{Babillot_horosphere}, \cite{Burger_horoc}, \cite{LL_horospherical}) for rank one settings, and \cite[Proposition 10.25]{lee2020invariant} for higher rank:

\begin{proposition}
    Let $\Ga < G$ be a Zariski dense discrete subgroup. Any $\Ga$-invariant $A_{\theta}$-semi-invariant Radon measure on $\Hor_{\theta}$ is proportional a Burger-Roblin measure associated with some $\Ga$-conformal measure on $\F_{\theta}$.
\end{proposition}

\subsection*{Ergodicity of horospherical foliations}
We now prove the ergodicity of horospherical foliations with respect to Burger-Roblin measures.
The size of the essential subgroup plays a role of criterion for the ergodicity of actions of horospherical foliations. The following was proved in \cite{Schmidt1977cocycles} for abstract measurable dynamical systems, and more direct proof for particular case of $\op{CAT}(-1)$ spaces was given in \cite[Proposition 2.1]{Roblin2003ergodicite}. Following \cite{Roblin2003ergodicite}, the higher rank version was obtained in \cite[Proposition 9.2]{lee2020invariant} when $\theta = \Pi$. The same proofs as in (\cite{lee2020invariant} and \cite{Roblin2003ergodicite}) works for general $\theta$:

\begin{proposition} \label{prop.essanderg}
    Let $\Ga < G$ be a Zariski dense discrete subgroup and $\nu$ a $\Ga$-conformal measure on $\F_{\theta}$. The $\Ga$-action on $(\Hor_{\theta}, m_{\nu}^{\BR})$ is ergodic if and only if the $\Ga$-action on $(\F_{\theta}, \nu)$ is ergodic and $\ess_{\nu}^{\theta}(\Ga) = \fa_{\theta}$.
\end{proposition}

\subsection*{Proof of Theorem \ref{thm.ergodichoro}}

Let $\Ga$ be a Zariski dense $\theta$-hypertransverse subgroup. Let $\nu$ be a $\Ga$-conformal measure of divergence type. By Theorem \ref{thm.fullessdiv}, we have $\ess_{\nu}^{\theta}(\Ga) = \fa_{\theta}$.
Moreover, $(\F_{\theta}, \Ga, \nu)$ is ergodic by Theorem \ref{thm.uniquepsdiv}.  Therefore, the ergodicity of the $\Ga$-action on $(\Hor_{\theta}, m_{\nu}^{\BR})$ follows from Proposition \ref{prop.essanderg}.
\qed

\subsection*{Ergodic decomposition} In the rest of the section, we consider the case $\theta = \Pi$; we omit the subscripts and superscripts for $\theta = \Pi$. Let $\Ga < G$ be a Zariski dense $\Pi$-hypertransverse subgroup. 

For a $(\Ga, \psi)$-conformal measure $\nu$ on $\F$ for some $\psi \in \fa^*$, the associated Burger-Roblin measure $\hat m_{\nu}^{\BR}$ on $\Ga \ba G$ is defined in \eqref{eqn.defBRGG}. Let $\nu_{\i}$ be a $(\Ga, \psi \circ \i)$-conformal measure on $\F$. We now define the Bowen-Margulis-Sullivan measure for the pair $(\nu, \nu_{\i})$. The generalized Hopf-parametrization for $G$ is an isomorphism $G/M \to \F^{(2)} \times \fa$ defined by $$gM \mapsto (g^+, g^-, \beta_{g^+}(e, g))$$
where $g^+ = gP$, $g^- = g w_0 P \in \F$. By fixing a Borel section $G/M \to G$, it induces an isomorphism 
\be \label{eqn.genhopf}
G \to \F^{(2)} \times \fa \times M.
\ee
Via \eqref{eqn.genhopf}, the following defines a left $\Ga$-invariant and right $AM$-invariant measure on $G$: for $g \in  G$, 
\be \label{eqn.defbmsgabag}
d\hat m^{\BMS}_{\nu, \nu_{\i}}(g) := e^{\psi\left( \beta_{g^+}(e, g) + \i(\beta_{g^-}(e, g)) \right)} d \nu(g^+) d \nu_{\i}(g^-) da dm
\ee
where $da$ and $dm$ denote the Haar measures on $\fa$ and $M$ respectively. Hence it induces an $AM$-invariant measure  on $\Ga \ba G$ which we also denote by $\hat m_{\nu, \nu_{\i}}^{\BMS}$ and call the Bowen-Margulis-Sullivan measure associated to the pair $(\nu, \nu_{\i})$. Note that when $\nu$ is of divergence type, $\nu_{\i}$ uniquely exists by Theorem \ref{thm.uniquepsdiv}, and therefore we simply write $\hat m_{\nu}^{\BMS} := \hat m_{\nu, \nu_{\i}}^{\BMS}$.

Recall from the introduction that $\mathfrak{D}_{\Ga}$ is the collection of all $P^\circ$-minimal subsets of $\Ga \ba G$ where $P^{\circ}$ is the identity component of $P$. For a fixed $\cal E_0 \in \frak D_{\Ga}$, we set $P_{\Ga} := \{p \in P : \cal E_0 p = \cal E_0\}$. Then $P_{\Ga}$ is a finite index co-abelian subgroup of $P$ and is independent of the choice of $\cal E_0$, and moreover the map  $P_{\Ga} \ba P \to \frak D_{\Ga}$, $[p] \mapsto \cal E_0 p$, is bijective \cite{Guivarch2007actions}. 
We now present the ergodic decompositions of the Burger-Roblin and Bowen-Margulis-Sullivan measures on $\Ga \ba G$, which is stated as Theorem \ref{thm.ergdecompintro} in the introduction:

\begin{theorem} \label{thm.ergdecompbody}
    Let $\Ga < G$ be a Zariski dense $\Pi$-hypertransverse subgroup. Let $\nu$ be a $\Ga$-conformal measure on $\F$ of divergence type. Then
    \begin{enumerate}
        \item $\hat m_{\nu}^{\BR} = \sum_{\cal E \in \frak D_{\Ga}} \hat m_{\nu}^{\BR}|_{\cal E}$ is an $N$-ergodic decomposition;
        \item $\hat m_{\nu}^{\BMS} = \sum_{\cal E \in \frak D_{\Ga}} \hat  m_{\nu}^{\BMS}|_{\cal E}$ is an $A$-ergodic decomposition.
    \end{enumerate}
    In particular, the number of $N$-ergodic components of $\hat m_{\nu}^{\BR}$ and the number of $A$-ergodic components of $\hat m_{\nu}^{\BMS}$ are given by $\# \frak D_{\Ga} = [P : P_{\Ga}]$.
\end{theorem}

In \cite{lee2020ergodic}, Lee-Oh deduced the ergodic decomposition theorem for $\Pi$-Anosov subgroups from the ergodicity of $NM$-action and $AM$-action on $\Ga \ba G$ respectively, which were shown in their another work \cite{lee2020invariant}. The Anosov property was used in order to have 
\begin{itemize}
    \item $\Pi$-regularity and $\Pi$-antipodality of $\Ga$;
    \item the ergodicity of $NM$-action and the complete conservativity and ergodicity of $AM$-action on $\Ga \ba G$;
    \item appropriate covering of the limit set to show that the $\fa \times M$-valued essential subgroup for a $\Ga$-conformal measure $\nu$ is the whole $\fa \times M$.
\end{itemize}

On the other hand, when $\Ga$ is $\Pi$-hypertransverse, it is $\Pi$-regular and $\Pi$-antipodal. Moreover, if $\nu$ is of divergence type, then we showed that the $NM$-action on $(\Ga \ba G, \hat m_{\nu}^{\BR})$ is ergodic in  Theorem \ref{thm.ergodichoroPi} and the complete conservativity and ergodicity of $AM$-action on $(\Ga \ba G, \hat m_{\nu}^{\BMS})$ were known (\cite{canary2023patterson}, \cite{kim2023growth}, Theorem \ref{ceq}). Finally, as we have shown in Section \ref{sec.ess}, the new covering of the limit set defined in terms of the visual metric on the Gromov boundary plays an appropriate role to prove that the essential subgroup is full. The deduction for the extended version of the essential subgroup, taking values in $\fa \times M$, can be done in a same way as in \cite{lee2020ergodic}. Therefore, Theorem \ref{thm.ergdecompbody} can be deduced by the same argument as in \cite{lee2020ergodic}, with these replacements of the above three items.

\subsection*{Dense $A^+$-orbits}
We now deduce the following which is stated as Theorem \ref{thm.denseAplusintro} in the introduction:

\begin{theorem}
    Let $\Ga < G$ be a Zariski dense $\Pi$-hypertransverse subgroup. Let $\nu$ be a $\Ga$-conformal measure on $\F$ of divergence type. Then for any $\cal E \in \frak D_{\Ga}$ and $\hat m_{\nu}^{\BMS}$-a.e. $x \in \cal E$,
    $$
    \ov{x A^+} = \supp \hat m_{\nu}^{\BMS}|_{\cal E}.
    $$
\end{theorem}

\begin{proof}
Let $\La^{\Pi} \subset \F$ be the limit set of $\Ga$ and set $\La^{(2)} := (\La^{\Pi} \times \La^{\Pi}) \cap \F^{(2)}$.
Via the isomorphism in \eqref{eqn.genhopf}, consider a subset
$$
\mathcal{S} := \La^{(2)} \times \fa \times M \subset G.
$$
Then $\Ga \ba \cal S = \supp \hat m_{\nu}^{\BMS}$, and the right $A$-action on $G$ corresponds to the translation action on the $\fa$-component. Let $\psi \in \fa^*$ be a $(\Ga, \Pi)$-proper linear form associated to $\nu$ and set $\cal S_{\psi} := \La^{(2)} \times \R \times M$ and the projection $\cal S \to \cal S_{\psi}$ given by $(\xi, \eta, u, m) \in \cal S \mapsto (\xi, \eta, \psi(u), m) \in \cal S_{\psi}$. By Theorem \ref{tproper2}, the induced $\Ga$-action on $\cal S_{\psi}$ is properly discontinuous, and hence we have the projection
$$
\Psi : \Ga \ba \cal S \to \Ga \ba \cal S_{\psi}.
$$
The translation on the $\fa$-component descends to the translation on the $\R$-component, under $\Psi$.

As in \eqref{eqn.defbmsgabag}, consider the following $\Ga$-invariant measure on $\cal S_{\psi}$ given by
$$
e^{\psi\left( \beta_{\xi}(e, g) + \i(\beta_{\eta}(e, g)) \right)} d \nu(\xi) d \nu_{\i}(\eta) dt dm
$$
for $g \in G$ such that $(g^+, g^-) = (\xi, \eta)$. This induces a measure $\hat m_{\nu}$ on $\Ga \ba \cal S_{\psi}$ which is invariant under the $\R$-translation. Then $\hat m_{\nu}^{\BMS}$ is the disintegration of $\hat m_{\nu}$ along the fiber $\ker \psi$.

Let $\cal E_\psi := \Psi(\Ga \ba \cal S \cap \cal E)$. By Theorem \ref{thm.ergdecompbody}, the $\R$-translation on $(\cal E_{\psi}, \hat m_{\nu}|_{\cal E_{\psi}})$ is ergodic. Moreover, since $M$ is compact, it follows from Theorem \ref{ceq} that the $\R$-translation on $(\cal E_\psi, \hat m_{\nu}|_{\cal E_\psi})$ is completely conservative. Therefore, $\hat m_{\nu}$-a.e. $\R_{+}$-orbit in $\cal E_\psi$ is dense. Denoting by $A_{\psi} := \exp \{ \psi > 0 \} \subset A$, this implies that for $\hat m_{\nu}^{\BMS}$-a.e. $x \in \cal E$, $\ov{xA_{\psi}} = \Ga \ba \cal S \cap \cal E$.

Fix $x \in \Ga \ba \cal S \cap \cal E$ with a dense $A_{\psi}$-orbit. We now show that $xA^+$ is dense as well. Since $\Ga \ba \cal S \cap \cal E - xA \subset \Ga \ba \cal S \cap \cal E$ is dense, it suffices to show that $\ov{xA^+} \supset \Ga \ba \cal S \cap \cal E - xA$. Let $y \in\Ga \ba \cal S \cap \cal E - xA$. We choose $g, h \in G$ such that $[g] = x$ and $[h] = y$. Then there exist sequences $\ga_n \in \Ga$ and $a_n \in A_{\psi}$ such that $\psi(\log a_n) \to \infty$ and $\ga_n g a_n \to h$. In particular, comparing the $\fa$-component of $\cal S$, we have that
$$
\beta_{g^+}(\ga_n^{-1}, e) + \log a_n \quad \text{is bounded.}
$$
Since $\psi(\log a_n) \to \infty$, this also implies $\psi(\beta_{g^+}(\ga_n^{-1}, e)) \to - \infty$. By \cite[Proof of Proposition 9.10]{kim2023growth}, we have for some $R > 0$ that
$$
g^+ \in O_R^{\Pi}(o, \ga_n^{-1} o) \quad \text{for all } n \ge 1.
$$
It then follows from Lemma \ref{lem.shadow} that
$$
- \mu(\ga_n^{-1}) + \log a_n \quad \text{is bounded.}
$$
Hence, for any fixed closed convex cone $\C \subset \fa$ such that $\fa^+ \subset \inte \C \cup \{0\}$, the $\exp \C$-orbit of $x$ is dense in $\Ga \ba \cal S \cap \cal E$. Since $\fa^+$ and $-\fa^+$ are the only Weyl chambers that can give a dense orbit in $\Ga \ba \cal S \cap \cal E$ by \cite[Lemma 8.13]{lee2020invariant}, we can choose $\C \subset \fa$ such that $(\C - \{0\}) \cap -\fa^+ = \emptyset$ and have that $\ov{xA^+} = \Ga \ba \cal S \cap \cal E$. This finishes the proof, since $\Ga \ba \cal S = \supp \hat m_{\nu}^{\BMS}$.
\end{proof}

\bibliographystyle{plain} 
\bibliography{Rigidity_DKim}

\end{document}